\documentclass[preprint,3p,times]{elsarticle}
\RequirePackage{multirow,booktabs,subfigure,color,array,hhline,makecell}
\usepackage{amssymb}
\usepackage{amsmath}
\usepackage{graphicx}
\usepackage{subfigure}
\usepackage{amsthm}
\usepackage{bbm}
\usepackage{mathrsfs}
\usepackage{indentfirst}
\usepackage[colorlinks,citecolor=blue,urlcolor=blue]{hyperref}
\allowdisplaybreaks
\usepackage[table]{xcolor}
\usepackage{tikz-network}
\usepackage{pgf}
\usepackage{tikz}
\usepackage{subfigure}
\usepackage{matlab-prettifier}
\usepackage{tabularx}
\usepackage{ragged2e}
\newcolumntype{Y}{>{\RaggedRight\arraybackslash}X}

\usetikzlibrary{arrows, decorations.pathmorphing, backgrounds, positioning, fit, petri, automata}
\usetikzlibrary{shadows,arrows,positioning}
\RequirePackage{algorithm,algpseudocode}
\allowdisplaybreaks
\usepackage{changes}
\usepackage{caption}
\usepackage{bbm}

\newtheorem{thm}{Theorem}
\newtheorem{defin}{Definition}
\newtheorem{lem}{Lemma}
\newtheorem{assum}{Assumption}
\newtheorem{rem}{Remark}

\newtheorem{Ex}{Example}
\allowdisplaybreaks[4]
\AtBeginDocument{%
    \providecommand\BibTeX{%
        \normalfont B\kern-0.5em{\scshape i\kern-0.25em b}\kern-0.8em\TeX}}
\journal{~}
\begin{document}
\captionsetup[figure]{labelfont={bf},labelformat={default},labelsep=period,name={Fig.}}
\begin{frontmatter}
\title{How many asymmetric communities are there in multi-layer directed networks?}
\author[label]{Huan Qing\corref{cor}}
\ead{qinghuan@cqut.edu.cn$\&$qinghuan@u.nus.edu}
\address[label]{School of Economics and Finance, Chongqing University of Technology, Chongqing, 400054, China}
\cortext[cor]{Corresponding author.}
\begin{abstract}
Estimating the asymmetric numbers of communities in multi-layer directed networks is a challenging problem due to the multi-layer structures and inherent directional asymmetry, leading to possibly different numbers of sender and receiver communities. This work addresses this issue under the multi-layer stochastic co-block model, a model for multi-layer directed networks with distinct community structures in sending and receiving sides, by proposing a novel goodness-of-fit test. The test statistic relies on the deviation of the largest singular value of an aggregated normalized residual matrix from the constant 2. The test statistic exhibits a sharp dichotomy: Under the null hypothesis of correct model specification, its upper bound converges to zero with high probability; under underfitting, the test statistic itself diverges to infinity. With this property, we develop a sequential testing procedure that searches through candidate pairs of sender and receiver community numbers in a lexicographic order. The process stops at the smallest such pair where the test statistic drops below a decaying threshold. For robustness, we also propose a ratio-based variant algorithm, which detects sharp changes in the sequence of test statistics by comparing consecutive candidates. Both methods are proven to consistently determine the true numbers of sender and receiver communities under the multi-layer stochastic co-block model. 
\end{abstract}
\end{frontmatter}
\section{Introduction}
Multi-layer directed networks have emerged as a fundamental representation for complex relational systems characterized by multiple, asymmetric interaction patterns. Such networks consist of a common set of nodes and multiple layers, each captured by a directed adjacency matrix. This structure naturally encodes two critical aspects of real-world relational data: the multiplicity of interaction contexts (layers) \citep{kivela2014multilayer,kim2015community,de2016physics,huang2021survey} and the inherent directionality of relationships within each context \citep{malliaros2013clustering,rohe2016co,su2024spectral,qing2025discovering}. By preserving both the variety of interaction types and their directional nature, multi-layer directed networks offer a richer and more faithful representation of complex systems than single-layer or undirected abstractions. Representative examples span diverse fields: international trade networks, where layers correspond to different commodities and directed edges represent export flows \citep{de2015structural}; brain connectivity studies, where layers reflect distinct cognitive tasks and edges model directed neural pathways \citep{bassett2011dynamic,bakken2016comprehensive}; and social communication systems, where separate layers may capture email correspondence, co-authorship, or online social interactions \citep{kim2015community,PengshengAOAS896,ji2022co,huang2021survey}.

A central problem in analyzing such networks is community detection—the identification of groups of nodes that exhibit similar connectivity patterns \citep{leicht2008community,fortunato2010community,malliaros2013clustering,fortunato2016community,javed2018community,paul2021null}. In directed networks, this task naturally decomposes into two complementary problems: identifying sender communities (nodes with similar outgoing connections) and receiver communities (nodes with similar incoming connections) \citep{rohe2016co, zhou2019analysis,wang2020spectral, zhang2022directed}. This distinction reflects the fundamental asymmetry of directed relationships, wherein a node's role in initiating interactions may be substantially different from its role in receiving them. The multi-layer stochastic co-block model (ML‑ScBM) considered in \citep{su2024spectral} provides a principled probabilistic framework for this setting. It extends the stochastic co-block model (ScBM) introduced in \citep{rohe2016co} for single-layer directed networks—which itself generalizes the classical stochastic block model (SBM) \citep{holland1983stochastic} by allowing separate community assignments for senders and receivers—to multiple layers. Notably, the ML-ScBM can also be viewed as a generalization of the multi-layer stochastic block model (ML-SBM) which is extensively studied in the network literature \citep{han2015consistent,paul2016consistent,paul2020spectral,lei2020consistent,xu2023covariate,lei2023bias,lei2024computational,qing2025communityEAAI} in recent years from the undirected to the directed setting. The ML‑ScBM assumes that each node maintains consistent sender and receiver community memberships across all layers, while the connection probabilities between communities can vary from layer to layer. This formulation achieves an elegant balance between flexibility (accommodating layer‑specific connectivity patterns) and parsimony (preserving a common community structure across layers), thereby naturally capturing asymmetric relational structures.

Under the ML‑ScBM, the development of methods for estimating community memberships from observed data remains an active area. A notable and effective approach is the spectral co‑clustering technique, which offers computational efficiency and theoretical guarantees. For instance, \citet{su2024spectral} developed a debiased spectral co‑clustering algorithm that aggregates information across layers via a bias‑adjusted sum of Gram matrices, followed by \(k\)-means applied to the leading eigenvectors. The work in \citep{su2024spectral} extends the bias‑adjusted spectral clustering framework introduced in \citet{lei2023bias} from multi‑layer undirected to multi‑layer directed networks. This method provides estimation consistency for community recovery, as is commonly established for spectral clustering approaches in community detection (see, e.g., \citep{qin2013regularized,lei2015consistency,SCORE,joseph2016impact,zhou2019analysis,paul2020spectral,wang2020spectral,qing2025communityASoC} to name a few). A common requirement shared by this and any prospective method for detecting asymmetric communities in multi-layer directed networks is the prior knowledge of the number of sender communities \(K_s\) and receiver communities \(K_r\). In practice, these quantities are rarely known in advance, making their estimation a fundamental prerequisite for any community detection procedure. This disconnect between methodological requirements and practical application underscores a pressing need for reliable, data‑driven techniques to determine these key parameters.

The problem of estimating the number of communities has been extensively investigated for simpler network models. For single‑layer undirected SBMs, a variety of approaches exist, including likelihood‑ratio tests \citep{LikeAos2017,ma2021determining}, Bayesian information criteria \citep{mcdaid2013improved,saldana2017many,hu2020corrected}, spectrum of the Bethe Hessian matrices \citep{CanMK2022,hwang2024estimation}, network cross-validation \citep{chen2018network,li2020network},  and goodness‑of‑fit tests \citep{fishkind2013consistent,bickel2016hypothesis,lei2016goodness,dong2020spectral,hu2021using,jin2023optimal,wu2024two,wu2024spectral}. However, a direct extension of these techniques to the multi‑layer directed setting faces several inherent challenges. First, the sender and receiver community numbers may differ, requiring their joint estimation rather than independent treatment. Second, the presence of multiple layers induces complex dependencies and necessitates the integration of information across layers in a manner that separates common community signals from layer‑specific noise. Third, the asymmetric nature of the network demands tools from non‑symmetric matrix analysis, moving beyond the symmetric spectral theory used for methods in undirected networks. Consequently, existing single‑layer methods are not directly applicable to the multi‑layer directed case, necessitating the development of novel methodologies for this complex setting.

This paper addresses this gap by introducing the first comprehensive framework for the joint, consistent estimation of asymmetric community numbers \(K_s\) and \(K_r\) in multi‑layer directed networks. Our approach is built upon a novel goodness‑of‑fit test that exhibits a sharp theoretical dichotomy: under a correctly specified model, the upper bound of the test statistic converges to zero in probability, whereas under any underfitted model, the statistic itself diverges to infinity. This dichotomy provides a rigorous statistical foundation for model selection. Leveraging this property, we design two sequential algorithms that efficiently search the two‑dimensional space of candidate pairs \((k_s, k_r)\). Both algorithms are computationally efficient and provably consistent under the ML‑ScBM with mild regularity conditions.

The main contributions of this work are as follows:

\begin{itemize}
    \item \textbf{A novel goodness‑of‑fit test for the ML‑ScBM.} We construct a test statistic based on the largest singular value of a normalized residual matrix and characterize its asymptotic behavior under both null (correctly specified) and alternative (underfitted) hypotheses. We establish that the upper bound of the statistic converges to zero under the null, while the statistic itself diverges to infinity under any form of underfitting, yielding a statistically principled criterion for model selection.

    \item \textbf{Two sequential selection procedures for joint community number estimation.} We develop two computationally efficient algorithms to search the two‑dimensional space of candidate community numbers. The first algorithm employs a level‑crossing rule with a data‑driven threshold, examining candidate pairs in a lexicographic order and stopping when the test statistic first falls below the threshold. The second algorithm employs a ratio‑based strategy: it computes the sequence of ratios of successive test statistics and selects the model at which this sequence exhibits a sharp rise, thereby identifying the transition from underfitted to adequate models.

    \item \textbf{Rigorous theoretical guarantees.} We prove the asymptotic properties of the test statistics and the consistency of both selection algorithms. The analysis hinges on two main theoretical developments. First, we derive sharp concentration bounds for the largest singular value of the aggregated, normalized residual matrix under the null hypothesis, leveraging non-symmetric random matrix theory tailored to the multi-layer directed setting. Second, under underfitting, we identify and quantify a deterministic low-rank signal—arising from the merging of distinct communities—that dominates the stochastic noise asymptotically, which is the key driver behind the test statistic's divergence. The consistency proofs for the sequential algorithms then build on this dichotomy by carefully choosing the threshold sequences to separate these two regimes. Furthermore, our theory covers settings where community numbers grow slowly with network size, and provides explicit conditions on threshold sequences that ensure correct stopping.
\end{itemize}

The remainder of this paper is structured as follows. Section~\ref{sec:model} formally introduces the multi‑layer stochastic co‑block model and states the community‑number estimation problem. Section~\ref{sec:test} develops the goodness‑of‑fit test and establishes its asymptotic properties. Section~\ref{sec:estimation} presents the sequential selection approaches and proves their estimation consistency. Section~\ref{sec:simulation} reports experimental studies that validate the theoretical findings and investigates the performance of the proposed methods. Section~\ref{sec:conclusion} concludes. All technical proofs are provided in the Appendix.
\section{Model and problem stepup}\label{sec:model}
This section formally defines the multi-layer stochastic co-block model (ML-ScBM) and precisely states the core estimation problem. We specify the ML-ScBM's generative mechanism, which incorporates distinct sender and receiver community memberships across multiple layers to model asymmetric edge directionality. The central problem addressed is the joint estimation of the unknown sender and receiver community counts \((K_s, K_r)\) from observed multi-layer directed network data. Necessary theoretical assumptions for establishing theoretical guarantees are also presented in this section.
\subsection{Multi-layer stochastic co-block model (ML-ScBM)}
We study a multi-layer directed network with \(L\) layers on \(n\) vertices. Each layer \(\ell\) is represented by an adjacency matrix \(A^{(\ell)} \in \{0,1\}^{n \times n}\), where \(A^{(\ell)}_{ii}=0\) and \(A^{(\ell)}_{ij}=1\) signifies a directed edge sending from node \(i\) to node \(j\) in layer \(\ell\).  
The \emph{multi-layer stochastic co-block model} (ML-ScBM) partitions the \(n\) vertices into \(K_s\) \emph{sender communities} and \(K_r\) \emph{receiver communities} across all layers. Formally, the model is specified as follows:
\begin{defin}[Multi-layer stochastic co-block model (ML-ScBM)]\label{def:ML-ScBM}
Consider a multi-layer directed network with \(L\) layers and \(n\) nodes per layer. Let \(A^{(\ell)} \in \{0,1\}^{n \times n}\) be the adjacency matrix for layer \(\ell\), where \(A^{(\ell)}_{ii} = 0\) and \(A^{(\ell)}_{ij} = 1\) indicates a directed edge from node \(i\) to node \(j\) in layer \(\ell\). The multi-layer stochastic co-block model (ML-ScBM) is parameterized by:
\begin{itemize}
    \item Sender community labels: \(g^s \in \{1,\dots,K_s\}^n\) (common across layers).
    \item Receiver community labels: \(g^r \in \{1,\dots,K_r\}^n\) (common across layers).
    \item Layer-specific block probability matrices: \(B^{(\ell)} \in [0,1]^{K_s \times K_r}\) for \(\ell = 1,\dots,L\).
\end{itemize}

Given \((g^s, g^r, \{B^{(\ell)}\}_{\ell=1}^L)\), the entries of \(A^{(\ell)}\) are independent with
\[
\mathbb{P}(A^{(\ell)}(i,j)= 1) = B^{(\ell)}(g^s(i), g^r(j)), \quad \forall i \neq j, \quad \ell = 1,\dots,L.
\]

The expected adjacency matrix \(\Omega^{(\ell)}\) for layer \(\ell\) satisfies \(\Omega^{(\ell)}(i,j) = B^{(\ell)}(g^s(i), g^r(j))\) for \(i \neq j\) and \(\Omega^{(\ell)}(i,i)= 0\).
\end{defin}

The sender-receiver asymmetry inherent in the ML‑ScBM—captured by the distinct membership vectors \(g^s\) and \(g^r\) and the possibly different numbers \(K_s\) and \(K_r\)—provides a flexible framework for modeling asymmetric relational patterns in multi‑layer directed networks. This generality allows the model to naturally incorporate direction‑specific community structures that are often observed in real‑world systems. The ML‑ScBM contains several well‑studied network models as special cases, thereby situating our work within a broader literature. When the sender and receiver memberships coincide (\(g^s = g^r\)) and every block probability matrix \(B^{(\ell)}\) is symmetric, the model reduces to the multi‑layer stochastic block model (ML‑SBM) studied in \citep{han2015consistent,paul2016consistent,paul2020spectral,lei2020consistent,xu2023covariate,lei2023bias,lei2024computational,qing2025communityEAAI}. In the single‑layer setting (\(L=1\)), the ML‑ScBM specializes to the stochastic co‑block model (ScBM) introduced by \citep{rohe2016co}. Finally, when \(L=1\) and the network is undirected, the model degenerates to the classical stochastic block model (SBM) of \citep{holland1983stochastic}. Thus, the ML‑ScBM offers a unified probabilistic framework that simultaneously generalizes the ML‑SBM, the ScBM, and the SBM, while preserving the essential directional asymmetry required for analyzing multi‑layer directed interactions.
\subsection{Problem statement: joint estimation of asymmetric community numbers}
In this paper, we address the fundamental problem of \emph{asymmetric community numbers estimation} in ML-ScBMs: given a multi-layer directed network, how can we jointly determine the numbers of sender communities \(K_s\) and receiver communities \(K_r\)? We formulate this as a sequential goodness-of-fit testing problem: for candidate pairs \((K_{s0}, K_{r0})\), we test
\[
H_0: (K_s, K_r) = (K_{s0}, K_{r0}) \quad \text{versus} \quad H_1: K_s > K_{s0} \text{ or } K_r > K_{r0},
\]
where the alternative hypothesis \(H_1\) explicitly encodes \emph{underfitting} scenarios where the hypothesized model lacks sufficient sender or receiver communities. This formulation is statistically challenging and practically significant for three reasons:
\begin{enumerate}
    \item The bivariate nature of the community structure requires a combinatorial search over candidate pairs \((k_s, k_r)\). When an upper bound \(K_{\mathrm{cand}}\) is set for the candidate community numbers as will be specified in Section \ref{sec3point1}, the size of this search space is \(K^{2}_{\mathrm{cand}}\), making efficient testing strategies essential.
    \item Underfitting (failing to reject \(H_0\) when \(H_1\) holds) leads to loss of structural resolution, while overfitting is mitigated algorithmically via ordered search.
    \item Existing goodness-of-fit tests for undirected networks rely on symmetric eigenvalue distributions that do not extend to the asymmetric singular value decompositions required for directed networks, and multi-layer extensions are non-trivial.
\end{enumerate}

Given that few existing methods provide theoretically guaranteed joint estimation of \((K_s, K_r)\) in ML-ScBM for multi-layer directed networks, this work bridges this gap by developing a testing framework based on singular value tail bounds of normalized residual matrices. Our approach leverages the asymptotic behavior of the largest singular value under \(H_0\) versus \(H_1\), enabling consistent community number estimation.
\subsection{Technical assumptions}\label{subsec:assumptions}
To establish theoretical guarantees, we require the following regularity conditions.

\begin{assum}\label{assump:a1}
For each layer \(\ell\), the block probability matrix satisfies \(\delta \leq B^{(\ell)}(k,l) \leq 1 - \delta\) for all \(k,l\), and some \(\delta \in(0,\frac{1}{2})\). 
\end{assum}
Assumption \ref{assump:a1} ensures well-defined variance in the residual matrix defined later and excludes degenerate cases where normalization becomes unstable. 

\begin{assum}\label{assump:a2}
The community sizes satisfy:
\[
\min_{k=1,\dots,K_s} |\{i: g^s(i) = k\}| \geq c_0 \frac{n}{K_s}, \quad 
\min_{l=1,\dots,K_r} |\{j: g^r(j) = l\}| \geq c_0 \frac{n}{K_r}
\]
for some \(c_0 > 0\). 
\end{assum}
Assumption \ref{assump:a2} prevents any community from being asymptotically negligible, guaranteeing that the sample size within every block grows linearly with \(n\). Note that Assumptions \ref{assump:a1} and \ref{assump:a2} are standard in the community detection literature \citep{lei2016goodness,hu2020corrected,hu2021using,wu2024two,wu2024spectral}, representing conventional regularity conditions for estimating the number of communities.
\begin{assum}\label{assump:a3}
$K_{\mathrm{max}}$ and \(L\) satisfy
\[
\frac{K^{2}_{\max} L\log n}{n} \to 0 \quad \text{as} \quad n \to \infty,
\]
where \(K_{\max} = \max(K_s,K_r)\).
\end{assum}

Assumption \ref{assump:a3} governs the asymptotic growth rates of the number of communities and layers relative to the network size. This condition, $\frac{K_{\max}^2 L \log n}{n} \to 0$, is fundamental for ensuring that the model complexity remains manageable as the sample size increases, which is essential for establishing the consistency of the estimation procedure. Its specific form arises directly from the convergence analysis of the normalized residual matrix $\hat{R}$: the spectral-norm error $\|\hat{R} - R\|$ is of order $O_P\bigl(K_{\max}\sqrt{L\log n / n}\bigr)$ provided by Lemma \ref{lem:norm-error} in the appendix, and squaring this bound yields the precise condition needed for the test statistic $\hat{T}_n$ to exhibit its required sharp dichotomy. Here, the term $K_{\max}^2$ captures the complexity from asymmetric sender and receiver communities, the factor $L$ accounts for the linear growth in variance when aggregating residuals across layers, and the $\log n$ factor originates from the concentration inequalities controlling maximum deviations. In practice, when the number of layers $L$ is fixed—a common scenario in multi-layer directed network studies—this condition simplifies to the standard single-layer scaling $K_{\max}^2 \log n / n \to 0$; when $L$ grows, it ensures the additional variance from multiple layers does not obscure the community signal. Thus, Assumption \ref{assump:a3} serves as a necessary scaling law that extends classical community detection theory to multi-layer directed networks, providing the foundation for the theoretical guarantees of the proposed estimation procedures.

\begin{assum}\label{assump:a4}
The community detection algorithm \(\mathcal{M}\) used is consistent under \(H_0\), i.e.,
\[
\mathbb{P}(\hat{g}^s = g^s) \to 1, \quad \mathbb{P}(\hat{g}^r = g^r) \to 1 \quad \text{as} \quad n \to \infty,
\]
where \(\hat{g}^s\) and \(\hat{g}^r\) are the estimated sending and receiving community label vectors returned by \(\mathcal{M}\)  with $K_s$ sending community numbers and $K_r$ receiving community numbers.
\end{assum}

Assumption \ref{assump:a4} is standard in theoretical analyses of community detection methods, analogous to the consistency requirement in \citep{lei2016goodness,wu2024spectral}, and is essential for establishing the asymptotic properties of the proposed goodness-of-fit test. In our context, this assumption ensures that the estimated community labels \(\hat{g}^s, \hat{g}^r\) converge to the true labels \(g^s, g^r\) with high probability as \(n\) grows under \(H_0\). The consistency guarantees that the plug-in estimates of the block probability matrices and the resulting residual matrix \(\hat{R}\) are sufficiently close to their oracle counterparts, thereby enabling the derivation of the sharp dichotomy of the test statistic—convergence to zero under the null and divergence under underfitting—without being obscured by label estimation errors.
\section{A spectral-based goodness-of-fit test}\label{sec:test}
This section develops a theoretically grounded goodness-of-fit test for ML-ScBM. We first introduce an ideal test statistic using oracle parameters, then derive its practical counterpart with estimated parameters, and finally establish its asymptotic behavior under both null and alternative hypotheses. The core innovation lies in leveraging singular value tail bounds of normalized residual matrices to detect community underfitting.
\subsection{Oracle test statistic and its asymptotics}
To formalize the test, we begin with the \emph{ideal residual matrix} \(R\) constructed using true parameters:
\[
R(i,j) = 
\begin{cases} 
\dfrac{\sum_{\ell=1}^L \left(A^{(\ell)}(i,j) - \Omega^{(\ell)}(i,j)\right)}{\sqrt{(n-1) \sum_{\ell=1}^L \Omega^{(\ell)}(i,j) (1 - \Omega^{(\ell)}(i,j))}} & i \neq j, \\
0 & i = j,
\end{cases}
\]
where \(\Omega^{(\ell)}(i,j) = B^{(\ell)}(g^s(i), g^r(j))\) is the true edge probability in layer \(\ell\). This normalization ensures \(\mathbb{E}[R(i,j)] = 0\) and \(\operatorname{Var}(R(i,j)) = \frac{1}{n-1}\) for \(i \neq j\), transforming \(R\) into a generalized random non-symmetric matrix with controlled variance. The ideal test statistic is defined as
\begin{align}\label{idealTestStatistic}
T_n = \sigma_1(R) - 2,
\end{align}
where \(\sigma_1(\cdot)\) denotes the largest singular value. The shift by 2 accounts for the asymptotic behavior of \(\sigma_1(R)\) under \(H_0\), as established below.
\begin{lem}[Asymptotic behavior of $T_n$]\label{ideal0}
When Assumptions \ref{assump:a1} and \ref{assump:a3} hold, for any $\epsilon > 0$, we have
\[
\mathbb{P}(T_n < \epsilon) \to 1 \quad \text{as} \quad n \to \infty.
\]
\end{lem}
Lemma \ref{idealTestStatistic} establishes the asymptotic spectral baseline for a correctly specified model. This result provides the essential theoretical anchor: it precisely quantifies the expected magnitude of the largest singular value in the ideal residual matrix under the null hypothesis, thereby forming the deterministic reference point required to later establish statistical power against underfitting.
\subsection{Practical test statistic and its theoretical guarantees}
In practice, model parameters are unknown. We approximate \(T_n\) by estimating the block probability matrices and community labels from the multi-layer adjacency matrices.

Let \(\mathcal{M}\) be a community detection algorithm for multi-layer directed networks. Apply \(\mathcal{M}\) to \(\{A^{(\ell)}\}_{\ell=1}^L\) to partition the \(n\) nodes into \(K_{s0}\) \textit{sender communities} and \(K_{r0}\) \textit{receiver communities}, yielding estimated labels \(\hat{g}^s\) for sender community and \(\hat{g}^r\) for receiver community. Then, for each layer \(\ell\), compute the plug-in estimator \(\hat{B}^{(\ell)} \in [0,1]^{K_{s0} \times K_{r0}}\) via  
\[  
\hat{B}^{(\ell)}(k,l) = \frac{\sum_{i:\hat{g}^s(i) = k} \sum_{j:\hat{g}^r(j) = l} A^{(\ell)}(i,j)}{|\{i: \hat{g}^s(i) = k\}| \cdot |\{j: \hat{g}^r(j) = l\}|}, \quad k=1,\dots,K_{s0}, \, l=1,\dots,K_{r0}.  
\]  

The estimated expected adjacency matrix \(\hat{\Omega}^{(\ell)}\) for layer \(\ell\) has entries \(\hat{\Omega}^{(\ell)}(i,j) = \hat{B}^{(\ell)}(\hat{g}^s(i), \hat{g}^r(j))\) for \(i \neq j\) and \(\hat{\Omega}^{(\ell)}(i,i) = 0\). 
Then we construct the normalized residual matrix \(\hat{R} \in \mathbb{R}^{n \times n}\) as 
    \begin{equation}\label{eq:residual_matrix}
    \hat{R}(i,j) = 
    \begin{cases} 
        \dfrac{\sum_{\ell=1}^L \left(A^{(\ell)}(i,j) - \hat{\Omega}^{(\ell)}(i,j)\right)}{\sqrt{(n-1) \sum_{\ell=1}^L \hat{\Omega}^{(\ell)}(i,j) (1 - \hat{\Omega}^{(\ell)}(i,j))}} & i \neq j, \\
        0 & i = j.
    \end{cases}
    \end{equation}

The practical test statistic is defined as
\begin{align}\label{eq:test_stat}  
\hat{T}_n = \sigma_1(\hat{R}) - 2.  
\end{align}   

The normalization in \(\hat{R}\) ensures \(\mathbb{E}[\hat{R}(i,j)] \approx 0\) and \(\operatorname{Var}[\hat{R}(i,j)] \approx (n-1)^{-1}\) under \(H_0\). The shift by 2  accounts for the asymptotic behavior of \(\sigma_1(\hat{R})\)'s upper bound under a correctly specified model, where \(\sigma_1(\hat{R})\)'s upper bound concentrates near 2 as shown by Theorem \ref{thm:null} given later. Under \(H_1\), \(\sigma_1(\hat{R})\) diverges due to unmodeled community structure guaranteed by Theorem \ref{thm:power} given later.  

While \(\mathcal{M}\) can be any method with consistent community recovery for multi-layer directed networks under \(H_0\), this paper employs the Debiased Sum of Gram matrices (DSoG) algorithm developed in \citep{su2024spectral}, where this algorithm extends the bias-adjusted spectral clustering idea firstly developed in \citep{lei2023bias} from multi-layer undirected networks to multi-layer directed networks.  Algorithm \ref{alg:spectral} details this procedure. 
\begin{algorithm}[H]
\caption{Debiased Sum of Gram matrices (DSoG)}
\label{alg:spectral}
\begin{algorithmic}[1]
\Require Multi-layer adjacency matrices \(\{A^{(\ell)}\}_{\ell=1}^L\), sender community number \(K_{s0}\), receiver community number \(K_{r0}\)
\Ensure Estimated labels \(\hat{g}^s, \hat{g}^r\)
\For{\(\ell = 1\) to \(L\)}
    \State Compute out-degree diagonal matrix \(D_{\ell}^{\text{out}}\) with \(D_{\ell}^{\text{out}}(i,i) = \sum_{j=1}^{n} A^{(\ell)}(i,j)\)
    \State Compute in-degree diagonal matrix \(D_{\ell}^{\text{in}}\) with \(D_{\ell}^{\text{in}}(i,i) = \sum_{j=1}^{n} A^{(\ell)}(j,i)\)
\EndFor
\For{side \( \in \{\)sender, receiver\(\}\)}
    \If{side is sender}
        \State \(S \gets \sum_{\ell=1}^{L} \left( A^{(\ell)} (A^{(\ell)})^{\top} - D_{\ell}^{\text{out}} \right)\)
        \State \(K \gets K_{s0}\)
    \Else
        \State \(S \gets \sum_{\ell=1}^{L} \left( (A^{(\ell)})^{\top} A^{(\ell)} - D_{\ell}^{\text{in}} \right)\)
        \State \(K \gets K_{r0}\)
    \EndIf
    \State Compute top \(\min(K_{s0}, K_{r0})\) eigenvectors of \(S\) as \(U \in \mathbb{R}^{n \times\min(K_{s0}, K_{r0})}\)
    \State Apply \(k\)-means clustering on rows of \(U\) with \(K\) clusters
    \State Assign resulting labels to \(\hat{g}^s\) (if sender) or \(\hat{g}^r\) (if receiver)
\EndFor
\State \Return \(\hat{g}^s, \hat{g}^r\)
\end{algorithmic}
\end{algorithm}

Given that a consistent estimator such as Algorithm \ref{alg:spectral} provides accurate community labels under \(H_0\), the plug-in residual matrix \(\hat{R}\) is asymptotically well-approximated by its oracle version \(R\). This approximation ensures that the spectral behavior of \(\hat{R}\) mirrors that of \(R\), whose largest singular value concentrates near 2 under the correct model. We therefore obtain the following fundamental convergence result for the goodness-of-fit statistic $\hat{T}_n$.
\begin{thm}[Asymptotic behavior of $\hat{T}_n$ under $H_0$]\label{thm:null}
Under \(H_0\) and Assumptions \ref{assump:a1}-\ref{assump:a4}, and let \(\hat{R}\) be obtained using consistent community estimators \(\hat{g}^{s}\) and \(\hat{g}^{r}\), then for any \(\epsilon>0\), we have
\[
\mathbb{P}(\hat{T}_n < \epsilon) \to 1 \quad \text{as} \quad n \to \infty.
\]
\end{thm}
Theorem \ref{thm:null} guarantees that the upper bound of \(\hat{T}_{n}\) converges to 0 in probability under \(H_0\). This provides the theoretical basis for decision rules: small values of \(\hat{T}_n\) support \(H_0\), while large values signal model inadequacy. The following theorem guarantees that the test statistic diverges for underfitted models where \(K_s > K_{s0}\) or \(K_r > K_{r0}\). 

\begin{thm}[Asymptotic behavior of $\hat{T}_n$ under $H_1$]\label{thm:power}
Under Assumptions \ref{assump:a1}--\ref{assump:a3}, and with the following additional conditions:
\begin{enumerate}
    \item[(A1)] There exists a constant \(\eta > 0\) such that for any two distinct true sender communities \(k \neq k'\), there exists a receiver community \(l\) satisfying
    \[
    \left| \frac{1}{L} \sum_{\ell=1}^L \left( B^{(\ell)}(k, l) - B^{(\ell)}(k', l) \right) \right| \ge \eta.
    \]
     The symmetric version should also hold.
    \item[(A2)] The network size, number of layers, and maximum community number satisfy
    \[
    \frac{nL}{K_{\max}^3} \to \infty \quad \text{as } n \to \infty.
    \]
\end{enumerate}

If either \(K_s > K_{s0}\) or \(K_r > K_{r0}\) (or both), then we have
\[
\hat{T}_n \xrightarrow{P} \infty.
\]
\end{thm}

This divergence property ensures that whenever the hypothesized model lacks sufficient sender or receiver communities, \(\hat{T}_n\) will exceed any fixed threshold with probability approaching 1. Combined with Theorem \ref{thm:null}, this guarantees asymptotically separation between correctly specified and underfitted models.

\begin{rem}
(Proof intuition for Theorem \ref{thm:power}). The divergence of $\hat{T}_n$ under underfitting is driven by a \textit{structural bias} that persists no matter how the parameters are estimated within the underspecified model. This bias arises when distinct true communities---say, in sender roles---are forced into a single estimated block. Condition (A1) guarantees the existence of a receiver community where the average connectivity of these merged groups differs by a fixed amount $\eta > 0$. This systematic gap creates a deterministic, low-rank signal in the residual matrix. The proof hinges on comparing the spectral norms of this signal and the random noise: the signal grows as $O_P(L n / K_{\max}^{3/2})$, while the noise is only $O_P(\sqrt{n L})$. After the variance normalization in $\hat{R}$, the signal term remains of order $\sqrt{n L}/K_{\max}^{3/2}$, whereas the noise is $O_P(1)$. Condition (A2) ($nL/K_{\max}^3 \to \infty$) ensures the signal dominates asymptotically, forcing $\sigma_1(\hat{R})$---and hence $\hat{T}_n$---to diverge. This sharp contrast between bounded fluctuations under a correct specification (Theorem \ref{thm:null}) and divergent growth under underfitting provides the rigorous foundation for the model-selection procedures developed in Section \ref{sec:estimation}.
\end{rem}
\section{Sequential testing algorithms for asymmetric community numbers selection}\label{sec:estimation}
Building on the goodness-of-fit test developed in Section \ref{sec:test}, we now address the core problem of joint community numbers estimation. The sequential testing framework leverages the asymptotic behavior of \(\hat{T}_n\) established in Theorems \ref{thm:null} and \ref{thm:power} to systematically identify the true \((K_s, K_r)\) while maintaining computational efficiency. This approach transforms model selection into an ordered exploration of candidate pairs, where the test statistic's dichotomous behavior—convergence to zero under correct specification versus divergence under underfitting—provides reliable stopping criteria. We further propose a ratio-based variant of this approach to enhance robustness in practical settings.
\subsection{The MLDiGoF algorithm and its estimation consistency}\label{sec3point1}
The estimation procedure evaluates candidate pairs \((k_s, k_r)\) in a lexicographical order that first compares the total number of communities \(k_s + k_r\), with smaller totals being prioritized. For pairs with an equal total, the pair with the smaller sender community number \(k_s\) is examined first. Formally, we define this search order by the following order:
\[
(k_s, k_r) < (k'_s, k'_r) \quad \text{if and only if} \quad 
\begin{cases}
k_s + k_r < k'_s + k'_r, \text{ or} \\
k_s + k_r = k'_s + k'_r \ \text{and} \ k_s < k'_s.
\end{cases}
\]

This strategy guides the search from simpler to more complex models. Let \(\mathcal{P} = \{(k_s^{(m)}, k_r^{(m)})\}_{m=1}^M\) be the complete sequence of candidate pairs from (1,1) to \((K_{\text{cand}}, K_{\text{cand}})\), listed in the order defined above, where \(K_{\text{cand}}\) is the maximum candidate number of communities, \(M = K_{\text{cand}}^2\), and the index \(m\) denotes the position of the candidate pair \((k_s, k_r)\) in the ordered sequence \(\mathcal{P}\). The detailed mapping between the index $m$ and the candidate pair $(k_s,k_r)$ for $K_{\text{cand}}=10$ can be found in the following example. 
\begin{Ex}\label{example:K10}
Consider searching over candidate sender community numbers \(k_s\) from 1 to 10 and receiver community numbers \(k_r\) from 1 to 10. Table \ref{tab:lex_order_K10} lists the candidate pairs \(\mathcal{P}\) in the order defined in this paper (first by increasing \(k_s+k_r\), then by increasing \(k_s\)), with the index \(m\) running from 1 to 100. 
\begin{table}[htbp]
\centering
\caption{Search order of candidate pairs $\mathcal{P}$ for $K_{\text{cand}}=10$}
\scriptsize
\setlength{\tabcolsep}{2pt}
\label{tab:lex_order_K10}
\begin{tabular}{|cc|cc|cc|cc|cc|cc|cc|cc|cc|cc|}
\toprule
$m$ & $(k_s, k_r)$ & $m$ & $(k_s, k_r)$ & $m$ & $(k_s, k_r)$ & $m$ & $(k_s, k_r)$ & $m$ & $(k_s, k_r)$ & $m$ & $(k_s, k_r)$ & $m$ & $(k_s, k_r)$ & $m$ & $(k_s, k_r)$ & $m$ & $(k_s, k_r)$ & $m$ & $(k_s, k_r)$ \\
\midrule
1 & (1,1) & 11 & (1,5) & 21 & (6,1) & 31 & (3,6) & 41 & (5,5) & 51 & (6,5) & 61 & (7,5) & 71 & (9,4) & 81 & (6,9) & 91 & (7,10) \\
2 & (1,2) & 12 & (2,4) & 22 & (1,7) & 32 & (4,5) & 42 & (6,4) & 52 & (7,4) & 62 & (8,4) & 72 & (10,3) & 82 & (7,8) & 92 & (8,9) \\
3 & (2,1) & 13 & (3,3) & 23 & (2,6) & 33 & (5,4) & 43 & (7,3) & 53 & (8,3) & 63 & (9,3) & 73 & (4,10) & 83 & (8,7) & 93 & (9,8) \\
4 & (1,3) & 14 & (4,2) & 24 & (3,5) & 34 & (6,3) & 44 & (8,2) & 54 & (9,2) & 64 & (10,2) & 74 & (5,9) & 84 & (9,6) & 94 & (10,7) \\
5 & (2,2) & 15 & (5,1) & 25 & (4,4) & 35 & (7,2) & 45 & (9,1) & 55 & (10,1) & 65 & (3,10) & 75 & (6,8) & 85 & (10,5) & 95 & (8,10) \\
6 & (3,1) & 16 & (1,6) & 26 & (5,3) & 36 & (8,1) & 46 & (1,10) & 56 & (2,10) & 66 & (4,9) & 76 & (7,7) & 86 & (6,10) & 96 & (9,9) \\
7 & (1,4) & 17 & (2,5) & 27 & (6,2) & 37 & (1,9) & 47 & (2,9) & 57 & (3,9) & 67 & (5,8) & 77 & (8,6) & 87 & (7,9) & 97 & (10,8) \\
8 & (2,3) & 18 & (3,4) & 28 & (7,1) & 38 & (2,8) & 48 & (3,8) & 58 & (4,8) & 68 & (6,7) & 78 & (9,5) & 88 & (8,8) & 98 & (9,10) \\
9 & (3,2) & 19 & (4,3) & 29 & (1,8) & 39 & (3,7) & 49 & (4,7) & 59 & (5,7) & 69 & (7,6) & 79 & (10,4) & 89 & (9,7) & 99 & (10,9) \\
10 & (4,1) & 20 & (5,2) & 30 & (2,7) & 40 & (4,6) & 50 & (5,6) & 60 & (6,6) & 70 & (8,5) & 80 & (5,10) & 90 & (10,6) & 100 & (10,10) \\
\bottomrule
\end{tabular}
\end{table}
\end{Ex}

The estimator \((\hat{K}_s, \hat{K}_r)\) is then the first pair in the search sequence for which the goodness-of-fit test does not reject:
\[
(\hat{K}_s, \hat{K}_r) = (k_s^{(\hat{m})}, k_r^{(\hat{m})}), \quad \text{where} \quad \hat{m} = \min\{ m : \hat{T}_n(k_s^{(m)}, k_r^{(m)}) < t_n \}.
\]

Here, \(t_n = n^{-\varepsilon}\) for some \(\varepsilon \in (0, 0.5)\) is a threshold that decays to zero with the network size \(n\). Algorithm \ref{alg:DiGoF} below summarizes the details of this sequential testing procedure.
\begin{algorithm}[H]
\caption{MLDiGoF}\label{alg:DiGoF}
\begin{algorithmic}[1]
\Require Multi-layer adjacency matrices \(\{A^{(\ell)}\}_{\ell=1}^L\), significance threshold \(t_n\) (default: \(n^{-1/5}\)), maximum candidate number \(K_{\mathrm{cand}}\) (default: \(\lfloor \sqrt{n / \log n} \rfloor\)), where \(n\) is the number of nodes
\Ensure Estimated community numbers \((\hat{K}_s, \hat{K}_r)\)
\State Generate candidate sequence \(\mathcal{P} = \{(k_s, k_r)\}_{m=1}^M\) with \(M = K_{\mathrm{cand}}^2\)
\For{\(m = 1\) to \(M\)}
    \State Let \((k_s, k_r) \gets \mathcal{P}(m)\)
    \State Compute \(\hat{T}_n(k_s, k_r)\) via Equation (\ref{eq:test_stat}) using Algorithm \ref{alg:spectral} for community estimation
    \If{\(\hat{T}_n(k_s, k_r) < t_n\)}
        \State \Return \((\hat{K}_s, \hat{K}_r) = (k_s, k_r)\)
    \EndIf
\EndFor
\State \Return \((\hat{K}_s, \hat{K}_r) = \mathcal{P}(M)\) \Comment{If no candidate satisfies \(\hat{T}_n < t_n\), return the largest candidate}
\end{algorithmic}
\end{algorithm}

\begin{rem}
The specification of \(K_{\text{cand}}\), the maximum candidate number of communities, follows naturally from the asymptotic regime imposed by Assumption \ref{assump:a3}. This assumption, requiring \(K_{\text{max}}^2 L \log n / n \to 0\), provides an implicit upper bound for the true community numbers. A theoretically coherent and computationally feasible choice is therefore \(K_{\text{cand}} = \lfloor \sqrt{n / \log n} \rfloor\). This selection guarantees that, with probability tending to one, the true pair \((K_s, K_r)\) resides within the candidate set \(\{(k_s, k_r): 1 \le k_s, k_r \le K_{\text{cand}}\}\) for all sufficiently large \(n\). Consequently, the sequential testing procedure remains consistent. Furthermore, it bounds the total search space by \(O(n / \log n)\) candidate pairs, ensuring the algorithm's practicality without compromising its theoretical foundations.
\end{rem}

The following theorem guarantees that the sequential procedure achieves joint consistency in recovering the true community numbers under ML-ScBM, provided the threshold sequence \(t_n\) satisfies certain conditions that balance the convergence rates under the null and the divergence rates under underfitting.

\begin{thm}[Consistency of the MLDiGoF algorithm]\label{thm:consistency}
Let the multi-layer directed network be generated from ML-ScBM with true parameters $(K_s, K_r, g^s, g^r, \{B^{(\ell)}\}_{\ell=1}^L)$. 
Assume:
\begin{enumerate}
    \item Assumptions \ref{assump:a1}-\ref{assump:a4} hold.
    \item Conditions (A1) and (A2) of Theorem \ref{thm:power} hold.
\end{enumerate}

Define \(\alpha_n \coloneqq \sqrt{\frac{K_{\max}^2 L \log n}{n}}\) and $\beta_n\coloneqq\frac{\sqrt{nL}}{K^{3/2}_{\max}}$. Let $\{t_n\}_{n \ge 1}$ be a sequence of positive thresholds satisfying:
\begin{enumerate}
    \item[(C1)] $\alpha_n = o(t_n)$ as $n \to \infty$.
    \item[(C2)] $t_n=o(\beta_n)$ as $n \to \infty$.
\end{enumerate}

Then the output $(\hat{K}_s, \hat{K}_r)$ of Algorithm \ref{alg:DiGoF} satisfies
\[
\lim_{n \to \infty} \mathbb{P}\bigl( (\hat{K}_s, \hat{K}_r) = (K_s, K_r) \bigr) = 1.
\]
\end{thm}

\begin{rem}[Interpretation of conditions]\label{rem:conditions}
Conditions (C1) and (C2) are necessary for MLDiGoF's estimation consistency.
\begin{itemize}
  \item Condition (C1) requires that the threshold \(t_n\) decays slower than the estimation error rate \(\alpha_n = \sqrt{K_{\max}^2 L \log n / n}\). This ensures that the difference between the practical test statistic \(\hat{T}_n\) and the oracle statistic \(T_n\) (which is of order \(O_P(\alpha_n)\)) is asymptotically negligible compared to \(t_n\) under the true model. 
  \item Condition (C2) requires that \(t_n\) grows slower than the divergence rate \(\beta_n = \sqrt{nL}/K_{\max}^{3/2}\) under underfitting. This guarantees that \(\hat{T}_n\) will eventually exceed \(t_n\) for any underfitted candidate. Together, conditions (C1) and (C2) ensure that the threshold sequence asymptotically separates the true model (where \(\hat{T}_n\) is small) from underfitted models (where \(\hat{T}_n\) is large).
  \item In practical scenarios, the number of communities \(K_{\max}\) and the number of layers \(L\) are typically bounded or grow very slowly with \(n\). In many applications, one may treat \(K_{\max}\) and \(L\) as constants (i.e., \(O(1)\)) for asymptotic analysis. Under this common setting, we have $\alpha_n = O\left(\sqrt{\frac{\log n}{n}}\right) = o(1),~~\beta_n = O\left(\sqrt{n}\right) \to \infty$. Then condition (C1) becomes \(t_n \gg \sqrt{\frac{\log n}{n}}\), and condition (C2) becomes \(t_n \ll \sqrt{n}\). A simple and convenient choice that satisfies both is \(t_n = n^{-\varepsilon}\) for any \(\varepsilon \in (0, 1/2)\). Indeed, for any $\varepsilon \in (0, 1/2)$, we have 
$\sqrt{\frac{\log n}{n}} = o(n^{-\varepsilon})$
and $n^{-\varepsilon} = o(\sqrt{n}) $. Therefore, under the typical assumption that \(K_{\max}\) and \(L\) are bounded, the default choice \(t_n = n^{-1/5}\) (i.e., \(\varepsilon = 1/5\)) in Algorithm \ref{alg:DiGoF} satisfies both (C1) and (C2). This choice provides a practical balance: it decays slowly enough to avoid premature stopping under the true model, yet shrinks to zero sufficiently fast to ensure that underfitted models are eventually rejected.
\end{itemize}
\end{rem}

Theorem \ref{thm:consistency} provides the formal guarantee that our MLDiGoF procedure is consistent—meaning it recovers the true sender and receiver community counts with probability tending to one as the network size grows. This result is crucial because it transforms the conceptually appealing sequential testing idea into a rigorously justified estimation tool. The proof carefully balances two competing rates: the convergence speed of the test statistic under the true model and its divergence rate under underfitted models. By choosing a threshold sequence that lies between these rates, the algorithm avoids two key pitfalls—premature stopping at an underfitted model and failure to stop at the true one. In practice, this theorem assures users that, under the stated regularity conditions, the method will not be misled by finite-sample fluctuations and will asymptotically locate the correct pair \((K_s, K_r)\). Thus, this theorem not only establishes theoretical credibility but also provides clear guidance for implementing the algorithm in applications where the true community structure is unknown.
\subsection{A ratio-based variant: the MLRDiGoF algorithm}\label{subsec:ratio}
Building on the established dichotomy of the test statistic, we now introduce a ratio-based estimator that identifies the transition point directly within the sequence of candidate models. By examining the ratio of successive test statistics, this method detects a clear peak corresponding to the point where the model first captures the true community structure. We refer to this alternative approach as MLRDiGoF, which provides a different operational perspective rooted in the same theoretical framework. Formally, we define the ratio statistic and present the complete algorithm as follows.

For each \(m \in \{1, 2,\dots,M\}\), let \(\hat{T}_n(m)\) denote the \(\hat{T}_{n}\) computed for the \(m\)-th candidate pair in \(\mathcal{P}\) using Equation (\ref{eq:test_stat}). Define the ratio statistic for the \(m\)-th candidate pair as
\begin{equation}\label{eq:ratio}
r_{m} = \left| \frac{\hat{T}_n(m-1)}{\hat{T}_n(m)} \right|, \quad m=2,3,\ldots, M,
\end{equation}
where the absolute value addresses potential negative value of \(\hat{T}_{n}\) for the true model \((K_s, K_r)\).

Theorems \ref{thm:null} and \ref{thm:power} guarantee that under the true model \((K_s, K_r)\) (which corresponds to a specific position \(m_*\) in the ordered candidate sequence \(\mathcal{P}\)), the upper bound of \(\hat{T}_n(m_*)\) converges to zero with high probability, while \(\hat{T}_n(m_*-1)\) diverges to infinity under underfitting when \(m_* > 1\). Consequently, we shall expect that \(r_{m_*} = |\hat{T}_n(m_*-1)/\hat{T}_n(m_*)|\) is the first significant peak (transition point) in the ratio sequence \(\{r_m\}^{M}_{m=2}\). We refer to this method as multi-layer Ratio-DiGoF (MLRDiGoF for short), which identifies the first peak in the sequence \(\{(m, r_m)\}_{m=2}^M\). The complete MLRDiGoF algorithm is summarized in Algorithm \ref{alg:RDIGoF}.

\begin{algorithm}[H]
\caption{MLRDiGoF}\label{alg:RDIGoF}
\begin{algorithmic}[1]
\Require Multi-layer adjacency matrices \(\{A^{(\ell)}\}_{\ell=1}^L\), thresholds \(t_n>0\) (default: \(n^{-1/5}\)) and \(\tau_n>0\) (default: \(8\log n\)), maximum candidate number \(K_{\mathrm{cand}}\) (default: \(\lfloor \sqrt{n / \log n} \rfloor\)), where \(n\) is the number of nodes
\Ensure Estimated community numbers \((\hat{K}_s, \hat{K}_r)\)
\State Generate candidate sequence \(\mathcal{P} = [(k_s, k_r)]_{m=1}^M\) with \(M = K_{\mathrm{cand}}^2\)
\State Compute \(\hat{T}_n(1)\) for candidate \((1,1)\) via Equation (\ref{eq:test_stat})
\If{\(\hat{T}_n(1) < t_n\)}
    \State \Return \((\hat{K}_{s},\hat{K}_{r})=(1,1)\)
\EndIf
\For{\(m = 2\) to \(M\)}
    \State Compute ratio statistic \(r_m\) via Equation (\ref{eq:ratio})
    \If{\(r_m > \tau_n\)}
        \State \Return \((\hat{K}_s, \hat{K}_r) = \mathcal{P}(m)\)
    \EndIf
\EndFor
\State \Return \((\hat{K}_s, \hat{K}_r) = \mathcal{P}(M)\) \Comment{If no candidate satisfies \(r_m > \tau_n\), return the largest candidate}
\end{algorithmic}
\end{algorithm}

The following theoretical results establish the asymptotic properties and estimation consistency of MLRDiGoF. Theorem \ref{thm:ratio} characterizes the dichotomous behavior of the ratio statistic $r_m$: it diverges at the true model due to the sharp transition of the underlying test statistic, while remaining uniformly bounded for underfitted models. This result, analogous to the asymptotic behavior of the test statistic established in Theorems \ref{thm:null} and \ref{thm:power}, provides the theoretical foundation for distinguishing the true community structure. Building upon this, Theorem \ref{thm:RDiGoF-consistency} establishes the estimation consistency of the MLRDiGoF algorithm, showing that it correctly identifies the true pair $(K_s, K_r)$ with probability tending to one, analogous to Theorem \ref{thm:consistency}.

\begin{thm}[Asymptotic behavior of $r_m$]\label{thm:ratio}
Under Assumptions \ref{assump:a1}--\ref{assump:a4} and condition (A1) of Theorem \ref{thm:power}, with \(K_s\) and \(K_r\) fixed (independent of \(n\)), let \(\{\hat{T}_n(m)\}_{m=1}^M\) be the sequence of test statistics computed via Equation (\ref{eq:test_stat}) for the lexicographically ordered candidate pairs \(\mathcal{P} = \{(k_s^{(m)}, k_r^{(m)})\}_{m=1}^M\) with \(M = K_{\mathrm{cand}}^2\) and \(K_{\mathrm{cand}} \ge \max(K_s, K_r)\). Let \(m_*\) be the index of the true pair \((K_s, K_r)\) in \(\mathcal{P}\). Define the ratio statistic for \(m = 2,\dots,M\) as in Equation (\ref{eq:ratio}). Then, we have
\begin{enumerate}
    \item (Divergence at the true model) For any fixed \(M_0 > 0\),
        \[
        \lim_{n \to \infty} \mathbb{P}\bigl( r_{m_*} > M_0 \bigr) = 1.
        \]
    \item (Boundedness under underfitting) There exists a constant \(C > 0\) (depending only on \(\delta, c_0, \eta\) from Assumptions \ref{assump:a1}, \ref{assump:a2} and condition (A1), but not on \(n\), \(L\), or the candidate index) such that for every \(m < m_*\),
        \[
        \lim_{n \to \infty} \mathbb{P}\bigl( r_m > C \bigr) = 0.
        \]
\end{enumerate}
\end{thm}

\begin{thm}[Consistency of the MLRDiGoF algorithm]\label{thm:RDiGoF-consistency}
Assume the conditions of Theorem \ref{thm:ratio} hold, i.e., Assumptions \ref{assump:a1}--\ref{assump:a4} and condition (A1) of Theorem \ref{thm:power} hold, with \(K_s\) and \(K_r\) fixed (independent of \(n\)). Let \(\tau_n\) be the threshold used in Algorithm \ref{alg:RDIGoF}. Suppose \(\tau_n\) satisfies the following three conditions:
\begin{enumerate}
    \item[(D1)] There exists a constant \(C_0 > 0\) and \(n_0 \in \mathbb{N}\) such that for all \(n \ge n_0\), \(\tau_n > C_0\), where \(C_0\) is any constant greater than the constant \(C\) from Theorem \ref{thm:ratio} part 2 (i.e., \(C_0 > C\)).
    \item[(D2)] \(\displaystyle \tau_n = o\left( \sqrt{\frac{n}{\log n}} \right)\).
\end{enumerate}
Then the output \((\hat{K}_s, \hat{K}_r)\) of Algorithm \ref{alg:RDIGoF} satisfies
\[
\lim_{n \to \infty} \mathbb{P}\bigl( (\hat{K}_s, \hat{K}_r) = (K_s, K_r) \bigr) = 1.
\]
\end{thm}
\begin{rem}[Interpretation of conditions]\label{rem:conditions-RDiGoF}
Conditions (D1) and (D2) are minimal requirements that together guarantee the threshold sequence \(\tau_n\) separates underfitted models from the true model. 
\begin{itemize}
  \item Condition (D1) demands that \(\tau_n\) eventually exceed the uniform bound \(C\) established in Theorem \ref{thm:ratio} for underfitted ratios. This ensures that, with high probability, no underfitted candidate will produce a ratio \(r_m\) larger than \(\tau_n\), thereby preventing false early stops. The strict inequality \(C_0>C\) avoids the degenerate situation where \(\tau_n\) exactly equals the bound, which could lead to unstable stopping behavior in finite samples.
  \item Condition (D2) requires that \(\tau_n\) grow slower than \(\sqrt{n/\log n}\). This is because the ratio at the true model satisfies \(r_{m_*}\gtrsim n/\sqrt{\log n}\) asymptotically (see the proof of Theorem \ref{thm:RDiGoF-consistency}). If \(\tau_n\) grew as fast as or faster than this rate, it could dominate \(r_{m_*}\) and prevent the algorithm from stopping at the true pair. Condition (D2) therefore guarantees that the exploding signal at the true model eventually surpasses the threshold.
  \item The default choice \(\tau_n=8\log n\) satisfies both conditions. Since \(\log n\to\infty\), it eventually exceeds any fixed constant \(C\) (satisfying condition (D1)). Moreover, \(\log n = o\bigl(\sqrt{n/\log n}\bigr)\) for any positive constant, so condition (D2) holds. This choice is simple, practical, and meets the theoretical requirements across a wide range of network sizes.
\end{itemize}
\end{rem}
Theorem \ref{thm:RDiGoF-consistency} provides the consistency guarantee for the ratio‑based MLRDiGoF algorithm under the assumption of a fixed number of communities. It confirms that the sharp transition in the ratio statistic—bounded for underfitted models and divergent at the true model—serves as a reliable selection criterion. While analogous to Theorem \ref{thm:consistency} in establishing consistency, Theorem \ref{thm:RDiGoF-consistency} validates a distinct, intuitive approach that detects a relative peak rather than an absolute threshold, offering a complementary perspective for model selection in the common scenario where community counts are small relative to network size.
\begin{rem}
The assumption that \(K_{\max}\) is fixed (i.e., does not grow with \(n\)) in Theorems \ref{thm:ratio} and \ref{thm:RDiGoF-consistency} is made to preserve the analytical sharpness of the ratio-based procedure. In the proofs of Theorems \ref{thm:ratio} and \ref{thm:RDiGoF-consistency}, the uniform bound \(C\) for underfitted ratios and the separation rate governing the divergence of \(r_{m_*}\) both depend on constants derived from the model parameters (such as \(\delta\), \(c_0\), and \(\eta\)), which are independent of \(K_{\max}\) only when \(K_{\max}\) is fixed. If \(K_{\max}\) were allowed to increase with \(n\), these key quantities would become functions of \(n\), intricately coupling the growth rates of the signal and noise terms in the ratio statistic. Consequently, the clean phase transition at the true model—where \(r_m\) stays bounded for all underfitted models and explodes only at the true pair—would be obscured. Moreover, choosing a single threshold sequence \(\tau_n\) that robustly separates these two regimes across all candidate pairs would require balancing multiple orders of growth, making the theory unnecessarily complex. Fixing \(K_{\max}\) isolates the core logic of the ratio test and delivers a transparent consistency guarantee, which matches the typical practical setting where the numbers of asymmetric communities are small relative to the network size.
\end{rem}

Both the MLDiGoF and MLRDiGoF procedures provide consistent estimators for the asymmetric community numbers under their respective theoretical conditions, as established in Theorems \ref{thm:consistency} and \ref{thm:RDiGoF-consistency}. However, the ratio-based MLRDiGoF tends to be more robust in practice. This robustness stems from a fundamental difference in the underlying detection logic. MLDiGoF relies on a level-crossing rule: it selects the model where the goodness-of-fit statistic $\hat{T}_n$ first falls below a vanishing threshold $t_n$. This decision is sensitive to the precise finite-sample value of $\hat{T}_n$, which itself is an estimate perturbed from its oracle counterpart by an error of order $\|\hat{R} - R\|$. Consequently, any systematic estimation bias can shift the crossing point. In contrast, MLRDiGoF implements a change-point detection strategy. It monitors the ratio statistic $r_m = |\hat{T}_n(m-1)/\hat{T}_n(m)|$, seeking a pronounced peak. Theorem \ref{thm:ratio} justifies this approach by showing that $r_m$ remains stochastically bounded for every underfitted model, yet diverges at the true model. Thus, instead of judging a noisy statistic against a decaying benchmark, MLRDiGoF identifies a clear structural break in the sequence of fits---a criterion that is invariant to any common bias affecting all $\hat{T}_n(m)$ similarly. For the common setting where the numbers of asymmetric communities are small, detecting this structural break often provides a more stable empirical criterion than judging the absolute level of a noisy statistic against a vanishing threshold.
\section{Numerical Experiments}\label{sec:simulation}
In this section, we evaluate the performance of the MLDiGoF and MLRDiGoF algorithms by simulations under the multi-layer stochastic co-block model and one real data example. 

We generate multi-layer directed networks from the ML-ScBM defined in Definition \ref{def:ML-ScBM} with the following detailed specifications. For all simulations, the sender (and the receiver) community assignments are generated by letting
each node belong to each sender (receiver) community with equal probability for all experiments.  

For each layer \(\ell = 1, \dots, L\), we generate a completely independent block probability matrix \(B^{(\ell)} \in [0,1]^{K_s \times K_r}\) as follows. First, we generate layer‑specific parameters independently for each layer: a diagonal strength parameter \(\alpha^{(\ell)} \sim \text{Uniform}[0.6, 0.8]\); an off‑diagonal base strength parameter \(\beta^{(\ell)} \sim \text{Uniform}[0.1, 0.3]\); and an additional medium strength parameter \(\gamma^{(\ell)} \sim \text{Uniform}[0.4, 0.6]\). Then, we define the base matrix \(H^{(\ell)}_0 \in \mathbb{R}^{K_s \times K_r}\) by
\[
H^{(\ell)}_0(k,l) = 
\begin{cases}
\alpha^{(\ell)}, & \text{if } k = l, \\[4pt]
\gamma^{(\ell)}, & \text{if } l = \bigl(k + K_s - 1 \bigr) \bmod K_r + 1, \\[4pt]
\beta^{(\ell)}, & \text{otherwise},
\end{cases}
\]
where \(\bmod\) denotes the modulo operation. This construction guarantees that, besides the diagonal, each sender community has a designated receiver community (cyclically shifted by \(K_s\)) that receives edges with medium probability, thereby creating distinct connectivity profiles for different receiver communities and ensuring that both the sender and receiver versions of condition (A1) are satisfied with high probability. To preserve asymmetry even when \(K_s = K_r\), we add an asymmetric perturbation. Define a matrix \(H^{(\ell)}_1 \in \mathbb{R}^{K_s \times K_r}\) whose entries are independent draws from \(\text{Uniform}[-0.1, 0.1]\). We then form \(\widetilde{B}^{(\ell)}(k,l) = H^{(\ell)}_0(k,l) + H^{(\ell)}_1(k,l)\) and clip every entry of \(\widetilde{B}^{(\ell)}\) to the interval \([0,1]\). 

Finally, let \(\rho \in (0,1)\) be a global sparsity parameter that controls the overall edge density across all layers. We scale the matrices by \(\rho\) to obtain the block‑probability matrices
\(B^{(\ell)}(k,l) = \rho \cdot \widetilde{B}^{(\ell)}(k,l), \ell = 1,\dots,L .\) We consider \(\rho \in \{0.05, 0.1, 0.2, 0.3, 0.4\}\) to examine networks with varying edge densities, from very sparse (\(\rho = 0.05\)) to moderately dense (\(\rho = 0.4\)). Then, for each layer \(\ell = 1, \dots, L\), we generate the adjacency matrix \(A^{(\ell)}\) with independent entries by Definition \ref{def:ML-ScBM}.

For all experiments, we use the following consistent evaluation metrics and default settings. The primary evaluation metric is the \textbf{accuracy}, defined as the proportion of Monte Carlo replications in which the algorithm correctly estimates both community numbers:
\[
\text{Accuracy} = \frac{\text{Number of replications with } (\hat{K}_s, \hat{K}_r) = (K_s, K_r)}{\text{Total number of replications}}.
\]
\subsection{Experiment 1: Behavior of test statistic $\hat{T}_n$ under null and alternative hypotheses}\label{subsec:exp1}
This experiment systematically verifies the theoretical properties of the test statistic \(\hat{T}_n\) under a comprehensive set of hypothesis testing scenarios as stated in Theorems \ref{thm:null} and \ref{thm:power}. We fix the number of layers \(L = 20\) and the global sparsity parameter \(\rho = 0.2\). 

We consider the true asymmetric community structure \((K_s, K_r) = (3, 5)\), which provides a representative case of moderate asymmetry with \(K_s < K_r\). To thoroughly validate the theoretical predictions, we examine the following four hypothesis testing scenarios that cover all underfitting cases implied by Theorem \ref{thm:power}.

\begin{itemize}
    \item Null hypothesis \(H_0\) (correct specification): candidate pair equals the true structure, \((K_{s0}, K_{r0}) = (3, 5)\).
    \item Alternative \(H_1\): sender-only underfitting: candidate has insufficient sender communities but correct receiver communities, \((K_{s0}, K_{r0}) = (2, 5)\).
    \item Alternative \(H_1\): receiver-only underfitting: candidate has insufficient receiver communities but correct sender communities, \((K_{s0}, K_{r0}) = (3, 4)\).
    \item Alternative \(H_1\): both-side underfitting: candidate has insufficient sender and receiver communities, \((K_{s0}, K_{r0}) = (2, 4)\).
\end{itemize}

We vary the network size \(n \in \{200,400, 600, 800, 1000\}\). For each \(n\) and each hypothesis scenario, we generate 200 independent networks and report the mean and standard deviation of \(\hat{T}_n\) over the 200 replications.

The results, detailed in Table \ref{tab:stat_behavior_comprehensive}, strongly support Theorems \ref{thm:null} and \ref{thm:power}. Under the correctly specified null hypothesis \((K_s, K_r) = (3,5)\), the absolute value of the test statistic \(\hat{T}_n\) converges to zero as \(n\) increases, with standard deviations also decreasing. Under all three underfitting alternatives, \(\hat{T}_n\) diverges. The mean values increase monotonically with \(n\). The divergence is most pronounced for both-side underfitting, followed by sender-only underfitting, while receiver-only underfitting exhibits a slower divergence. These results validate Theorem \ref{thm:power}, demonstrating that \(\hat{T}_n\) diverges under any form of underfitting, ensuring a clear separation from the null.

\begin{table}[!ht]
\centering
\caption{Behavior of the test statistic \(\hat{T}_n\) under various hypotheses for true structure \((K_s,K_r)=(3,5)\). Values are mean (standard deviation) over 200 replications.}
\label{tab:stat_behavior_comprehensive}
\begin{tabular}{c|cccc}
\hline
 & \multicolumn{4}{c}{\(\hat{T}_n\)} \\
\cline{2-5}
\(n\) & \(H_0\): (3,5) & \(H_1\): (2,5) & \(H_1\): (3,4) & \(H_1\): (2,4) \\ \hline
200 & -0.014(0.024) & 3.259 (0.229) & 0.165 (0.112) & 3.324 (0.277)\\
400 & -0.012(0.014) & 5.270 (0.339) & 0.626 (0.357) & 5.458 (0.408)\\
600 & -0.008 (0.010) & 6.900 (0.437) & 0.893 (0.262) & 7.187 (0.489)\\
800 & -0.007 (0.010) & 8.282 (0.516) & 1.247 (0.322) & 8.629 (0.546)\\
1000 & -0.007 (0.006) & 9.438 (0.627) & 1.493 (0.384) & 9.769 (0.709) \\ \hline
\end{tabular}
\end{table}

\subsection{Experiment 2: Statistical discrimination power and robustness analysis}\label{subsec:exp2}
This experiment evaluates the statistical discrimination power of the test statistic \(\hat{T}_n\) across a comprehensive set of asymmetric community structures. The primary goal is to verify that this statistic perfectly discriminates between correctly specified models and underfitted models in practical settings. Fix the network size \(n = 800\), number of layers \(L = 15\), and global sparsity parameter \(\rho = 0.2\). We consider eight true asymmetric community structures covering various asymmetry patterns:
\[
(K_s, K_r) \in \{(2,3), (2,4), (3,2), (3,4), (3,5), (4,3), (4,5), (5,4)\},
\]
where these structures include cases like \(K_s < K_r\), \(K_s > K_r\), and \(|K_s - K_r| > 1\). For each true structure, we evaluate \(\hat{T}_n\) under four scenarios:
\begin{enumerate}
    \item Correct specification (\(H_0\)): candidate pair equals the true structure.
    \item Sender-only underfitting (\(H_1^{(s)}\)): candidate has one fewer sender community.
    \item Receiver-only underfitting (\(H_1^{(r)}\)): candidate has one fewer receiver community.
    \item Both-side underfitting (\(H_1^{(b)}\)): candidate has one fewer sender and receiver communities.
\end{enumerate}

For each configuration (true community structure $\times$ hypothesis scenario), we generate 200 independent networks. The evaluation of the test statistic \(\hat{T}_n\) follows the following decision rule as defined in the MLDiGoF algorithm:

\begin{itemize}
    \item For the $\hat{T}_n$ statistic (used in MLDiGoF), we compute $\hat{T}_n$ for the specified candidate pair $(K_{s0}, K_{r0})$ via Equation (\ref{eq:test_stat}) using Algorithm \ref{alg:spectral}. The decision is correct if either:
    \begin{itemize}
        \item Under $H_0$ (correct specification): $\hat{T}_n < t_n$, where $t_n = n^{-1/5} \approx 0.0043$ for $n=800$.
        \item Under $H_1$ (any underfitting): $\hat{T}_n \geq t_n$.
    \end{itemize}
    We record the empirical probability of correct decisions over the 200 replications, along with its standard error.
\end{itemize}

Table \ref{tab:Tn_discrimination} presents the discrimination performance. Across all eight asymmetric community structures and all three underfitting types, \(\hat{T}_n\) achieves perfect discrimination, with empirical probabilities of 1.00. This provides strong evidence that:
\begin{itemize}
    \item Under correct specification, \(\hat{T}_n\) consistently falls below the threshold \(t_n = n^{-1/5} \approx 0.0043\).
    \item Under any form of underfitting (sender-only, receiver-only, or both), \(\hat{T}_n\) consistently exceeds \(t_n\).
    \item The discrimination power is robust to the direction and degree of asymmetry between sender and receiver communities.
\end{itemize}

Thus, \(\hat{T}_n\) perfectly distinguishes between correctly specified and underfitted models in these settings.
\begin{table}[!ht]
\centering
\caption{Discrimination power of the \(\hat{T}_n\) statistic. Values are empirical probabilities (standard errors) over 200 replications.}
\label{tab:Tn_discrimination}
\scalebox{0.85}{
\begin{tabular}{c|c|c|c|c}
\hline
True \((K_s, K_r)\) & Hypothesized \((K_{s0}, K_{r0})\) & Underfitting type & \(\mathbb{P}(\hat{T}_n < t_n)\) for \(H_0\) & \(\mathbb{P}(\hat{T}_n \geq t_n)\) for \(H_1\) \\ \hline
\multirow{4}{*}{\((2,3)\)} & \((2,3)\) & True model (\(H_0\)) & 1.00 (0.000) & -- \\
& \((1,3)\) & Sender only (\(H_1^{(s)}\)) & -- & 1.00 (0.000) \\
& \((2,2)\) & Receiver only (\(H_1^{(r)}\)) & -- & 1.00 (0.000) \\
& \((1,2)\) & Both (\(H_1^{(b)}\)) & -- & 1.00 (0.000) \\ \hline
\multirow{4}{*}{\((2,4)\)} & \((2,4)\) & True model (\(H_0\)) & 1.00 (0.000) & -- \\
& \((1,4)\) & Sender only (\(H_1^{(s)}\)) & -- & 1.00 (0.000) \\
& \((2,3)\) & Receiver only (\(H_1^{(r)}\)) & -- & 1.00 (0.000) \\
& \((1,3)\) & Both (\(H_1^{(b)}\)) & -- & 1.00 (0.000) \\ \hline
\multirow{4}{*}{\((3,2)\)} & \((3,2)\) & True model (\(H_0\)) & 1.00 (0.000) & -- \\
& \((2,2)\) & Sender only (\(H_1^{(s)}\)) & -- & 1.00 (0.000) \\
& \((3,1)\) & Receiver only (\(H_1^{(r)}\)) & -- & 1.00 (0.000) \\
& \((2,1)\) & Both (\(H_1^{(b)}\)) & -- & 1.00 (0.000) \\ \hline
\multirow{4}{*}{\((3,4)\)} & \((3,4)\) & True model (\(H_0\)) & 1.00 (0.000) & -- \\
& \((2,4)\) & Sender only (\(H_1^{(s)}\)) & -- & 1.00 (0.000) \\
& \((3,3)\) & Receiver only (\(H_1^{(r)}\)) & -- & 1.00 (0.000) \\
& \((2,3)\) & Both (\(H_1^{(b)}\)) & -- & 1.00 (0.000) \\ \hline
\multirow{4}{*}{\((3,5)\)} & \((3,5)\) & True model (\(H_0\)) & 1.00 (0.000) & -- \\
& \((2,5)\) & Sender only (\(H_1^{(s)}\)) & -- & 1.00 (0.000) \\
& \((3,4)\) & Receiver only (\(H_1^{(r)}\)) & -- & 1.00 (0.000) \\
& \((2,4)\) & Both (\(H_1^{(b)}\)) & -- & 1.00 (0.000) \\ \hline
\multirow{4}{*}{\((4,3)\)} & \((4,3)\) & True model (\(H_0\)) & 1.00 (0.000) & -- \\
& \((3,3)\) & Sender only (\(H_1^{(s)}\)) & -- & 1.00 (0.000) \\
& \((4,2)\) & Receiver only (\(H_1^{(r)}\)) & -- & 1.00 (0.000) \\
& \((3,2)\) & Both (\(H_1^{(b)}\)) & -- & 1.00 (0.000) \\ \hline
\multirow{4}{*}{\((4,5)\)} & \((4,5)\) & True model (\(H_0\)) & 1.00 (0.000) & -- \\
& \((3,5)\) & Sender only (\(H_1^{(s)}\)) & -- & 1.00 (0.000) \\
& \((4,4)\) & Receiver only (\(H_1^{(r)}\)) & -- & 1.00 (0.000) \\
& \((3,4)\) & Both (\(H_1^{(b)}\)) & -- & 1.00 (0.000) \\ \hline
\multirow{4}{*}{\((5,4)\)} & \((5,4)\) & True model (\(H_0\)) & 1.00 (0.000) & -- \\
& \((4,4)\) & Sender only (\(H_1^{(s)}\)) & -- & 1.00 (0.000) \\
& \((5,3)\) & Receiver only (\(H_1^{(r)}\)) & -- & 1.00 (0.000) \\
& \((4,3)\) & Both (\(H_1^{(b)}\)) & -- & 1.00 (0.000) \\ \hline
\end{tabular}
}
\end{table}
\subsection{Experiment 3: Estimation accuracy under varied network sizes and sparsity levels}
This experiment evaluates the estimation accuracy of MLDiGoF and MLRDiGoF across a wide range of network sizes, global sparsity levels, and asymmetric community structures. We set $L=15$, network sizes \(n \in \{200, 400, 600, 800, 1000\}\), global sparsity parameters \(\rho \in \{0.1, 0.2, 0.3, 0.4,0.5\}\), and several true asymmetric community structures: \((K_s, K_r) \in \{(1,1), (1,3), (2,2), (2,3), (2,4), (3,4), (3,5), (4,4), (4,5)\}\). For each combination of parameters, we generate 200 independent networks and run both MLDiGoF and MLRDiGoF with their default settings. For each replication, we record the estimated community numbers \((\hat{K}_s, \hat{K}_r)\) and compute the accuracy as defined earlier.

The accuracy results are shown in Table \ref{tab:accuracy_summary}. Several key trends emerge:
\begin{itemize}
    \item For fixed sparsity \(\rho\) and true community structure \((K_s, K_r)\), accuracy improves monotonically with network size \(n\), confirming the consistency results in Theorems \ref{thm:consistency} and \ref{thm:RDiGoF-consistency}. 
    \item For fixed \(n\) and \((K_s, K_r)\), accuracy improves with increasing global sparsity \(\rho\). The improvement is particularly dramatic for complex asymmetric structures. For instance, for \((K_s, K_r) = (3,5)\) with \(n=600\), accuracy jumps from 0.28 at \(\rho=0.1\) to 0.95 at \(\rho=0.2\) and 1.00 at \(\rho \ge 0.3\) for the MLDiGoF method.
    \item Structures with larger total communities \(K_s + K_r\) are inherently more challenging to estimate. For example, at \(n=600\) and \(\rho=0.2\), the accuracy for the relatively simple structure \((2,3)\) is 1.00, while for the more complex \((3,5)\) it is 0.95 for MLDiGoF and 1.00 for MLRDiGoF.
    \item MLRDiGoF generally matches or slightly outperforms MLDiGoF, particularly in sparse and challenging settings. For \((3,5)\) with \(n=600\) and \(\rho=0.2\), MLRDiGoF achieves an accuracy of 1.00 compared to 0.95 for MLDiGoF. For \((2,4)\) with \(n=200\) and \(\rho=0.3\), MLRDiGoF accuracy is 0.95 while MLDiGoF is 0.68, highlighting the robustness of the ratio-based approach.
\end{itemize}

Overall, both algorithms demonstrate consistent estimation performance across diverse asymmetric community structures, with accuracy approaching 1 as either \(n\) or \(\rho\) increases, confirming the consistency theorems.
\begin{table}[htbp]
\centering
\caption{Estimation accuracy of MLDiGoF and MLRDiGoF under selected challenging asymmetric settings. Values are accuracy over 100 replications.}\label{tab:accuracy_summary}
\scriptsize
\setlength{\tabcolsep}{2pt}
\begin{tabular}{@{}cc ccccc ccccc@{}}
\toprule
\multirow{2}{*}{$(K_s, K_r)$} & \multirow{2}{*}{$n$} & \multicolumn{5}{c}{MLDiGoF} & \multicolumn{5}{c}{MLRDiGoF} \\
\cmidrule(lr){3-7} \cmidrule(lr){8-12}
 & & $\rho=0.1$ & $\rho=0.2$ & $\rho=0.3$ & $\rho=0.4$& $\rho=0.5$ & $\rho=0.1$ & $\rho=0.2$ & $\rho=0.3$ & $\rho=0.4$ & $\rho=0.5$\\
\midrule
\multirow{5}{*}{ (1,1) } 
& 200  & 1.00 & 1.00 & 1.00 & 1.00 & 1.00 & 1.00 & 1.00 & 1.00 &1.00&1.00\\
& 400  & 1.00 & 1.00 & 1.00 & 1.00 & 1.00 & 1.00 & 1.00 & 1.00 &1.00&1.00\\
& 600  & 1.00 & 1.00 & 1.00 & 1.00 & 1.00 & 1.00 & 1.00 & 1.00 &1.00&1.00\\
& 800  & 1.00 & 1.00 & 1.00 & 1.00 & 1.00 & 1.00 & 1.00 & 1.00 &1.00&1.00\\
& 1000 & 1.00 & 1.00 & 1.00 & 1.00 & 1.00 & 1.00 & 1.00 & 1.00 &1.00&1.00\\
\addlinespace
\hline
\multirow{5}{*}{ (1,3) } 
& 200  & 0.68 & 0.99 & 1.00 & 1.00 & 1.00 & 1.00 & 1.00 & 1.00 &1.00&1.00\\
& 400  & 0.98 & 1.00 & 1.00 & 1.00 & 1.00 & 1.00 & 1.00 & 1.00 &1.00&1.00\\
& 600  & 1.00 & 1.00 & 1.00 & 1.00 & 1.00 & 1.00 & 1.00 & 1.00 &1.00&1.00\\
& 800  & 1.00 & 1.00 & 1.00 & 1.00 & 1.00 & 1.00 & 1.00 & 1.00 &1.00&1.00\\
& 1000 & 1.00 & 1.00 & 1.00 & 1.00 & 1.00 & 1.00 & 1.00 & 1.00 &1.00&1.00\\
\addlinespace
\hline   
\multirow{5}{*}{ (2,2) } 
& 200  & 1.00 & 1.00 & 1.00 & 1.00 & 1.00 & 1.00 & 1.00 & 1.00 &1.00&1.00\\
& 400  & 1.00 & 1.00 & 1.00 & 1.00 & 0.98 & 1.00 & 1.00 & 1.00 &1.00&1.00\\
& 600  & 1.00 & 1.00 & 1.00 & 1.00 & 1.00 & 1.00 & 1.00 & 1.00 &1.00&1.00\\
& 800  & 1.00 & 1.00 & 1.00 & 1.00 & 1.00 & 1.00 & 1.00 & 1.00 &1.00&1.00\\
& 1000 & 1.00 & 1.00 & 1.00 & 1.00 & 1.00 & 1.00 & 1.00 & 1.00 &1.00&1.00\\
\addlinespace
\hline   
\multirow{5}{*}{ (2,3) } 
& 200  & 0.90 & 1.00 & 1.00 & 1.00 & 1.00 & 0.96 & 1.00 & 1.00 &1.00&1.00\\
& 400  & 1.00 & 1.00 & 1.00 & 1.00 & 1.00 & 1.00 & 1.00 & 1.00 &1.00&1.00\\
& 600  & 1.00 & 1.00 & 1.00 & 1.00 & 1.00 & 1.00 & 1.00 & 1.00 &1.00&1.00\\
& 800  & 1.00 & 1.00 & 1.00 & 1.00 & 1.00 & 1.00 & 1.00 & 1.00 &1.00&1.00\\
& 1000 & 1.00 & 1.00 & 1.00 & 1.00 & 1.00 & 1.00 & 1.00 & 1.00 &1.00&1.00\\
\addlinespace
\hline
\multirow{5}{*}{ (2,4) } 
& 200  & 0.00 & 0.36 & 0.68 & 0.95 & 1.00 & 0.15 & 0.85 & 0.95 &1.00&1.00\\
& 400  & 0.32 & 0.94 & 1.00 & 0.99 & 1.00 & 0.50 & 1.00 & 1.00 &1.00&1.00\\
& 600  & 0.85 & 0.99 & 1.00 & 1.00 & 1.00 & 1.00 & 1.00 & 1.00 &1.00&1.00\\
& 800  & 0.96 & 1.00 & 1.00 & 1.00 & 1.00 & 1.00 & 1.00 & 1.00 &1.00&1.00\\
& 1000 & 1.00 & 1.00 & 1.00 & 1.00 & 1.00 & 1.00 & 1.00 & 1.00 &1.00&1.00\\
\addlinespace
\hline
\multirow{5}{*}{ (3,4) } 
& 200  & 0.85 & 1.00 & 1.00 & 1.00 & 1.00 & 0.79 & 1.00 & 1.00 &1.00&1.00\\
& 400  & 1.00 & 1.00 & 1.00 & 1.00 & 1.00 & 1.00 & 1.00 & 1.00 &1.00&1.00\\
& 600  & 1.00 & 1.00 & 1.00 & 1.00 & 1.00 & 1.00 & 1.00 & 1.00 &1.00&1.00\\
& 800  & 1.00 & 1.00 & 1.00 & 1.00 & 1.00 & 1.00 & 1.00 & 1.00 &1.00&1.00\\
& 1000 & 1.00 & 1.00 & 1.00 & 1.00 & 1.00 & 1.00 & 1.00 & 1.00 &1.00&1.00\\
\addlinespace
\hline
\multirow{5}{*}{ (3,5) } 
& 200  & 0.02 & 0.02 & 0.27 & 0.68 & 0.90 & 0.20 & 0.40 & 0.85 &0.95&1.00\\
& 400  & 0.01 & 0.63 & 0.91 & 0.990 & 1.00 & 0.45 & 1.00 & 1.00 &1.00&1.00\\
& 600  & 0.28 & 0.95 & 1.00 & 1.00 & 1.00 & 0.75 & 1.00 & 1.00 &1.00&1.00\\
& 800  & 0.62 & 1.00 & 1.00 & 1.00 & 1.00 & 1.00 & 1.00 & 1.00 &1.00&1.00\\
& 1000 & 0.87 & 1.00 & 1.00 & 1.00 & 1.00 & 1.00 & 1.00 & 1.00 &1.00&1.00\\
\addlinespace
\hline
\multirow{5}{*}{ (4,4) } 
& 200  & 0.89 & 1.00 & 1.00 & 1.00 & 1.00 & 0.13 & 0.82 & 0.92 &0.96&0.98\\
& 400  & 1.00 & 1.00 & 1.00 & 1.00 & 1.00 & 0.92 & 1.00 & 1.00 &1.00&1.00\\
& 600  & 1.00 & 1.00 & 1.00 & 1.00 & 1.00 & 1.00 & 1.00 & 1.00 &1.00&1.00\\
& 800  & 1.00 & 1.00 & 1.00 & 1.00 & 1.00 & 1.00 & 1.00 & 1.00 &1.00&1.00\\
& 1000 & 1.00 & 1.00 & 1.00 & 1.00 & 1.00 & 1.00 & 1.00 & 1.00 &1.00&1.00\\
\addlinespace
\hline
\multirow{5}{*}{ (4,5) } 
& 200  & 0.21 & 0.98 & 1.00 & 1.00 & 1.00 & 0.00 & 0.71 & 0.97 &1.00&1.00\\
& 400  & 0.98 & 1.00 & 1.00 & 1.00 & 1.00 & 0.56 & 1.00 & 1.00 &1.00&1.00\\
& 600  & 1.00 & 1.00 & 1.00 & 1.00 & 1.00 & 1.00 & 1.00 & 1.00 &1.00&1.00\\
& 800  & 1.00 & 1.00 & 1.00 & 1.00 & 1.00 & 1.00 & 1.00 & 1.00 &1.00&1.00\\
& 1000 & 1.00 & 1.00 & 1.00 & 1.00 & 1.00 & 1.00 & 1.00 & 1.00 &1.00&1.00\\
\addlinespace
\bottomrule
\end{tabular}
\end{table}
\subsection{Experiment 4: Sensitivity to threshold parameters}
This experiment examines the sensitivity of MLDiGoF to the decay parameter \(\varepsilon\) in \(t_n = n^{-\varepsilon}\), and of MLRDiGoF to the threshold \(\tau_n\). We fix the network size \(n = 800\), the number of layers \(L = 15\), the global sparsity parameter \(\rho = 0.2\), and the true asymmetric community structure \((K_s, K_r) = (3,5)\). For MLDiGoF, we vary \(\varepsilon\) from 0.1 to 1.00 in steps of 0.1. For each \(\varepsilon\), we generate 200 independent networks, run MLDiGoF with threshold \(t_n = n^{-\varepsilon}\), and record the accuracy. For MLRDiGoF, we consider two types of threshold sequences: constant thresholds \(\tau \in \{2, 4, 6, 8, 10, 12, 14, 16, 18, 20\}\) and growing thresholds \(\tau_n = a \log n\) with \(a \in \{0.5, 1.0, 1.5, 2.0, 2.5, 3.0, 3.5, 4, 4.5, 5\}\). For each threshold setting, we generate 200 independent networks and run MLRDiGoF, recording the accuracy.

The sensitivity analysis results are presented in Figure \ref{tab:threshold_sensitivity}. For MLDiGoF, it maintains high accuracy (\(\ge 0.9\)) for the decay parameter \(\varepsilon\) in the range \((0.05, 0.70]\). Its accuracy begins to decline for \(\varepsilon \geq0.70\). This pattern aligns precisely with Remark \ref{rem:conditions}, which requires \(0 < \varepsilon < 0.5\) for the threshold conditions (C1) and (C2) to hold. The default choice \(\varepsilon = 0.2\) yields almost perfect accuracy in this experiment. For MLRDiGoF with constant thresholds $\tau$, accuracy is poor for $\tau\leq6$ but improves as $\tau$ increases. With growing thresholds $\tau_n = a \log n$, accuracy is larger than 0.9 for $a\geq4$. The algorithm's default $\tau_n = 8 \log n$ (corresponding to $a = 8$) performs alsmost perfectly, as expected. These results confirm Theorem \ref{thm:RDiGoF-consistency}, which requires $\tau_n$ to exceed a certain constant (Condition D1) and grow slower than $\sqrt{n/\log n}$ (Condition D2). Both constant thresholds $\tau \ge 8$ and growing thresholds with $a \ge 8$ satisfy these conditions.
\begin{figure}[H]
\centering
{\includegraphics[width=0.66\textwidth]{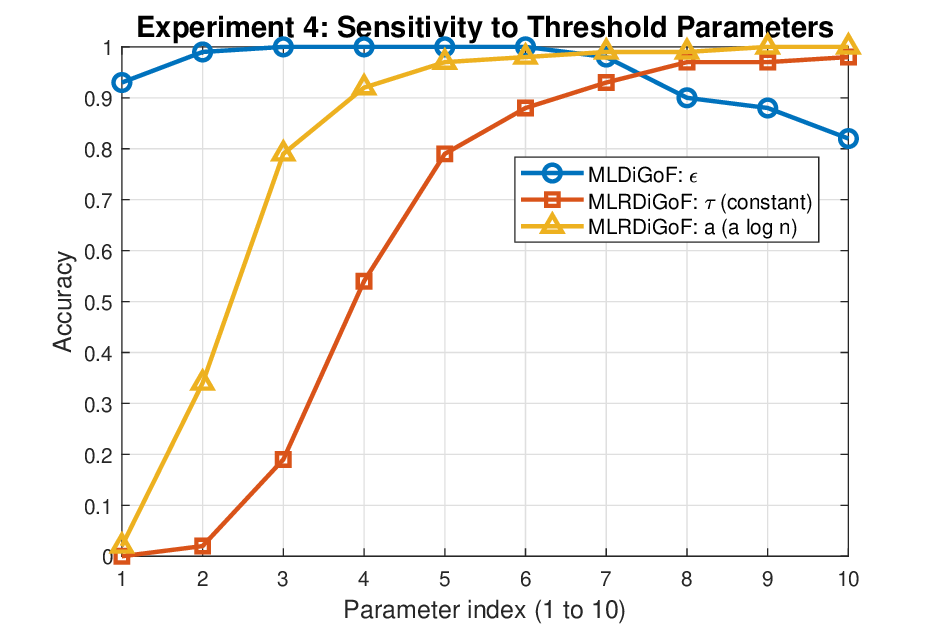}}
\caption{Sensitivity analysis of MLDiGoF and MLRDiGoF to threshold parameters ($n=800$, $L=15$, $\rho=0.2$, $(K_s,K_r)=(3,5)$). For each method, the $x$-axis index $1$ to $10$ corresponds to an increasing sequence of the associated threshold parameter: $\varepsilon$ for MLDiGoF (blue circles), constant $\tau$ for MLRDiGoF (red squares), and scale factor $a$ for MLRDiGoF with $\tau_n = a\log n$ (yellow triangles). Accuracy is computed over independent replications.}
\label{tab:threshold_sensitivity} 
\end{figure}
\subsection{Real data example}
We consider a multi-layer directed network built from the Food and Agriculture Organization (FAO) Multiplex Trade Network, available at \url{https://manliodedomenico.com/data.php}. In this dataset, nodes represent countries and each layer corresponds to a specific agricultural product. The data are for the year 2010 and record annual import and export values among countries. We select the 30 products with the largest total trade volume. For each selected product, we construct a layer by placing a directed edge from the exporter to the importer if the trade value between the two countries reaches or exceeds 100. The resulting network consists of 213 countries and 30 layers, covering a range of major food and agricultural commodities. This dataset provides a realistic multi-layer directed network for evaluating our community number estimation methods.

Figure~\ref{fig:FAO} displays the test statistics for the FAO network with \(K_{\mathrm{cand}}=10\), which is chosen to be larger than the default \(\lfloor\sqrt{n/\log n}\rfloor\approx 6\) (for \(n=213\)) to expand the search space.  The left panel of Figure~\ref{fig:FAO} shows that \(\hat{T}_n(m)\) remains above \(10\) for all \(m\), far exceeding the MLDiGoF threshold \(t_n=n^{-0.2}\approx 0.3422\). Consequently, MLDiGoF never encounters a candidate with \(\hat{T}_n(m)<t_n\) and returns the largest candidate pair \((K_{\mathrm{cand}},K_{\mathrm{cand}})=(10,10)\). This indicates that the FAO data lacks a sufficiently strong community signal to drive \(\hat{T}_n\) below the prescribed decay threshold.  The right panel plots the ratio statistic \(r_m\). Its global maximum is below \(1.8\), which is much smaller than the default threshold \(\tau_n=8\log n\) used in MLRDiGoF. Hence, with the default threshold, MLRDiGoF would also terminate at \((10,10)\). However, by detecting the peak of \(r_m\), one can select a smaller threshold to stop at the global maximum located at \(m=42\), corresponding to the candidate pair \((k_s,k_r)=(6,4)\) from Table \ref{tab:lex_order_K10}. This yields an estimated community structure of \((6,4)\) for the FAO network.  

The contrast between the two algorithms highlights a fundamental advantage of the ratio‑based approach. MLDiGoF relies on an absolute threshold \(t_n\); its stopping decision is sensitive to the precise finite‑sample value of \(\hat{T}_n\) and fails when the test statistic never drops below that threshold. In contrast, MLRDiGoF monitors the relative changes in consecutive test statistics. Even when all \(\hat{T}_n(m)\) are large, a sharp transition in the ratio sequence—appearing as a clear peak—can signal the move from underfitted to adequately specified models. This relative measure makes MLRDiGoF more robust in real‑world settings where clear community structure may be absent. By detecting this peak rather than an absolute level, MLRDiGoF provides a meaningful estimate even when the absolute goodness‑of‑fit measure remains uniformly high.
\begin{figure}[htbp]
\centering
\resizebox{\columnwidth}{!}{
{\includegraphics[width=3\textwidth]{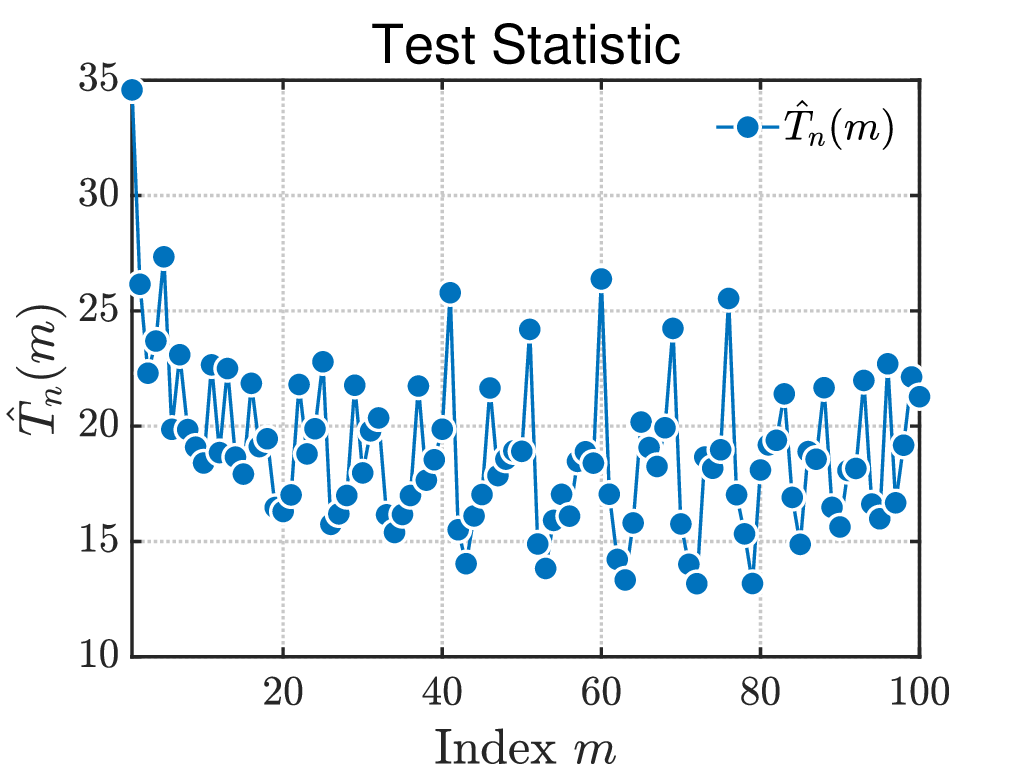}}
{\includegraphics[width=3\textwidth]{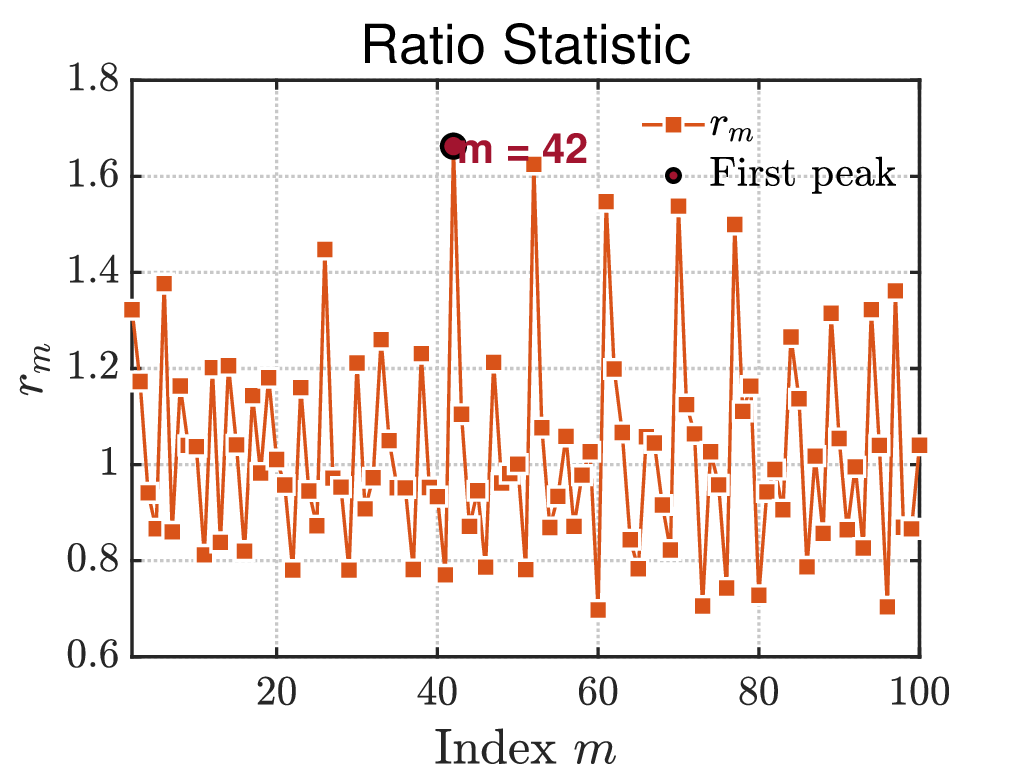}}
}
\caption{Test statistic \(\hat{T}_n(m)\) and ratio statistic \(r_m\) for ordered candidate pairs \(1 \leq m \leq 100\) (i.e., \(K_{\mathrm{cand}}=10\)) for the FAO network. The red circle in the right panel marks the global maximum of \(r_m\) at \(m=42\), corresponding to the candidate pair \((k_s,k_r)=(6,4)\).}
\label{fig:FAO}
\end{figure}
\section{Conclusion}\label{sec:conclusion}
This paper addresses the fundamental challenge of jointly estimating the number of sender and receiver communities in multi-layer directed networks under the multi-layer stochastic co-block model. By introducing a novel goodness-of-fit test based on the largest singular value of a normalized residual matrix, we establish a sharp dichotomy: under the null model, the test statistic’s upper bound converges to 0 with high probability, while itself diverges to infinity under underfitting. This theoretical insight enables the design of two efficient sequential testing algorithms MLDiGoF and its ratio-based variant MLRDiGoF that lexicographically search through candidate pairs of community numbers and terminate at the smallest adequate model. Both methods are proven to consistently recover the true asymmetric community counts. Extensive numerical experiments validate the efficacy and accuracy of the proposed methods, demonstrating their robustness across varying network sparsity levels and community structures. Real data applications further confirm the practical utility of the ratio-based approach in recovering meaningful asymmetric community structures in complex multi-layer directed networks.

The proposed framework can be extended in several meaningful directions. Promising model-based extensions include developing analogous goodness-of-fit tests for multi-layer degree-corrected ScBMs to account for degree heterogeneity \citep{karrer2011stochastic,ma2021determining}, for multi-layer bipartite networks \citep{zhou2019analysis} to consider different types of nodes, and for multi-layer mixed-membership ScBMs to allow for overlapping asymmetric communities \citep{MMSB,mao2021estimating,jin2024mixed,qing2025discovering}. Methodologically, relaxing conditions such as balanced community sizes or perfect label recovery would enhance robustness. Further adaptations could address networks with weighted edges or dynamic settings where community memberships evolve smoothly across layers. From a computational perspective, accelerating the methodology via efficient  randomized algorithms developed in \citep{guo2023randomized,deng2024subsampling}  for large-scale multi-layer directed networks remains an important challenge for theory and practice.
\section*{CRediT authorship contribution statement}
\textbf{Huan Qing} is the sole author of this article.
\section*{Declaration of competing interest}
The authors declare no competing interests.

\section*{Data availability}
Data will be made available on request.


\appendix
\section{Behaviors of the test statistics}\label{sec:proofs}
Throughout the appendix, we use the notation introduced in the main text. \(\|\cdot\|\) denotes the spectral norm, \(\|\cdot\|_F\) the Frobenius norm, 
and \(\sigma_1(\cdot)\) the largest singular value. For two sequences \(a_n,b_n\) we write \(a_n\ll b_n\) when \(a_n=o(b_n)\). The symbols \(C\) and \(c\) denote generic positive constants whose values may change from line to line.
\subsection{Proof of Lemma \ref{ideal0}}
\begin{proof}
By Assumption \ref{assump:a1}, we have
\[
\sum_{\ell=1}^L \Omega^{(\ell)}(i,j)(1 - \Omega^{(\ell)}(i,j)) \geq L \delta(1 - \delta),
\]
which gives
  \[
  |R(i,j)| \leq \frac{L}{\sqrt{(n-1) L \delta(1-\delta)}} = \sqrt{\frac{L}{(n-1)\delta(1-\delta)}} =: M_n.
  \]

The entries \(\{R(i,j): i \neq j\}\) are mutually independent (because edges are independent across layers and pairs). The diagonal entries are zero. The non-symmetric version of Corollary 3.12 in
 \citep{Afonso2016}) is helpful for our proof. We state it below:
\begin{lem}\label{Rec312}
(Rectangular version of Corollary 3.12 in \citep{Afonso2016}) Let \(X\) be an \(n_1 \times n_2\) random matrix with independent entries \(X_{ij}\) satisfying \(|X_{ij}| \leq \tilde{\sigma}_{*}\). Define
\[
\tilde{\sigma} = \max\left\{ \max_{i} \sqrt{\sum_{j} \mathbb{E}[X_{ij}^2]}, \max_{j} \sqrt{\sum_{i} \mathbb{E}[X_{ij}^2]} \right\}.
\]

Then there exists a universal constant \(C > 0\) such that for any \(0 < \eta \leq 1/2\) and \(t \geq 0\),
\[
\mathbb{P}\left( \|X\| \geq (1+\eta) 2\tilde{\sigma} + t \right) \leq (n_1\wedge n_2) \exp\left(-\frac{t^2}{C \tilde{\sigma}^2_{*}}\right).
\]
\end{lem}

We apply Lemma \ref{Rec312} to the ideal residual matrix \(R\), which is \(n \times n\). Given that
  \[
  \sum_{j} \mathbb{E}[R(i,j)^2] = \sum_{j \neq i} \frac{1}{n-1} = 1, \sum_{i} \mathbb{E}[R(i,j)^2] = \sum_{i \neq j} \frac{1}{n-1} = 1,
  \]
we have \(\tilde{\sigma} = \max\{1, 1\} = 1\).

Choose \(\eta = \epsilon/4\) with \(\epsilon \leq 2\), we have
\[
(1+\eta)2\tilde{\sigma} = 2(1 + \epsilon/4) = 2 + \epsilon/2.
\]

Set \(t = \epsilon/2\). Then
\[
(1+\eta)2\tilde{\sigma} + t = 2 + \epsilon/2 + \epsilon/2 = 2 + \epsilon.
\]

The entries of $R$ are bounded by \(\tilde{\sigma}_{*} = M_n\). Applying Lemma \ref{Rec312} obtains
\[
\mathbb{P}\left( \|R\| \geq 2 + \epsilon \right) \leq n \exp\left(-\frac{t^2}{C M_n^2}\right) = n \exp\left(-\frac{\epsilon^2}{4C M_n^2}\right).
\]

Recall \(M_n^2 = \dfrac{L}{(n-1)\delta(1-\delta)}\). Let \(C' = \dfrac{4C}{\delta(1-\delta)}\). We have
\begin{align}\label{eqp1}
\mathbb{P}\left( \|R\| \geq 2 + \epsilon \right) \leq n \exp\left(-\frac{\epsilon^2 \delta(1-\delta)(n-1)}{4C L}\right) = n \exp\left(-\frac{\epsilon^2 (n-1)}{C' L}\right).
\end{align}

By Assumption \ref{assump:a3}, we have \( \frac{K_{\max}^2 L \log n}{n} \to 0 \). Since \( K_{\max} \geq 1 \), this implies \( \frac{L \log n}{n} \to 0 \), and consequently,
\[
\frac{n-1}{L \log n} \to \infty.
\]

Taking the logarithm of the right-hand side in Equation \eqref{eqp1}  yields:
\[
\log\left[ n \exp\left(-\frac{\epsilon^2 (n-1)}{C' L}\right) \right]
= \log n - \frac{\epsilon^2 (n-1)}{C' L}.
\]

We can rewrite this as:
\[
\log n - \frac{\epsilon^2 (n-1)}{C' L}
= \log n \left( 1 - \frac{\epsilon^2}{C'} \cdot \frac{n-1}{L \log n} \right).
\]

Since \( \frac{n-1}{L \log n} \to \infty \), the term inside the parentheses tends to \(-\infty\), and thus:
\[
\log n - \frac{\epsilon^2 (n-1)}{C' L} \to -\infty.
\]

Exponentiating, we obtain:
\[
n \exp\left(-\frac{\epsilon^2 (n-1)}{C' L}\right) \to 0.
\]

Therefore,
\[
\mathbb{P}\left( \|R\| \geq 2 + \epsilon \right) \to 0.
\]

Since \(\sigma_1(R) = \|R\|\), we have
\[
\mathbb{P}\left( \sigma_1(R) < 2 + \epsilon \right) \to 1.
\]

Therefore, we get
\[
\mathbb{P}(T_n < \epsilon) = \mathbb{P}(\sigma_1(R) - 2 < \epsilon) = \mathbb{P}(\sigma_1(R) < 2 + \epsilon) \to 1.
\]

This completes the proof of this lemma.
\end{proof}
\subsection{Properties of estimated parameters}
\begin{lem}[Concentration of the block probability estimator]\label{lem:concentration}
For any estimated sender community $s_0$ and receiver community $r_0$ with sizes $n_{s_0} = |\{i: \hat{g}^s(i) = s_0\}|$ and $n_{r_0} = |\{j: \hat{g}^r(j) = r_0\}|$, we have
\[
|\hat{B}^{(\ell)}(s_0, r_0) - \mathbb{E}[\hat{B}^{(\ell)}(s_0, r_0)]| \leq \sqrt{\frac{6 \log n}{n_{s_0} n_{r_0}}}
\]
with probability at least $1 - O(n^{-3})$, where $\mathbb{E}[\hat{B}^{(\ell)}(s_0, r_0)]$ is the expectation conditioned on $\hat{g}^s$ and $\hat{g}^r$.
\end{lem}

\begin{proof}
Define the sets: \(\hat{C}_{s_0}^s = \{ i \in [n] : \hat{g}^s(i) = s_0 \}, \quad \hat{C}_{r_0}^r = \{ j \in [n] : \hat{g}^r(j) = r_0 \}.\) Set $m = n_{s_0} n_{r_0}$. If $m = 0$, then the estimated community is empty, and by definition $\hat{B}^{(\ell)}(s_0, r_0) = 0$ and $\mathbb{E}[\hat{B}^{(\ell)}(s_0, r_0) \mid \hat{g}^s, \hat{g}^r] = 0$, so the inequality holds trivially. We thus assume $m \ge 1$ in the following.

For each $i \in \hat{C}_{s_0}^s$ and $j \in \hat{C}_{r_0}^r$, $A^{(\ell)}(i,j)$ is a Bernoulli random variable with success probability
\[
p_{ij} = B^{(\ell)}(g^s(i), g^r(j)),
\]
where $g^s(i)$ and $g^r(j)$ are the true (unknown) sender and receiver community labels. 

Define the centered random variables \(X_{ij} = A^{(\ell)}(i,j) - p_{ij}, \quad \forall i \in \hat{C}_{s_0}^s, j \in \hat{C}_{r_0}^r.\) We see that $X_{ij}$ are independent and satisfy
\[
\mathbb{E}[X_{ij} \mid \hat{g}^s, \hat{g}^r] = 0, \quad 
|X_{ij}| \le 1, \quad 
\operatorname{Var}(X_{ij} \mid \hat{g}^s, \hat{g}^r) = p_{ij}(1-p_{ij}) \le \frac{1}{4}.
\]

Set \(S = \sum_{i \in \hat{C}_{s_0}^s} \sum_{j \in \hat{C}_{r_0}^r} X_{ij}\). We have
\[
\mathbb{E}[S \mid \hat{g}^s, \hat{g}^r] = 0, \quad 
\sigma^2 := \operatorname{Var}(S \mid \hat{g}^s, \hat{g}^r) = \sum_{i \in \hat{C}_{s_0}^s} \sum_{j \in \hat{C}_{r_0}^r} p_{ij}(1-p_{ij}) \le \frac{m}{4}.
\]

Note that the block probability estimator is defined as \(\hat{B}^{(\ell)}(s_0, r_0) = \frac{1}{m} \sum_{i \in \hat{C}_{s_0}^s} \sum_{j \in \hat{C}_{r_0}^r} A^{(\ell)}(i,j)\),
and its conditional expectation is \(\mathbb{E}[\hat{B}^{(\ell)}(s_0, r_0) \mid \hat{g}^s, \hat{g}^r] = \frac{1}{m} \sum_{i \in \hat{C}_{s_0}^s} \sum_{j \in \hat{C}_{r_0}^r} p_{ij}\). Thus, we have
\[
\hat{B}^{(\ell)}(s_0, r_0) - \mathbb{E}[\hat{B}^{(\ell)}(s_0, r_0) \mid \hat{g}^s, \hat{g}^r] = \frac{S}{m}.
\]

Our goal is to control the magnitude of $S/m$. We apply Bernstein's inequality in \citep{tropp2012user}. For any $v > 0$,
\begin{align}\label{eqp2}
\mathbb{P}( |S| \ge v \mid \hat{g}^s, \hat{g}^r ) \le 2 \exp\left( -\frac{v^2/2}{ \sigma^2 + v/3 } \right) \le 2 \exp\left( -\frac{v^2/2}{ m/4 + v/3 } \right). 
\end{align}

We choose $v = \sqrt{6 m \log n}$. Note that since $|X_{ij}| \le 1$, we have $|S| \le m$. Therefore, if $v > m$, the event $\{ |S| \ge v \}$ cannot occur. Now $v > m$ if and only if
\[
\sqrt{6 m \log n} > m \quad \Longleftrightarrow \quad 6 \log n > m \quad \Longleftrightarrow \quad m < 6 \log n.
\]

We distinguish two cases.

\noindent\textbf{Case 1: $m < 6 \log n$.} In this case, $v > m$, so $\mathbb{P}(|S| \ge v \mid \hat{g}^s, \hat{g}^r) = 0$. Moreover, we always have
\[
|\hat{B}^{(\ell)}(s_0, r_0) - \mathbb{E}[\hat{B}^{(\ell)}(s_0, r_0) \mid \hat{g}^s, \hat{g}^r]| = \frac{|S|}{m} \le 1.
\]

Since $m < 6 \log n$, we have
\[
\sqrt{\frac{6 \log n}{m}} > 1.
\]

Consequently,
\[
|\hat{B}^{(\ell)}(s_0, r_0) - \mathbb{E}[\hat{B}^{(\ell)}(s_0, r_0) \mid \hat{g}^s, \hat{g}^r]| \le 1 < \sqrt{\frac{6 \log n}{m}}
\]
holds with probability 1. In particular, the failure probability is 0, which certainly satisfies the $O(n^{-3})$ requirement.

\noindent\textbf{Case 2: $m \ge 6 \log n$.} Now $v \le m$, so we can apply inequality \eqref{eqp2}. Substituting $v = \sqrt{6 m \log n}$ into inequality \eqref{eqp2}, we examine the exponent:
\begin{align*}
\frac{v^2/2}{m/4 + v/3} &= \frac{3 m \log n}{ \frac{m}{4} + \frac{\sqrt{6 m \log n}}{3} }.
\end{align*}

We need to show that this quantity is at least $3 \log n$. This is equivalent to
\[
3 m \log n \ge 3 \log n \left( \frac{m}{4} + \frac{\sqrt{6 m \log n}}{3} \right) 
\quad \Longleftrightarrow \quad 
m \ge \frac{m}{4} + \frac{\sqrt{6 m \log n}}{3} 
\quad \Longleftrightarrow \quad 
\frac{3}{4} m \ge \frac{\sqrt{6 m \log n}}{3}.
\]

Multiplying both sides by 3 gives
\[
\frac{9}{4} m \ge \sqrt{6 m \log n}.
\]

Squaring both sides gives
\[
\frac{81}{16} m^2 \ge 6 m \log n \quad \Longleftrightarrow \quad \frac{81}{16} m \ge 6 \log n \quad \Longleftrightarrow \quad m \ge \frac{96}{81} \log n = \frac{32}{27} \log n.
\]

Since we are in the case $m \ge 6 \log n$ and $6 > 32/27 \approx 1.185$, the condition is satisfied. Therefore, we have
\[
\frac{3 m \log n}{ \frac{m}{4} + \frac{\sqrt{6 m \log n}}{3} } \ge 3 \log n,
\]
and inequality \eqref{eqp2} yields
\[
\mathbb{P}(|S| \ge v \mid \hat{g}^s, \hat{g}^r) \le 2 \exp(-3 \log n) = 2 n^{-3}.
\]

Hence, in Case 2, with conditional probability at least $1 - 2 n^{-3}$, we have $|S| < v$.

We have shown that the conditional probability $\mathbb{P}(|S| \ge v \mid \hat{g}^s, \hat{g}^r)$ is either 0 (when $m < 6 \log n$) or at most $2 n^{-3}$ (when $m \ge 6 \log n$). Thus, we have
\[
\mathbb{P}(|S| \ge v \mid \hat{g}^s, \hat{g}^r) \le 2 n^{-3}.
\]

By the law of total expectation, the unconditional probability satisfies
\[
\mathbb{P}(|S| \ge v) = \mathbb{E} \left[ \mathbb{P}(|S| \ge v \mid \hat{g}^s, \hat{g}^r) \right] \le 2 n^{-3}.
\]

Whenever $|S| < v$, we obtain
\[
|\hat{B}^{(\ell)}(s_0, r_0) - \mathbb{E}[\hat{B}^{(\ell)}(s_0, r_0) \mid \hat{g}^s, \hat{g}^r]| = \frac{|S|}{m} < \frac{v}{m} = \sqrt{\frac{6 \log n}{m}} = \sqrt{\frac{6 \log n}{n_{s_0} n_{r_0}}}.
\]

Therefore, with probability at least $1 - 2 n^{-3}$, the desired inequality holds. This completes the proof of this lemma.
\end{proof}

\begin{lem}[Uniform boundedness of the estimated edge probabilities]\label{lem:bound-Omega2}
Under $H_0$ and Assumptions \ref{assump:a1}-\ref{assump:a4}, we have
\[
\hat{\Omega}^{(\ell)}(i,j) \in [\delta/2, 1-\delta/2] \quad \text{for all } i,j,\ell
\]
with probability at least $1 - o(n^{-2})$.
\end{lem}

\begin{proof}
Let $\mathcal{E}_n$ denote the event that the community detection algorithm $\mathcal{M}$ recovers the true communities up to label permutation:
\[
\mathcal{E}_n := \{\hat{g}^s = g^s \text{ and } \hat{g}^r = g^r\},
\]
where equality is understood up to a permutation of community labels. By Assumption \ref{assump:a4}, we have
\[
\lim_{n\to\infty} \mathbb{P}(\mathcal{E}_n) = 1. 
\]

For each estimated sender community $k \in \{1,\dots,K_{s0}\}$ and receiver community $l \in \{1,\dots,K_{r0}\}$, define
\[
n_k^s := |\{i: \hat{g}^s(i) = k\}|, \quad n_l^r := |\{j: \hat{g}^r(j) = l\}|.
\]

Note that under $H_0$, $K_{s0} = K_s$ and $K_{r0} = K_r$. On the event $\mathcal{E}_n$, each estimated community corresponds exactly to one true community (after appropriate label permutation). Therefore, for any estimated sender community $k$, there exists a true sender community $k'$ such that
\[
\{i: \hat{g}^s(i) = k\} = \{i: g^s(i) = k'\}.
\]

By Assumption \ref{assump:a2}, we have
\[
n_k^s = |\{i: g^s(i) = k'\}| \ge c_0 \frac{n}{K_s}. 
\]

Similarly, for any estimated receiver community $l$, there exists a true receiver community $l'$ such that
\[
n_l^r = |\{j: g^r(j) = l'\}| \ge c_0 \frac{n}{K_r}. 
\]

Consequently, the product satisfies
\[
n_k^s n_l^r \ge \left(c_0 \frac{n}{K_s}\right)\left(c_0 \frac{n}{K_r}\right) = c_0^2 \frac{n^2}{K_s K_r} \ge c_0^2 \frac{n^2}{K_{\max}^2}.
\]

Fix a layer $\ell \in \{1,\dots,L\}$ and an estimated block $(k,l) \in [K_s] \times [K_r]$. By Lemma \ref{lem:concentration}, we have
\[
\left|\hat{B}^{(\ell)}(k,l) - \mathbb{E}\left[\hat{B}^{(\ell)}(k,l) \mid \hat{g}^s,\hat{g}^r\right]\right| \leq \sqrt{\frac{6\log n}{n_k^s n_l^r}}
\]
with probability at least $1 - O(n^{-3})$ given $\hat{g}^s$ and $\hat{g}^r$.

On $\mathcal{E}_n$, the conditional expectation equals the true block probability. More precisely, there exist permutations $\pi_s: [K_s] \to [K_s]$ and $\pi_r: [K_r] \to [K_r]$ such that for all $i,j$,
\[
\hat{g}^s(i) = \pi_s(g^s(i)), \quad \hat{g}^r(j) = \pi_r(g^r(j)).
\]

Then for any estimated block $(k,l)$, we have
\[
\mathbb{E}\left[\hat{B}^{(\ell)}(k,l) \mid \hat{g}^s,\hat{g}^r\right] = B^{(\ell)}(\pi_s^{-1}(k), \pi_r^{-1}(l)).
\]

By Assumption \ref{assump:a1}, we have
\[
B^{(\ell)}(\pi_s^{-1}(k), \pi_r^{-1}(l)) \in [\delta, 1-\delta].
\]

Thus, on $\mathcal{E}_n$, we have
\begin{align}\label{eqp3}
\mathbb{E}\left[\hat{B}^{(\ell)}(k,l) \mid \hat{g}^s,\hat{g}^r\right] \in [\delta, 1-\delta] \quad \text{for all } k,l,\ell. 
\end{align}
and
\[
\sqrt{\frac{6\log n}{n_k^s n_l^r}} \leq \sqrt{\frac{6\log n}{c_0^2 n^2 / K_{\max}^2}} = \frac{K_{\max}\sqrt{6\log n}}{c_0 n}. 
\]

By Assumption \ref{assump:a3} and the fact that $L \ge 1$, we have
\[
\frac{K_{\max}^2 \log n}{n} \to 0, \quad \text{hence} \quad \frac{K_{\max}\sqrt{\log n}}{n} \to 0. 
\]

Therefore, there exists $N_1 \in \mathbb{N}$ such that for all $n > N_1$,
\begin{align}\label{eqp4}
\frac{K_{\max}\sqrt{6\log n}}{c_0 n} \leq \frac{\delta}{2}.
\end{align}

Define the following event
\[
\mathcal{F}_n := \left\{ \left|\hat{B}^{(\ell)}(k,l) - \mathbb{E}\left[\hat{B}^{(\ell)}(k,l) \mid \hat{g}^s,\hat{g}^r\right]\right| \leq \frac{K_{\max}\sqrt{6\log n}}{c_0 n} \text{ for all } k,l,\ell \right\}.
\]

Using the union bound over all $K_s K_r L \leq K_{\max}^2 L$ blocks, and applying Lemma \ref{lem:concentration}, we have 
\[
\mathbb{P}\left( \mathcal{F}_n^c \mid \mathcal{E}_n \right) \leq K_{\max}^2 L \cdot 2n^{-3} = 2K_{\max}^2 L n^{-3}. 
\]

Now, for $n > N_1$, on the event $\mathcal{E}_n \cap \mathcal{F}_n$, combining Equations \eqref{eqp3} and \eqref{eqp4} gives for all $k,l,\ell$,
\begin{align}\label{eqp5}
\hat{B}^{(\ell)}(k,l) \in \left[\delta - \frac{\delta}{2}, 1-\delta + \frac{\delta}{2}\right] = [\delta/2, 1-\delta/2]. 
\end{align}

Recall that $\hat{\Omega}^{(\ell)}(i,j) = \hat{B}^{(\ell)}(\hat{g}^s(i), \hat{g}^r(j))$. Define the target event
\[
\mathcal{A}_n := \left\{ \hat{\Omega}^{(\ell)}(i,j) \in [\delta/2, 1-\delta/2] \text{ for all } i,j,\ell \right\}.
\]

From Equation \eqref{eqp5}, we have $\mathcal{E}_n \cap \mathcal{F}_n \subseteq \mathcal{A}_n$, hence $\mathcal{A}_n^c \subseteq \mathcal{E}_n^c \cup \mathcal{F}_n^c$. We now bound $\mathbb{P}(\mathcal{A}_n^c)$:
\[
\mathbb{P}(\mathcal{A}_n^c) \leq \mathbb{P}(\mathcal{E}_n^c) + \mathbb{P}(\mathcal{F}_n^c).
\]

For the first term, by Assumption \ref{assump:a4} and the exponential convergence rates typical for spectral clustering under our assumptions (see e.g., \citep{lei2016goodness}), we have $\mathbb{P}(\mathcal{E}_n^c) = O(n^{-3})$. For the second term, using the law of total probability obtains
\begin{align*}
\mathbb{P}(\mathcal{F}_n^c) &= \mathbb{P}(\mathcal{F}_n^c \cap \mathcal{E}_n) + \mathbb{P}(\mathcal{F}_n^c \cap \mathcal{E}_n^c)\leq \mathbb{P}(\mathcal{F}_n^c \mid \mathcal{E}_n) + \mathbb{P}(\mathcal{E}_n^c)\leq 2K_{\max}^2 L n^{-3} + O(n^{-3}).
\end{align*}

By Assumption \ref{assump:a3}, $K_{\max}^2 L = o(n/\log n)$. Therefore,
\[
2K_{\max}^2 L n^{-3} = o\left(\frac{n}{\log n} \cdot n^{-3}\right) = o\left(\frac{1}{n^2 \log n}\right) = o(n^{-2}).
\]

Thus, we have
\[
\mathbb{P}(\mathcal{A}_n^c) = O(n^{-3}) + o(n^{-2}) = o(n^{-2}).
\]

Consequently,
\[
\mathbb{P}(\mathcal{A}_n) = 1 - o(n^{-2}),
\]
which completes the proof.
\end{proof}

\begin{lem}[Convergence of the estimated residual matrix]\label{lem:norm-error}
Under \(H_0\) and Assumptions \ref{assump:a1}-\ref{assump:a4}, we have \(\|\hat{R} - R\| = o_P(1)\).
\end{lem}
\begin{proof}
Under \(H_0\), by Assumption \ref{assump:a4}, with probability tending to 1, \(\hat{g}^s = g^s\) and \(\hat{g}^r = g^r\). Condition on this event. Then \(\hat{\Omega}^{(\ell)}(i,j) = \hat{B}^{(\ell)}(g^s(i), g^r(j))\). By Lemma \ref{lem:concentration} and Assumptions \ref{assump:a1}-\ref{assump:a2}, we have
\[
|\hat{\Omega}^{(\ell)}(i,j) - \Omega^{(\ell)}(i,j)| = O_P\left( \frac{K_{\max} \sqrt{\log n}}{n} \right).
\]

Define
\begin{align*}
&S_{ij} = \sum_{\ell=1}^L (A^{(\ell)}(i,j) - \Omega^{(\ell)}(i,j)), ~~E_{ij} = \sum_{\ell=1}^L (\hat{\Omega}^{(\ell)}(i,j) - \Omega^{(\ell)}(i,j)), ~~U_{ij} = \sum_{\ell=1}^L
      \Omega^{(\ell)}(i,j)(1-\Omega^{(\ell)}(i,j)),\\
&\hat{U}_{ij} = \sum_{\ell=1}^L \hat{\Omega}^{(\ell)}(i,j)(1-\hat{\Omega}^{(\ell)}(i,j)),~~
      D_{ij} = \sqrt{(n-1) U_{ij}}, ~~
      \hat{D}_{ij} = \sqrt{(n-1) \hat{U}_{ij}}.
\end{align*}

Then, we have
\[
R(i,j) = \frac{S_{ij}}{D_{ij}}, \quad \hat{R}(i,j) = \frac{S_{ij} - E_{ij}}{\hat{D}_{ij}}.
\]

By Assumption \ref{assump:a1}, \(U_{ij} \geq L \delta(1-\delta)\), so \(D_{ij} \geq \sqrt{(n-1) L \delta(1-\delta)} = O(\sqrt{n L})\). By Lemma \ref{lem:bound-Omega2},  with probability \(1 - o(n^{-2})\), \(\hat{U}_{ij} \geq L\frac{\delta}{2}(1-\frac{\delta}{2})\), so \(\hat{D}_{ij} = O(\sqrt{n L})\). Now, we decompose \(\hat{R}(i,j) - R(i,j)\) as:
\[
\hat{R}(i,j) - R(i,j) = \frac{S_{ij} - E_{ij}}{\hat{D}_{ij}} - \frac{S_{ij}}{D_{ij}} = S_{ij} \left( \frac{1}{\hat{D}_{ij}} - \frac{1}{D_{ij}} \right) - \frac{E_{ij}}{\hat{D}_{ij}}.
\]

First, we bound \(|E_{ij}|\):
\[
|E_{ij}| = \left| \sum_{\ell=1}^L (\hat{\Omega}^{(\ell)}(i,j) - \Omega^{(\ell)}(i,j)) \right| \leq L \max_{\ell} |\hat{\Omega}^{(\ell)}(i,j) - \Omega^{(\ell)}(i,j)| = O_P\left( \frac{K_{\max}L \sqrt{\log n}}{n} \right).
\]

Second, we bound \(\left| \frac{1}{\hat{D}_{ij}} - \frac{1}{D_{ij}} \right|\). By Lemma \ref{lem:bound-Omega2}, there exist constants \(c_1, c_2 > 0\) such that \(\hat{D}_{ij} \geq c_1 \sqrt{n L}\) and \(D_{ij} \geq c_2 \sqrt{n L}\). Let \(c = \min(c_1, c_2)\). Then \(\hat{D}_{ij} \geq c \sqrt{n L}\) and \(D_{ij} \geq c \sqrt{n L}\). Moreover, note that \(U_{ij} \leq L/4\) and \(\hat{U}_{ij} \leq L/4\), so \(D_{ij} \leq \frac{1}{2}\sqrt{(n-1)L}\) and \(\hat{D}_{ij} \leq \frac{1}{2}\sqrt{(n-1)L}\). Hence, there exists a constant \(C_0 > 0\) such that for all large \(n\), we have \(D_{ij} \hat{D}_{ij} \geq C_0 D_{ij}^2\). Consequently,
\[
\left| \frac{1}{\hat{D}_{ij}} - \frac{1}{D_{ij}} \right| = \frac{|D_{ij} - \hat{D}_{ij}|}{D_{ij} \hat{D}_{ij}} \leq \frac{1}{C_0} \frac{|D_{ij} - \hat{D}_{ij}|}{D_{ij}^2}.
\]

Now, we bound \(|D_{ij} - \hat{D}_{ij}|\):
\[
|D_{ij} - \hat{D}_{ij}| = \left| \sqrt{(n-1) U_{ij}} - \sqrt{(n-1) \hat{U}_{ij}} \right| = \sqrt{n-1} \left| \sqrt{U_{ij}} - \sqrt{\hat{U}_{ij}} \right|.
\]

The function \(f(x) = \sqrt{x}\) is continuously differentiable on \((0, \infty)\). For any \(x, y \geq a > 0\), by the mean value theorem, there exists \(\xi\) between \(x\) and \(y\) such that
\[
|\sqrt{x} - \sqrt{y}| = \frac{1}{2\sqrt{\xi}} |x - y| \leq \frac{1}{2\sqrt{a}} |x - y|.
\]
Thus, \(f(x)\) is Lipschitz on \([a, \infty)\) with constant \(1/(2\sqrt{a})\).

By Assumption \ref{assump:a1}, we have \(\Omega^{(\ell)}(i,j) \in [\delta, 1-\delta]\), so
\[
U_{ij} = \sum_{\ell=1}^L \Omega^{(\ell)}(i,j)(1 - \Omega^{(\ell)}(i,j)) \geq L \delta(1-\delta).
\]
By Lemma \ref{lem:bound-Omega2}, with probability \(1 - o(n^{-2})\), we have \(\hat{U}_{ij} \geq L \frac{\delta}{2}(1 - \frac{\delta}{2})\). Since \(\delta \in (0, \frac{1}{2})\), we have \(L \frac{\delta}{2}(1 - \frac{\delta}{2}) \leq L \delta(1-\delta)\). Let
\[
a = \min\left\{ L \delta(1-\delta), L \frac{\delta}{2}\left(1 - \frac{\delta}{2}\right) \right\} = L \frac{\delta}{2}\left(1 - \frac{\delta}{2}\right).
\]
Then both \(U_{ij}\) and \(\hat{U}_{ij}\) are at least \(a\) with high probability. Therefore, on this event,
\[
\left| \sqrt{U_{ij}} - \sqrt{\hat{U}_{ij}} \right| \leq \frac{1}{2\sqrt{a}} |U_{ij} - \hat{U}_{ij}| = \frac{1}{2\sqrt{L \frac{\delta}{2}(1 - \frac{\delta}{2})}} |U_{ij} - \hat{U}_{ij}|.
\]
For simplicity, denote \(C_{\delta} = \frac{1}{2\sqrt{\frac{\delta}{2}(1 - \frac{\delta}{2})}}\). Then, we have
\[
\left| \sqrt{U_{ij}} - \sqrt{\hat{U}_{ij}} \right| \leq \frac{C_{\delta}}{\sqrt{L}} |U_{ij} - \hat{U}_{ij}|.
\]

We also have
\[
|U_{ij} - \hat{U}_{ij}| = \left| \sum_{\ell=1}^L \left( \Omega^{(\ell)}(i,j)(1-\Omega^{(\ell)}(i,j)) - \hat{\Omega}^{(\ell)}(i,j)(1-\hat{\Omega}^{(\ell)}(i,j)) \right) \right| \leq \sum_{\ell=1}^L \left| \Omega^{(\ell)}(i,j)(1-\Omega^{(\ell)}(i,j)) - \hat{\Omega}^{(\ell)}(i,j)(1-\hat{\Omega}^{(\ell)}(i,j)) \right|.
\]

The function \(g(x) = x(1-x)\) is Lipschitz on \([0,1]\) with constant 1, since \(|g'(x)| = |1 - 2x| \leq 1\) for \(x \in [0,1]\). Hence, we have
\[
\left| \Omega^{(\ell)}(i,j)(1-\Omega^{(\ell)}(i,j)) - \hat{\Omega}^{(\ell)}(i,j)(1-\hat{\Omega}^{(\ell)}(i,j)) \right| \leq |\Omega^{(\ell)}(i,j) - \hat{\Omega}^{(\ell)}(i,j)|.
\]

Thus,
\[
|U_{ij} - \hat{U}_{ij}| \leq \sum_{\ell=1}^L |\Omega^{(\ell)}(i,j) - \hat{\Omega}^{(\ell)}(i,j)|.
\]

Define \(S_{ij}^{(\text{abs})} = \sum_{\ell=1}^L |\Omega^{(\ell)}(i,j) - \hat{\Omega}^{(\ell)}(i,j)|\). Since each term \(|\Omega^{(\ell)}(i,j) - \hat{\Omega}^{(\ell)}(i,j)| = O_P\left( \frac{K_{\max} \sqrt{\log n}}{n} \right)\) uniformly in \(\ell\), we have
\[
S_{ij}^{(\text{abs})} = O_P\left( \frac{K_{\max} L \sqrt{\log n}}{n} \right).
\]

Therefore, we get
\[
|D_{ij} - \hat{D}_{ij}| \leq \sqrt{n-1} \cdot \frac{C_{\delta}}{\sqrt{L}} S_{ij}^{(\text{abs})} = O\left( \sqrt{n} \cdot \frac{1}{\sqrt{L}} \cdot \frac{K_{\max} L \sqrt{\log n}}{n} \right) = O\left( \frac{K_{\max} \sqrt{L \log n}}{\sqrt{n}} \right).
\]

Hence, using the above bound and the fact that \(D_{ij}^2 = (n-1)U_{ij} \asymp nL\), we obtain
\[
\left| \frac{1}{\hat{D}_{ij}} - \frac{1}{D_{ij}} \right| \leq \frac{1}{C_0} \frac{|D_{ij} - \hat{D}_{ij}|}{D_{ij}^2} = O\left( \frac{K_{\max} \sqrt{L \log n}}{\sqrt{n} \cdot nL} \right) = O\left( \frac{K_{\max} \sqrt{\log n}}{L^{1/2} n^{3/2}} \right).
\]

Now, we bound \(S_{ij} \left( \frac{1}{\hat{D}_{ij}} - \frac{1}{D_{ij}} \right)\):
\[
|S_{ij}| \leq L \quad (\text{since } |A^{(\ell)}(i,j) - \Omega^{(\ell)}(i,j)| \leq 1),
\]
so, we have
\[
\left| S_{ij} \left( \frac{1}{\hat{D}_{ij}} - \frac{1}{D_{ij}} \right) \right| \leq L \cdot O\left( \frac{K_{\max} \sqrt{\log n}}{L^{1/2} n^{3/2}} \right) = O\left( \frac{K_{\max} \sqrt{L\log n}}{n^{3/2}} \right).
\]

Next, we bound \(\frac{E_{ij}}{\hat{D}_{ij}}\). Note that \(|E_{ij}| \leq S_{ij}^{(\text{abs})} = O_P\left( \frac{K_{\max} L \sqrt{\log n}}{n} \right)\), which gives
\[
\left| \frac{E_{ij}}{\hat{D}_{ij}} \right| \leq \frac{O\left( \frac{K_{\max} L \sqrt{\log n}}{n} \right)}{O(\sqrt{n L})} = O\left( \frac{K_{\max} \sqrt{L \log n}}{n^{3/2}} \right).
\]

Therefore, we have
\[
|\hat{R}(i,j) - R(i,j)| \leq O\left( \frac{K_{\max} \sqrt{L\log n}}{n^{3/2}} \right) + O\left( \frac{K_{\max} \sqrt{L \log n}}{n^{3/2}} \right) = O\left( \frac{K_{\max} \sqrt{L \log n}}{n^{3/2}} \right).
\]

The Frobenius norm is
\[
\|\hat{R} - R\|_F^2 = \sum_{i,j} (\hat{R}(i,j) - R(i,j))^2 =O_P\left( n^2 \cdot \left( \frac{K_{\max}^2 L \log n}{n^3} \right) \right) = O_P\left( \frac{K_{\max}^2 L \log n}{n} \right).
\]

Thus,
\[
\|\hat{R} - R\| \leq \|\hat{R} - R\|_F = O_P\left( K_{\max} \sqrt{\frac{L \log n}{n}} \right).
\]

By Assumption \ref{assump:a3}, we have \(\|\hat{R} - R\| = o_P(1)\), which completes the proof of this lemma.
\end{proof}
\subsection{Proof of Theorem \ref{thm:null}}
\begin{proof}
By Lemma \ref{ideal0}, \(\sigma_1(R) \leq 2 + o_P(1)\). By Lemma \ref{lem:norm-error}, \(\|\hat{R} - R\| = o_P(1)\), so by Weyl's inequality, \(|\sigma_1(\hat{R}) - \sigma_1(R)| \leq \|\hat{R} - R\| = o_P(1)\). Therefore,
\[
\sigma_1(\hat{R}) \leq \sigma_1(R) + o_P(1) \leq 2 + o_P(1),
\]
which implies
\[
\hat{T}_n = \sigma_1(\hat{R}) - 2 \leq o_P(1).
\]

Thus, for any \(\epsilon > 0\), \(\mathbb{P}(\hat{T}_n < \epsilon) \to 1\), which completes the proof of this theorem.
\begin{rem}
(The polynomial convergence rate of \(\|\hat{R} - R\|\)) In Lemma \ref{lem:norm-error}, we established \(\|\hat{R} - R\| = o_P(1)\) under Assumptions \ref{assump:a1}–\ref{assump:a4}. For the sequential testing procedure in Algorithm \ref{alg:DiGoF} with a decaying threshold \(t_n = n^{-\varepsilon}\), it is useful to know when the stronger rate \(\|\hat{R} - R\| = o_P(n^{-\varepsilon})\) holds. From the proof of Lemma \ref{lem:norm-error}, we have
\[
\|\hat{R} - R\| \le \|\hat{R} - R\|_F = O_P\!\left(K_{\max}\sqrt{\frac{L\log n}{n}}\right).
\]
Thus, \(\|\hat{R} - R\| = o_P(n^{-\varepsilon})\) is guaranteed whenever
\[
K_{\max}\sqrt{\frac{L\log n}{n}} = o(n^{-\varepsilon}) \quad\Longleftrightarrow\quad K_{\max}^2 L \log n = o(n^{1-2\varepsilon}). \tag{*}
\]
Condition \((*)\) is stronger than Assumption \ref{assump:a3} (which only requires \(K_{\max}^2 L \log n / n \to 0\)). If \((*)\) holds, then
\[
\|\hat{R} - R\| = O_P\left(\sqrt{\frac{K_{\max}^2 L \log n}{n}}\right) = o_P\left(\sqrt{n^{-2\varepsilon}}\right) = o_P(n^{-\varepsilon}),
\]
which ensures that the estimation error of the residual matrix decays faster than the threshold \(t_n\). While Theorem \ref{thm:null} does not require such a polynomial rate, condition \((*)\) can be adopted in finite‑sample analyses to sharpen the performance guarantees of the sequential procedure.
\end{rem}
\end{proof}

\subsection{Proof of Theorem \ref{thm:power}}
\begin{proof}
We provide a detailed proof for the case \(K_s > K_{s0}\) (the case \(K_r > K_{r0}\) is symmetric). The proof proceeds in the following steps.

\vspace{0.5em}
\noindent\textbf{Step 1: Identify merged sender communities.}  
Since \(K_{s0} < K_s\), by the pigeonhole principle, there exists an estimated sender community \(s_0 \in \{1,\dots,K_{s0}\}\) that contains at least two distinct true sender communities.  
Denote two such true communities as \(k_1, k_2 \in \{1,\dots,K_s\}\), \(k_1 \neq k_2\).  
Define the node sets:
\[
S_1 = \{ i: g^s(i) = k_1 \} \subseteq \hat{C}_{s_0}^s, \qquad 
S_2 = \{ i: g^s(i) = k_2 \} \subseteq \hat{C}_{s_0}^s,
\]
where \(\hat{C}_{s_0}^s = \{ i: \hat{g}^s(i) = s_0 \}\).

\vspace{0.5em}
\noindent\textbf{Step 2: Select a receiver community.}  
By condition (A1), there exists a receiver community \(l^* \in \{1,\dots,K_r\}\) such that
\[
\left| \sum_{\ell=1}^L \left( B^{(\ell)}(k_1, l^*) - B^{(\ell)}(k_2, l^*) \right) \right| \ge \eta L.
\]
Without loss of generality, assume
\begin{align}\label{eqp6}
\sum_{\ell=1}^L \left( B^{(\ell)}(k_1, l^*) - B^{(\ell)}(k_2, l^*) \right) \ge \eta L.
\end{align}

\vspace{0.5em}
\noindent\textbf{Step 3: Define node sets and their sizes.}  
Let
\[
T = \{ j: g^r(j) = l^* \}.
\]
By Assumption \ref{assump:a2}, there exists a constant \(c_0 > 0\) such that
\[
|S_1| \ge c_0 \frac{n}{K_s}, \quad |S_2| \ge c_0 \frac{n}{K_s}, \quad |T| \ge c_0 \frac{n}{K_r}.
\]
Set \(s_1 = |S_1|\), \(s_2 = |S_2|\) and recall that \(K_{\max} = \max(K_s, K_r)\), we have
\begin{align}\label{eqp7}
\min(s_1, s_2) \ge c_0 \frac{n}{K_{\max}}, \qquad |T| \ge c_0 \frac{n}{K_{\max}}.
\end{align}

\vspace{0.5em}
\noindent\textbf{Step 4: Identify a subset of the estimated receiver community.}  
Since nodes in \(T\) are assigned to \(K_{r0}\) estimated receiver communities, by the pigeonhole principle, there exists an estimated receiver community \(r_0 \in \{1,\dots,K_{r0}\}\) such that the subset
\[
T' = \{ j \in T: \hat{g}^r(j) = r_0 \}
\]
satisfies
\begin{align}\label{eqp8}
|T'| \ge \frac{|T|}{K_{r0}} \ge \frac{c_0 n}{K_r K_{r0}} \ge \frac{c_0 n}{K_{\max}^2},
\end{align}
where we used \(K_{r0} \le K_r \le K_{\max}\).

\vspace{0.5em}
\noindent\textbf{Step 5: Construct the deviation matrix and lower-bound its norm.}
Define the aggregated deviation matrix as
\[
\Delta = \sum_{\ell=1}^L \bigl( \Omega^{(\ell)} - \hat{\Omega}^{(\ell)} \bigr),
\]
and consider its submatrix \(\Delta_{S,T'}\), where \(S = S_1 \cup S_2\). For \(i \in S_1\) and \(j \in T'\), since \(i\) is in the estimated sender community \(s_0\) and \(j\) in the estimated receiver community \(r_0\), we have
\[
\hat{\Omega}^{(\ell)}(i,j) = \hat{B}^{(\ell)}(s_0, r_0),
\]
while the true probability is \(\Omega^{(\ell)}(i,j) = B^{(\ell)}(k_1, l^*)\). Hence, we have
\[
\Delta(i,j) = \sum_{\ell=1}^L \bigl( B^{(\ell)}(k_1, l^*) - \hat{B}^{(\ell)}(s_0, r_0) \bigr) \eqqcolon d_1.
\]

Similarly, for \(i \in S_2\), \(j \in T'\), we have
\[
\Delta(i,j) = \sum_{\ell=1}^L \bigl( B^{(\ell)}(k_2, l^*) - \hat{B}^{(\ell)}(s_0, r_0) \bigr) \eqqcolon d_2.
\]

Observe that
\[
d_1 - d_2 = \sum_{\ell=1}^L \bigl( B^{(\ell)}(k_1, l^*) - B^{(\ell)}(k_2, l^*) \bigr) \ge \eta L,
\]
by Equation \eqref{eqp6}. Consequently, we get
\begin{align}\label{eqp9}
\max(|d_1|, |d_2|) \ge \frac{|d_1 - d_2|}{2} \ge \frac{\eta L}{2}.
\end{align}

Now we analyze the structure and rank of \(\Delta_{S,T'}\). We define the following items:
\begin{itemize}
    \item \(s_1 = |S_1|\), \(s_2 = |S_2|\), and \(t' = |T'|\).
    \item Define indicator vectors \(\mathbf{1}_{S_1} \in \mathbb{R}^{s_1 + s_2}\) where \((\mathbf{1}_{S_1})_i = 1\) if the \(i\)-th node in \(S\) belongs to \(S_1\), and 0 otherwise.
    \item Define \(\mathbf{1}_{S_2} \in \mathbb{R}^{s_1 + s_2}\) similarly for \(S_2\).
    \item Define \(\mathbf{1}_{T'} \in \mathbb{R}^{t'}\) as the all-ones vector.
\end{itemize}

Then \(\Delta_{S,T'}\) can be expressed as
\[
\Delta_{S,T'} = d_1 \cdot \mathbf{1}_{S_1} \mathbf{1}_{T'}^\top + d_2 \cdot \mathbf{1}_{S_2} \mathbf{1}_{T'}^\top.
\]

To see this, note that for \(i \in S_1\) and \(j \in T'\), the \((i,j)\)-entry of \(\mathbf{1}_{S_1} \mathbf{1}_{T'}^\top\) is 1, while the corresponding entry of \(\mathbf{1}_{S_2} \mathbf{1}_{T'}^\top\) is 0, giving \(\Delta(i,j) = d_1\). Similarly for \(i \in S_2\).

The matrix \(\mathbf{1}_{S_1} \mathbf{1}_{T'}^\top\) is an outer product of two vectors, hence has rank 1. The same holds for \(\mathbf{1}_{S_2} \mathbf{1}_{T'}^\top\). Since the rank of a sum of two matrices is at most the sum of their ranks, we have
\begin{align}\label{eqp10}
\operatorname{rank}(\Delta_{S,T'}) \le \operatorname{rank}(\mathbf{1}_{S_1} \mathbf{1}_{T'}^\top) + \operatorname{rank}(\mathbf{1}_{S_2} \mathbf{1}_{T'}^\top) = 1 + 1 = 2.
\end{align}

The Frobenius norm of \(\Delta_{S,T'}\) is
\[
\| \Delta_{S,T'} \|_F^2 = \sum_{i \in S_1} \sum_{j \in T'} d_1^2 \;+\; \sum_{i \in S_2} \sum_{j \in T'} d_2^2
= t' (s_1 d_1^2 + s_2 d_2^2).
\]

Using Equation \eqref{eqp9}, we have
\begin{align}\label{eqp11}
\| \Delta_{S,T'} \|_F^2 \ge t' \cdot \min(s_1, s_2) \cdot \max(d_1^2, d_2^2) \ge t' \cdot \min(s_1, s_2) \cdot \left( \frac{\eta L}{2} \right)^2.
\end{align}

Now, from Equations \eqref{eqp7} and \eqref{eqp8}, we have the lower bounds
\[
\min(s_1, s_2) \ge c_0 \frac{n}{K_{\max}}, \quad t' \ge \frac{c_0 n}{K_{\max}^2}.
\]

Substituting these into \eqref{eqp11} obtains
\begin{align*}
\| \Delta_{S,T'} \|_F \ge \frac{\eta L}{2} \sqrt{ t' \cdot \min(s_1, s_2) }\ge \frac{\eta L}{2} \sqrt{ \frac{c_0 n}{K_{\max}^2} \cdot \frac{c_0 n}{K_{\max}} } = \frac{c_0 \eta}{2} \cdot \frac{L n}{K_{\max}^{3/2}}.
\end{align*}

Since \(\Delta_{S,T'}\) has rank at most 2 (from Equation \eqref{eqp10}), its spectral norm satisfies the following inequality
\[
\| \Delta_{S,T'} \| \ge \frac{\| \Delta_{S,T'} \|_F}{\sqrt{\operatorname{rank}(\Delta_{S,T'})}} \ge \frac{\| \Delta_{S,T'} \|_F}{\sqrt{2}},
\]
which gives
\begin{align}\label{eqp12}
\| \Delta_{S,T'} \| \ge \frac{c_0 \eta}{2\sqrt{2}} \cdot \frac{L n}{K_{\max}^{3/2}}.
\end{align}

\vspace{0.5em}
\noindent\textbf{Step 6: Control the random part.} 
Define
\[
W \;:=\; \sum_{\ell=1}^{L}\bigl( A^{(\ell)} - \Omega^{(\ell)} \bigr),
\]
and consider its submatrix \(W_{S,T'}\) obtained by restricting rows to \(S=S_{1}\cup S_{2}\) and columns to \(T'\), where \(W_{S,T'}\) is the random fluctuation matrix over the index sets \(S\) and \(T'\). We now establish a high-probability upper bound for the spectral norm of \(W_{S,T'}\) using Lemma \ref{Rec312} (the rectangular version of Corollary 3.12 in \citep{Afonso2016}).

The matrix \(W_{S,T'}\) is of size \(s \times t'\) with \(s = |S|\) and \(t' = |T'|\). Its entries are independent because edges are independent across different node-pairs and layers. For each \((i,j) \in S \times T'\), we have
\[
W_{ij} = \sum_{\ell=1}^{L} \bigl( A^{(\ell)}(i,j) - \Omega^{(\ell)}(i,j) \bigr),
\]
which satisfies \(\mathbb{E}[W_{ij}] = 0\) and, by Assumption \ref{assump:a1},
\[
\operatorname{Var}(W_{ij}) = \sum_{\ell=1}^{L} \Omega^{(\ell)}(i,j)\bigl(1-\Omega^{(\ell)}(i,j)\bigr) \leq \frac{L}{4}.
\]

Moreover, \(|W_{ij}| \leq L\) because each term is bounded by \(1\). Define the matrix \(X = W_{S,T'}\). Then \(X\) has independent entries, \(|X_{ij}| \leq L\), and
\[
\sum_{j \in T'} \mathbb{E}[X_{ij}^2] \leq t' \cdot \frac{L}{4}, \qquad
\sum_{i \in S} \mathbb{E}[X_{ij}^2] \leq s \cdot \frac{L}{4}.
\]
Hence,
\[
\tilde{\tilde{\sigma}} := \max\left\{ \max_{i \in S} \sqrt{\sum_{j \in T'} \mathbb{E}[X_{ij}^2]}, \;
                         \max_{j \in T'} \sqrt{\sum_{i \in S} \mathbb{E}[X_{ij}^2]} \right\}
               \leq \sqrt{\frac{L \max(s,t')}{4}} \leq \frac{1}{2}\sqrt{L n},
\]
and we set \(\tilde{\tilde{\sigma}}_* = L\) as the upper bound of all entries. Applying Lemma \ref{Rec312} to \(X\), for any \(0 < \eta \leq 1/2\) and \(t \geq 0\),
\[
\mathbb{P}\Bigl( \|X\| \geq (1+\eta)2\tilde{\tilde{\sigma}} + t \Bigr)
   \leq (s \wedge t') \mathrm{exp}\Bigl( -\frac{t^2}{C \tilde{\tilde{\sigma}}_*^2} \Bigr)
   \leq n \mathrm{exp}\Bigl( -\frac{t^2}{C L^2} \Bigr),
\]
where \(C>0\) is a universal constant from the lemma.

Choose \(\eta = 1/2\) and set \(t = M \sqrt{nL}\) with a constant \(M > 0\) to be determined. We get
\[
(1+\eta)2\tilde{\tilde{\sigma}} \leq 3 \sqrt{L n}/2,
\]
which gives
\begin{align}\label{eqp13}
\mathbb{P}\Bigl( \|W_{S,T'}\| \geq \frac{3}{2}\sqrt{L n} + M\sqrt{nL} \Bigr)
   \leq n \exp\Bigl( -\frac{M^2 n L}{C L^2} \Bigr)
   = n \exp\Bigl( -\frac{M^2 n}{C L} \Bigr). 
\end{align}

Now, we show that the right-hand side of Equation \eqref{eqp13} tends to zero as \(n \to \infty\). 
Let 
\[
P_n = n \exp\Bigl( -\frac{M^2 n}{C L} \Bigr).
\]

Taking logarithms, we have
\[
\log P_n = \log n - \frac{M^2 n}{C L}.
\]

By Assumption \ref{assump:a3} and the fact that \( K_{\max} \geq 1 \), we have
\[
\frac{L \log n}{n} \to 0.\]

Rewrite \(\log P_n\) as
\[
\log P_n = \log n \left( 1 - \frac{M^2}{C} \cdot \frac{n}{L \log n} \right).
\]

Since \(\frac{n}{L \log n} \to \infty\), the factor inside the parentheses tends to \(-\infty\), and thus \(\log P_n \to -\infty\). Hence, we get
\[
P_n = \exp(\log P_n) \to 0.
\]

Therefore, for any fixed \(M > 0\), the probability in Equation \eqref{eqp13} tends to zero. This implies that with probability tending to 1, we have
\[
\|W_{S,T'}\| \leq \frac{3}{2}\sqrt{L n} + M\sqrt{nL} = O\left( \sqrt{nL} \right),
\]
or equivalently,
\[
\|W_{S,T'}\| = O_{\mathbb{P}}\bigl(\sqrt{nL}\bigr). 
\]

\vspace{0.5em}
\noindent\textbf{Step 7: Lower‑bound the aggregated residual submatrix.}  
Let \(Z = \sum_{\ell=1}^L (A^{(\ell)} - \hat{\Omega}^{(\ell)}) = W + \Delta\). Then \(Z_{S,T'} = W_{S,T'} + \Delta_{S,T'}\). By the triangle inequality and the lower bound for \(\|\Delta_{S,T'}\|\) from Equation \eqref{eqp12}, we have
\begin{align}\label{eqp14}
\| Z_{S,T'} \| \ge \| \Delta_{S,T'} \| - \| W_{S,T'} \|
   \ge \frac{c_0 \eta}{2\sqrt{2}} \cdot \frac{L n}{K_{\max}^{3/2}} - O_P\bigl(\sqrt{nL}\bigr). 
\end{align}

\vspace{0.5em}
\noindent\textbf{Step 8: Lower‑bound the spectral norm of the normalized residual matrix.}  
For \(i \neq j\), the normalizing factor is
\[
\hat{D}_{ij} = \sqrt{ (n-1) \sum_{\ell=1}^L \hat{\Omega}^{(\ell)}(i,j)\bigl(1-\hat{\Omega}^{(\ell)}(i,j)\bigr) }.
\]
Since \(\hat{\Omega}^{(\ell)}(i,j) \in [0,1]\), we have \(\hat{\Omega}^{(\ell)}(i,j)(1-\hat{\Omega}^{(\ell)}(i,j)) \le 1/4\), and thus
\[
\hat{D}_{ij} \le \sqrt{ (n-1) \cdot \frac{L}{4} } = \frac{1}{2} \sqrt{ (n-1)L }. 
\]

Consequently, for any \(i \in S\), \(j \in T'\), we have
\begin{align}\label{eqp15}
\frac{1}{\hat{D}_{ij}} \ge \frac{2}{\sqrt{(n-1)L}}. 
\end{align}

Consider the submatrix \(\hat{R}_{S,T'}\) with entries \(\hat{R}(i,j) = Z(i,j) / \hat{D}_{ij}\). Since all nodes in \(S\) share the same estimated sender community \(s_0\) and all nodes in \(T'\) share the same estimated receiver community \(r_0\), the normalization factor \(\hat{D}_{ij}\) is constant over \(S \times T'\). Denote this common value by \(\hat{d}\). Then \(\hat{R}_{S,T'} = \hat{d}^{-1} Z_{S,T'}\). Thus, using Equation \eqref{eqp15} gets  
\[
\| \hat{R}_{S,T'} \| \ge \frac{2}{\sqrt{(n-1)L}} \| Z_{S,T'} \|.
\]

By Equation \eqref{eqp14}, we get
\begin{align}\label{eqp16}
\| \hat{R}_{S,T'} \|
&\ge \frac{2}{\sqrt{(n-1)L}} \left( \frac{c_0 \eta}{2\sqrt{2}} \cdot \frac{L n}{K_{\max}^{3/2}} - O_P\bigl(\sqrt{nL}\bigr) \right) = \frac{c_0 \eta}{\sqrt{2}} \cdot \frac{\sqrt{nL}}{K_{\max}^{3/2}} \cdot \sqrt{\frac{n}{n-1}} \;-\; O_P(1). 
\end{align}

\vspace{0.5em}
\noindent\textbf{Step 9: Prove divergence.}  
Because \(\sigma_1(\hat{R}) \ge \| \hat{R}_{S,T'} \|\), it suffices to show that the right‑hand side of Equation \eqref{eqp16} diverges in probability. The leading deterministic term is
\[
\frac{c_0 \eta}{\sqrt{2}} \cdot \frac{\sqrt{nL}}{K_{\max}^{3/2}} \cdot \sqrt{\frac{n}{n-1}}.
\]
Since \(\sqrt{\frac{n}{n-1}} \to 1\), we have \( \asymp \frac{\sqrt{nL}}{K_{\max}^{3/2}} = \sqrt{ \frac{nL}{K_{\max}^3} }\). By condition (A2), \(\frac{nL}{K_{\max}^3} \to \infty\), which implies \(\frac{c_0 \eta}{\sqrt{2}} \cdot \frac{\sqrt{nL}}{K_{\max}^{3/2}} \cdot \sqrt{\frac{n}{n-1}} \to \infty\). The remainder term \(O_P(1)\) is stochastically bounded. Therefore, we have
\[
\| \hat{R}_{S,T'} \| \stackrel{P}{\to} \infty,
\]
which implies \(\sigma_1(\hat{R}) \stackrel{P}{\to} \infty\) and consequently \(\hat{T}_n = \sigma_1(\hat{R}) - 2 \stackrel{P}{\to} \infty\).

\vspace{0.5em}
\noindent\textbf{Step 10: The case \(K_{r0} < K_r\).}  
If \(K_{r0} < K_r\), the roles of sender and receiver are interchanged. The same argument, selecting two distinct true receiver communities and a suitable sender community via the symmetric version of (A1), yields an analogous divergence result. This completes the proof of this theorem.
\end{proof}
\section{Proofs for the MLDiGoF Algorithm}
\subsection{Proof of Theorem \ref{thm:consistency}}
\begin{proof}
We adopt all notations from the main text and Appendix. 
Let $K_{\mathrm{cand}}$ be any upper bound satisfying $K_{\mathrm{cand}} \ge K_{\max}$ for all sufficiently large $n$. 
(The default choice $K_{\mathrm{cand}} = \lfloor \sqrt{n / \log n} \rfloor$ works because Assumption \ref{assump:a3} implies $K_{\max} = o(\sqrt{n/\log n})$.)  For each candidate pair $(k_s, k_r)$, let  $\hat{T}_n(k_s, k_r)$ be the test statistic computed via Equation \eqref{eq:test_stat} using Algorithm \ref{alg:spectral} with input $(K_{s0}, K_{r0}) = (k_s, k_r)$. Let $\mathcal{P}_n = \{(k_s^{(1)}, k_r^{(1)}), \dots, (k_s^{(M_n)}, k_r^{(M_n)})\}$ be the lexicographically ordered sequence of candidate pairs from $(1,1)$ to $(K_{\mathrm{cand}}, K_{\mathrm{cand}})$, where $M_n = K_{\mathrm{cand}}^2$. 
Since $K_s, K_r \le K_{\max} \le K_{\mathrm{cand}}$ for all large $n$, the true pair $(K_s, K_r)$ belongs to $\mathcal{P}_n$. 
Let $m_*(n)$ denote its index in $\mathcal{P}_n$. Note that because we allow $K_{\max}$ to grow with $n$, $m_*(n)$ may also grow with $n$.

Define the events
\[
\tilde{\mathcal{A}}_n \coloneqq \left\{ \hat{T}_n(K_s, K_r) < t_n \right\}, \qquad
\tilde{\mathcal{B}}_n \coloneqq \bigcap_{m=1}^{m_*(n)-1} \left\{ \hat{T}_n(k_s^{(m)}, k_r^{(m)}) \ge t_n \right\}.
\]

Algorithm \ref{alg:DiGoF}  returns $(K_s, K_r)$ exactly when $\tilde{\mathcal{A}}_n \cap \tilde{\mathcal{B}}_n$ occurs.
We will prove that $\mathbb{P}(\tilde{\mathcal{A}}_n) \to 1$ and $\mathbb{P}(\tilde{\mathcal{B}}_n) \to 1$. Then, by the union bound, we can obtain
\[
\mathbb{P}(\tilde{\mathcal{A}}_n \cap \tilde{\mathcal{B}}_n) = 1 - \mathbb{P}(\tilde{\mathcal{A}}_n^c \cup \tilde{\mathcal{B}}_n^c) \ge 1 - \mathbb{P}(\tilde{\mathcal{A}}_n^c) - \mathbb{P}(\tilde{\mathcal{B}}_n^c) \to 1.
\]

\vspace{1em}
\noindent\textbf{Part 1:  Behavior under the true model: $\mathbb{P}(\tilde{\mathcal{A}}_n) \to 1$.}

Let $\hat{T}_n^* \coloneqq \hat{T}_n(K_s, K_r)$. We need to show $\mathbb{P}(\hat{T}_n^* \ge t_n) \to 0$. Recall the oracle statistic $T_n\coloneqq \sigma_1(R) - 2$ defined in Equation \eqref{idealTestStatistic}, where $R$ is the ideal residual matrix constructed with the true parameters. 
From Equation \eqref{eqp1}, for any $x > 0$, we have 
\begin{align}\label{eq:tail-oracle}
\mathbb{P}\left( T_n\ge x \right) \le n \exp\left( - \frac{x^2 (n-1)}{C_0 L} \right),
\end{align}
where $C_0 = 4C/[\delta(1-\delta)]$ and $C$ is the universal constant from Lemma \ref{Rec312}. By Weyl's inequality for singular values, we have
\[
| \sigma_1(\hat{R}) - \sigma_1(R) | \le \| \hat{R} - R \|,
\]
where $\hat{R}$ is the normalized residual matrix constructed with the estimated communities from Algorithm \ref{alg:spectral} using $(K_{s0},K_{r0}) = (K_s,K_r)$. Consequently, we get
\begin{align*}
| \hat{T}_n^* - T_n| \le \| \hat{R} - R \|.
\end{align*}

From this inequality, we obtain the one-sided bound:
\begin{equation}\label{eq:Tn_hat_upper}
\hat{T}_n^* = \sigma_1(\hat{R}) - 2 \le \sigma_1(R) - 2 + \| \hat{R} - R \| = T_n + \| \hat{R} - R \|.
\end{equation}

Now consider the event $\{ \hat{T}_n^* \ge t_n \}$. By Equation \eqref{eq:Tn_hat_upper}, we have
\[
\{ \hat{T}_n^* \ge t_n \} \subseteq \{ T_n + \| \hat{R} - R \| \ge t_n \}.
\]

If $T_n + \| \hat{R} - R \| \ge t_n$, then at least one of $T_n \ge t_n/3$ or $\| \hat{R} - R \| \ge t_n/3$ must hold. Indeed, if both $T_n < t_n/3$ and $\| \hat{R} - R \| < t_n/3$, then $T_n + \| \hat{R} - R \| < 2t_n/3 < t_n$ (since $t_n > 0$). Thus, we have
\[
\{ T_n + \| \hat{R} - R \| \ge t_n \} \subseteq \{ T_n \ge t_n/3 \} \cup \{ \| \hat{R} - R \| \ge t_n/3 \}.
\]

Applying the union bound, we obtain
\begin{equation}\label{eq:prob_split}
\mathbb{P}\bigl( \hat{T}_n^* \ge t_n \bigr) 
\le \mathbb{P}\bigl( T_n + \| \hat{R} - R \| \ge t_n \bigr) 
\le \mathbb{P}\bigl( T_n \ge t_n/3 \bigr) + \mathbb{P}\bigl( \| \hat{R} - R \| \ge t_n/3 \bigr).
\end{equation}

\paragraph{Bounding the first term in Equation \eqref{eq:prob_split}} Applying Equation \eqref{eq:tail-oracle} with $x = t_n/3$ gives
\[
\mathbb{P}\left( T_n\ge t_n/3 \right) \le n \exp\left( - \frac{t_n^2 (n-1)}{9 C_0 L} \right).
\]

We show the right-hand side tends to zero. Condition (C1) states $\alpha_n = o(t_n)$. Squaring yields
\[
\frac{K_{\max}^2 L \log n}{n} = o(t_n^2) \quad\Longleftrightarrow\quad \frac{t_n^2 n}{K_{\max}^2 L \log n} \to \infty .
\]

Since $K_{\max}\ge 1$, the above implies $t_n^2 n / (L\log n) \to \infty$. Hence, we have
\[
n \exp\left( - \frac{t_n^2 (n-1)}{9 C_0 L} \right) 
=n \exp\left( - \frac{t_n^2 n}{9 C_0 L} \cdot \frac{n-1}{n} \right) 
\le n \exp\left( - \frac{t_n^2 n}{10 C_0 L} \right) \to 0.
\]
for sufficiently large $n$ (because $(n-1)/n \to 1$). Thus, we get
\begin{align}\label{eq:oracle-bound}
\lim_{n\to\infty} \mathbb{P}\left( T_n\ge t_n/3 \right) = 0 .
\end{align}

\paragraph{Bounding the second term in Equation \eqref{eq:prob_split}} From Lemma \ref{lem:norm-error} and its proof, we have
\[
\| \hat{R} - R \| = O_P\!\left( K_{\max} \sqrt{\frac{L \log n}{n}} \right) = O_P(\alpha_n).
\]

More precisely, there exists a constant $C_{\mathrm{est}} > 0$ such that for all sufficiently large $n$, we have
\begin{align*}
\mathbb{P}\left( \| \hat{R} - R \| \ge C_{\mathrm{est}} \alpha_n \right) \le \frac{2}{n^2}.
\end{align*}

Condition (C1) gives $\alpha_n = o(t_n)$. Hence, for any fixed $\varepsilon > 0$, we have $C_{\mathrm{est}} \alpha_n \le \varepsilon t_n$ for all large $n$. Choose $\varepsilon = 1/3$. Then for $n$ large enough, we have
\[
\mathbb{P}\left( \| \hat{R} - R \| \ge t_n/3 \right) 
\le \mathbb{P}\left( \| \hat{R} - R \| \ge C_{\mathrm{est}} \alpha_n \right) 
\le \frac{2}{n^2} \to 0.
\]

Thus, we get
\begin{align}\label{eq:error-bound}
\lim_{n\to\infty} \mathbb{P}\left( \| \hat{R} - R \| \ge t_n/3 \right) = 0 .
\end{align}

\paragraph{Conclusion for $\tilde{\mathcal{A}}_n$} Inserting Equations \eqref{eq:oracle-bound} and \eqref{eq:error-bound} into Equation \eqref{eq:prob_split} yields $\mathbb{P}(\hat{T}_n^* \ge t_n) \to 0$, i.e. $\mathbb{P}(\tilde{\mathcal{A}}_n) \to 1$.

\vspace{1em}
\noindent\textbf{Part 2:  Behavior under underfitted models: $\mathbb{P}(\tilde{\mathcal{B}}_n) \to 1$.}

For any underfitted candidate $(k_s, k_r)$ (i.e., $k_s < K_s$ or $k_r < K_r$), we need a lower bound for $\hat{T}_n(k_s, k_r)$ that holds uniformly with high probability. From the proof of Theorem \ref{thm:power}, we have that for any underfitted candidate, there exist node sets $S$ and $T'$ with sizes satisfying $|S| \ge c_0 n/K_{\max}$ and $|T'| \ge c_0 n/K_{\max}^2$ such that
\[
\hat{T}_n(k_s, k_r) \ge \frac{c_0\eta}{\sqrt{2}}\,\frac{\sqrt{nL}}{K_{\max}^{3/2}}\sqrt{\frac{n}{n-1}} - \frac{2}{\sqrt{(n-1)L}}\|W_{S,T'}\| - 2,
\]
where $W_{S,T'}$ is the restriction of $W=\sum_{\ell=1}^L (A^{(\ell)}-\Omega^{(\ell)})$ to $S\times T'$. Define
\[
Z_n^{(k_s,k_r)} := \frac{2}{\sqrt{(n-1)L}}\|W_{S,T'}\| + 2.
\]
Since $\sqrt{n/(n-1)} \to 1$, there exists a constant $c_2 > 0$ (e.g., $c_2 = c_0\eta/(2\sqrt{2})$) such that for all sufficiently large $n$, we have
\begin{equation}\label{eq:underfit-lower}
\hat{T}_n(k_s, k_r) \ge c_2 \frac{\sqrt{nL}}{K_{\max}^{3/2}} - Z_n^{(k_s,k_r)}.
\end{equation}

Now we control the tail of $Z_n^{(k_s,k_r)}$. Recall from Equation \eqref{eqp13} in the proof of Theorem \ref{thm:power} that for any $M>0$, we have
\begin{equation}\label{eq:W-tail}
\mathbb{P}\Bigl( \|W_{S,T'}\| \ge \tfrac32\sqrt{L n} + M\sqrt{nL} \Bigr) \le n \exp\Bigl( - \frac{M^2 n}{C L} \Bigr),
\end{equation}
where $C>0$ is a universal constant. Thus, we have $Z_n^{(k_s,k_r)} = O_{\mathbb{P}}(1)$, and the bound in Equation \eqref{eq:W-tail} is uniform over all underfitted candidates because it depends only on $n$, $L$, and the entrywise bounds of $W_{S,T'}$, which are uniform.

Define $\gamma_n := c_2 \frac{\sqrt{nL}}{K_{\max}^{3/2}}=c_2\beta_n$. By condition (C2), we must have $t_n \le \gamma_n/2$ for all large $n$. We now bound $\mathbb{P}\bigl( \hat{T}_n(k_s, k_r) < t_n \bigr)$. By Equation \eqref{eq:underfit-lower}, we have
\begin{align*}
\mathbb{P}\bigl( \hat{T}_n(k_s, k_r) < t_n \bigr) 
&\le \mathbb{P}\bigl( c_2 \frac{\sqrt{nL}}{K_{\max}^{3/2}} - Z_n^{(k_s,k_r)} < t_n \bigr) \le \mathbb{P}\bigl( Z_n^{(k_s,k_r)} > \gamma_n - t_n \bigr) \le \mathbb{P}\bigl( Z_n^{(k_s,k_r)} > \gamma_n/2 \bigr) \quad \text{(since $t_n \le \gamma_n/2$)}.
\end{align*}

Now, $Z_n^{(k_s,k_r)} > \gamma_n/2$ implies $\frac{2}{\sqrt{(n-1)L}}\|W_{S,T'}\| > \gamma_n/2 - 2$. For large $n$, $\gamma_n/2 - 2 \ge \gamma_n/4$ because $\gamma_n \to \infty$. Hence, we have
\[
\mathbb{P}\bigl( Z_n^{(k_s,k_r)} > \gamma_n/2 \bigr) \le \mathbb{P}\Bigl( \|W_{S,T'}\| > \frac{\gamma_n}{4} \cdot \frac{\sqrt{(n-1)L}}{2} \Bigr) \le \mathbb{P}\Bigl( \|W_{S,T'}\| > \frac{c_2}{8} \frac{nL}{K_{\max}^{3/2}} \Bigr),
\]
where we used that $\frac{\gamma_n}{4} \cdot \frac{\sqrt{(n-1)L}}{2} \sim \frac{c_2}{8} \frac{nL}{K_{\max}^{3/2}}$. To apply Equation \eqref{eq:W-tail}, set $M$ such that $\tfrac32\sqrt{L n} + M\sqrt{nL} = \frac{c_2}{8} \frac{nL}{K_{\max}^{3/2}}$, i.e.
\[
M = \frac{c_2}{8} \frac{\sqrt{nL}}{K_{\max}^{3/2}} - \frac{3}{2}.
\]

For large $n$, $M \ge \frac{c_2}{16} \frac{\sqrt{nL}}{K_{\max}^{3/2}}$. Then by Equation \eqref{eq:W-tail}, we have
\begin{align*}
\mathbb{P}\Bigl( \|W_{S,T'}\| > \frac{c_2}{8} \frac{nL}{K_{\max}^{3/2}} \Bigr) 
&\le n \exp\Bigl( - \frac{M^2 n}{C L} \Bigr) \le n \exp\Bigl( - \frac{c_2^2}{256 C} \cdot \frac{nL}{K_{\max}^3} \cdot \frac{n}{L} \Bigr)
= n \exp\Bigl( - \frac{c_2^2}{256 C} \cdot \frac{n^2}{K_{\max}^3} \Bigr).
\end{align*}

First, recall that Assumption~\ref{assump:a3} gives \( \frac{K_{\max}^2 L \log n}{n} \to 0 \).  
Since \( L \ge 1 \), we have \( \frac{K_{\max}^2 \log n}{n} \to 0 \), i.e., \(K_{\max}^2 = o\!\left(\frac{n}{\log n}\right)\). Thus, there exists a constant \( a > 0 \) such that for all sufficiently large \( n \),  
\begin{align}\label{eq:Kmax-bound}
K_{\max}^2 \le a \frac{n}{\log n}. 
\end{align}

Consequently,  we get
\[
K_{\max}^3 = K_{\max}^2 \cdot K_{\max} \le \left(a \frac{n}{\log n}\right) \cdot \sqrt{a \frac{n}{\log n}} = a^{3/2} \frac{n^{3/2}}{(\log n)^{3/2}}. 
\]  
 
Hence, we have
\begin{align*}
\frac{n^2}{K_{\max}^3} \ge \frac{n^2}{a^{3/2} n^{3/2} / (\log n)^{3/2}} = \frac{1}{a^{3/2}} \, n^{1/2} (\log n)^{3/2},
\end{align*}
which implies  
\[
\exp\Bigl( - \frac{c_2^2}{256 C} \cdot \frac{n^2}{K_{\max}^3} \Bigr) \le \exp\Bigl( - \frac{c_2^2}{256 C a^{3/2}} \, n^{1/2} (\log n)^{3/2} \Bigr). 
\]  

Multiplying both sides by \( n \) obtains
\begin{align}\label{eq:exp-ineq}
n \exp\Bigl( - \frac{c_2^2}{256 C} \cdot \frac{n^2}{K_{\max}^3} \Bigr) \le n \exp\Bigl( - \frac{c_2^2}{256 C a^{3/2}} \, n^{1/2} (\log n)^{3/2} \Bigr). 
\end{align}

Next, we show that the right‑hand side of Equation \eqref{eq:exp-ineq} is \( o\bigl(1/K_{\max}^2\bigr) \). From Equation \eqref{eq:Kmax-bound}, we have \( \displaystyle \frac{1}{K_{\max}^2} \ge \frac{\log n}{a n} \).  
Therefore, it suffices to prove  
\begin{align}\label{eq:ratio-goal}
\frac{ n \exp\Bigl( - \frac{c_2^2}{256 C a^{3/2}} \, n^{1/2} (\log n)^{3/2} \Bigr) }{ (\log n)/(a n) } = a n^2 \, \frac{ \exp\Bigl( - \frac{c_2^2}{256 C a^{3/2}} \, n^{1/2} (\log n)^{3/2} \Bigr) }{ \log n } \; \xrightarrow[n\to\infty]{} \; 0. 
\end{align}

Taking logarithms gives 
\[
\log\Bigl( a n^2 / \log n \Bigr) \;-\; \frac{c_2^2}{256 C a^{3/2}} \, n^{1/2} (\log n)^{3/2}. 
\]  

As \( n \to \infty \), the second (negative) term dominates because \( n^{1/2} (\log n)^{3/2} \) grows faster than \( \log n \). Hence the whole expression tends to \( -\infty \), which implies the ratio in Equation \eqref{eq:ratio-goal} converges to zero. Thus, we get  
\begin{align}\label{eq:final-o-bound}
n \exp\Bigl( - \frac{c_2^2}{256 C a^{3/2}} \, n^{1/2} (\log n)^{3/2} \Bigr) = o\!\Bigl(\frac{1}{K_{\max}^2}\Bigr). 
\end{align}

Combining Equations \eqref{eq:exp-ineq} and \eqref{eq:final-o-bound} we obtain the desired chain
\[
n \exp\Bigl( - \frac{c_2^2}{256 C} \cdot \frac{n^2}{K_{\max}^3} \Bigr) \le n \exp\Bigl( - \frac{c_2^2}{256 C a^{3/2}} n^{1/2} (\log n)^{3/2} \Bigr) = o\!\left(\frac{1}{K_{\max}^2}\right).
\]

Therefore, for any underfitted candidate, we have
\[
\mathbb{P}\bigl( \hat{T}_n(k_s, k_r) < t_n \bigr) = o\!\left(\frac{1}{K_{\max}^2}\right).
\]

Now consider the number of underfitted candidates. The true pair $(K_s, K_r)$ satisfies $K_s \le K_{\max}$ and $K_r \le K_{\max}$. Since the candidate pairs are ordered lexicographically from $(1,1)$ to $(K_{\mathrm{cand}}, K_{\mathrm{cand}})$, every underfitted candidate (i.e., with $k_s < K_s$ or $k_r < K_r$) must satisfy $k_s \le K_{\max}$ and $k_r \le K_{\max}$. Therefore, the number of underfitted candidates, $m_*(n)-1$, is at most $K_{\max}^2$. Applying the union bound gives
\[
\mathbb{P}(\tilde{\mathcal{B}}_n^c) \le \sum_{m=1}^{m_*(n)-1} \mathbb{P}\bigl( \hat{T}_n(k_s^{(m)}, k_r^{(m)}) < t_n \bigr) \le K_{\max}^2 \cdot o\!\left(\frac{1}{K_{\max}^2}\right) = o(1).
\]

Thus, we get $\mathbb{P}(\tilde{\mathcal{B}}_n) \to 1$.

\vspace{1em}
\noindent\textbf{Part 3:  Completion of the proof.}

We have shown $\mathbb{P}(\tilde{\mathcal{A}}_n) \to 1$ and $\mathbb{P}(\tilde{\mathcal{B}}_n) \to 1$. Therefore, we have
\[
\mathbb{P}\bigl( (\hat{K}_s, \hat{K}_r) = (K_s, K_r) \bigr)
= \mathbb{P}(\tilde{\mathcal{A}}_n \cap \tilde{\mathcal{B}}_n) \ge 1 - \mathbb{P}(\tilde{\mathcal{A}}_n^c) - \mathbb{P}(\tilde{\mathcal{B}}_n^c) \to 1,
\]
which completes the proof of this theorem.
\end{proof}

\section{Proofs for the MLRDiGoF algorithm}\label{sec:proofs-RDiGoF}

\subsection{Preliminary lemmas}

We first establish two lemmas that characterize the asymptotic behavior of the test statistic \(\hat{T}_n\) under correctly specified and underfitted models. These lemmas are essential for proving Theorems \ref{thm:ratio} and \ref{thm:RDiGoF-consistency}. Throughout, we allow \(K_{\max} = \max(K_s, K_r)\) to grow with \(n\), subject to Assumption \ref{assump:a3} and condition (A2) of Theorem \ref{thm:power}.

\begin{lem}[Behavior of the test statistic at the true model]\label{lem:true-model-ratio}
For the true candidate \(m_*\) (i.e., \((k_s^{(m_*)}, k_r^{(m_*)}) = (K_s, K_r)\)), we have
\[
|\hat{T}_n(m_*)| = o_{\mathbb{P}}(1).
\]
\end{lem}
\begin{proof}
By Theorem \ref{thm:null}, for any \(\epsilon > 0\), \(\mathbb{P}(\hat{T}_n(m_*) < \epsilon) \to 1\) when \((K_{s0}, K_{r0}) = (K_s, K_r)\). Moreover, from Lemma \ref{ideal0} and Lemma \ref{lem:norm-error}, we have
\[
|\hat{T}_n(m_*) - T_n| \leq \|\hat{R} - R\| = o_{\mathbb{P}}(1),
\]
and \(T_n = o_{\mathbb{P}}(1)\). Since \(\hat{T}_n(m_*) = \sigma_1(\hat{R}) - 2\) and \(\sigma_1(\hat{R}) \geq 0\), we also have \(\hat{T}_n(m_*) \geq -2\). Therefore, \(|\hat{T}_n(m_*)| = o_{\mathbb{P}}(1)\). 
\end{proof}

\begin{lem}[Uniform bounds for underfitted models]\label{lem:underfit-bounds-ratio}
For any underfitted candidate \(m < m_*\) (i.e., \(k_s^{(m)} < K_s\) or \(k_r^{(m)} < K_r\)), there exist positive constants \(c_1, c_2\) (depending only on \(\delta, c_0, \eta\) from Assumptions \ref{assump:a1}, \ref{assump:a2}, and condition (A1)) such that with probability at least \(1 - O(K_{\max}^2 L n^{-3})\) as \(n \to \infty\), we have
\[
c_1 \frac{\sqrt{nL}}{K_{\max}^{3/2}} \leq \hat{T}_n(m) \leq \sqrt{2nL},
\]
where the constant \(c_1\) is uniform over all underfitted candidates.
\end{lem}
\begin{proof}
We prove the two-sided bound separately.

\vspace{0.5em}
\noindent\textbf{Part 1: Lower bound.}  
The lower bound follows directly from the proof of Theorem \ref{thm:power}. For an underfitted candidate $m$, Theorem \ref{thm:power} establishes that $\hat{T}_n(m) \xrightarrow{P} \infty$. More precisely, revisiting the proof of Theorem \ref{thm:power} (specifically, Equations \eqref{eqp12} and \eqref{eqp16}), there exist node sets $S$ and $T'$ with $|S| \ge c_0 n/K_{\max}$ and $|T'| \ge c_0 n/K_{\max}^2$ such that
\[
\| \hat{R}_{S,T'} \| \ge \frac{c_0\eta}{2\sqrt{2}} \cdot \frac{\sqrt{nL}}{K_{\max}^{3/2}} \cdot \sqrt{\frac{n}{n-1}} - O_{\mathbb{P}}(1).
\]

The probability that this inequality fails is at most $O(K_{\max}^2 L n^{-3})$, as shown via the concentration results in Lemmas \ref{lem:concentration} and \ref{lem:bound-Omega2} and the bound on $\|W_{S,T'}\|$ in Equation \eqref{eqp13}. Since $\sigma_1(\hat{R}) \ge \| \hat{R}_{S,T'} \|$, we have with probability at least $1 - O(K_{\max}^2 L n^{-3})$ that
\[
\hat{T}_n(m) = \sigma_1(\hat{R}) - 2 \ge \frac{c_0\eta}{4\sqrt{2}} \cdot \frac{\sqrt{nL}}{K_{\max}^{3/2}} =: c_L \frac{\sqrt{nL}}{K_{\max}^{3/2}}.
\]

\vspace{0.5em}
\noindent\textbf{Part 2: Upper bound.}  
We now prove a deterministic upper bound that holds for any candidate pair $(k_s,k_r)$ (underfitted or not). Let $\hat{R}$ be the normalized residual matrix defined in Equation \eqref{eq:residual_matrix}. For each estimated block $(s_0, r_0)$, let $n_{s_0} = |\{i: \hat{g}^s(i)=s_0\}|$, $n_{r_0} = |\{j: \hat{g}^r(j)=r_0\}|$, and $m = n_{s_0} n_{r_0}$. Denote $\hat{p}_{s_0r_0}^{(\ell)} = \hat{B}^{(\ell)}(s_0, r_0)$ and define
\[
\bar{D}_{s_0r_0} = \sum_{\ell=1}^L \hat{p}_{s_0r_0}^{(\ell)}(1-\hat{p}_{s_0r_0}^{(\ell)}).
\]

For any $(i,j)$ in this block (i.e., $i\in\hat{C}_{s_0}^s$, $j\in\hat{C}_{r_0}^r$), we have $\hat{\Omega}^{(\ell)}(i,j) = \hat{p}_{s_0r_0}^{(\ell)}$ and $D_{ij} = \bar{D}_{s_0r_0}$. By the Cauchy–Schwarz inequality, we have
\[
\left[ \sum_{\ell=1}^L \bigl( A^{(\ell)}(i,j) - \hat{p}_{s_0r_0}^{(\ell)} \bigr) \right]^2
\le L \sum_{\ell=1}^L \bigl( A^{(\ell)}(i,j) - \hat{p}_{s_0r_0}^{(\ell)} \bigr)^2.
\]

Summing over all pairs $(i,j)$ inside the block yields
\[
\sum_{i\in\hat{C}_{s_0}^s} \sum_{j\in\hat{C}_{r_0}^r}
\left[ \sum_{\ell=1}^L \bigl( A^{(\ell)}(i,j) - \hat{p}_{s_0r_0}^{(\ell)} \bigr) \right]^2
\le L \sum_{\ell=1}^L \sum_{i\in\hat{C}_{s_0}^s} \sum_{j\in\hat{C}_{r_0}^r}
\bigl( A^{(\ell)}(i,j) - \hat{p}_{s_0r_0}^{(\ell)} \bigr)^2.
\]

For each fixed layer $\ell$, the sum of squared deviations inside the block satisfies the algebraic identity
\[
\sum_{i\in\hat{C}_{s_0}^s} \sum_{j\in\hat{C}_{r_0}^r}
\bigl( A^{(\ell)}(i,j) - \hat{p}_{s_0r_0}^{(\ell)} \bigr)^2
= m\, \hat{p}_{s_0r_0}^{(\ell)}\bigl(1-\hat{p}_{s_0r_0}^{(\ell)}\bigr).
\]

This identity holds deterministically for any binary matrix and its block‑average. It is a direct consequence of the definition of $\hat{p}_{s_0r_0}^{(\ell)}$. Consequently, we have
\[
\sum_{i,j} \left[ \sum_{\ell=1}^L \bigl( A^{(\ell)}(i,j) - \hat{p}_{s_0r_0}^{(\ell)} \bigr) \right]^2
\le L \sum_{\ell=1}^L m\, \hat{p}_{s_0r_0}^{(\ell)}\bigl(1-\hat{p}_{s_0r_0}^{(\ell)}\bigr)
= m L \bar{D}_{s_0r_0}.
\]

Now consider the contribution of this block to $\|\hat{R}\|_F^2$. If $\bar{D}_{s_0r_0}>0$, then
\[
\sum_{i\in\hat{C}_{s_0}^s} \sum_{j\in\hat{C}_{r_0}^r} \hat{R}(i,j)^2
= \frac{1}{n-1} \frac{1}{\bar{D}_{s_0r_0}}
\sum_{i,j} \left[ \sum_{\ell=1}^L \bigl( A^{(\ell)}(i,j) - \hat{p}_{s_0r_0}^{(\ell)} \bigr) \right]^2
\le \frac{1}{n-1} \frac{m L \bar{D}_{s_0r_0}}{\bar{D}_{s_0r_0}}
= \frac{m L}{n-1}.
\]

If $\bar{D}_{s_0r_0}=0$, then $\hat{p}_{s_0r_0}^{(\ell)}\in\{0,1\}$ for every $\ell$, and the definition of $\hat{p}_{s_0r_0}^{(\ell)}$ implies that $A^{(\ell)}(i,j)=\hat{p}_{s_0r_0}^{(\ell)}$ for all $i,j$ in the block. Hence the numerator is zero and the contribution is zero, so the same inequality $\frac{m L}{n-1}$ (which is zero) holds trivially.

Summing over all blocks and using $\sum_{s_0,r_0} m = \sum_{s_0} n_{s_0} \sum_{r_0} n_{r_0} = n \cdot n = n^2$, we obtain
\[
\|\hat{R}\|_F^2 \le \frac{L}{n-1} \sum_{s_0,r_0} m = \frac{n^2 L}{n-1} \le 2nL \qquad (\text{for } n\ge 2).
\]

Therefore,  we have
\[
\sigma_1(\hat{R}) \le \|\hat{R}\|_F \le \sqrt{2nL},
\]
and consequently
\[
\hat{T}_n(m) = \sigma_1(\hat{R}) - 2 \le \sqrt{2nL}.
\]

This bound holds deterministically, so the required probability statement is satisfied trivially.

\vspace{0.5em}
\noindent\textbf{Combining the bounds.}  
By the union bound, both the lower bound and the upper bound hold simultaneously with probability at least $1 - O(K_{\max}^2 L n^{-3})$. This completes the proof of this lemma.
\end{proof}
\subsection{Proof of Theorem \ref{thm:ratio}}
\begin{proof}
Since \(K_s\) and \(K_r\) are fixed, \(K_{\max} = \max(K_s, K_r)\) is a constant independent of \(n\). This simplification is crucial for the following analysis.

\vspace{0.5em}
\noindent\textbf{Part 1: Divergence at the true model.}

Let \(m_*\) be the index of the true pair \((K_s, K_r)\). For any fixed \(M_0 > 0\), we need to show that \(\lim_{n \to \infty} \mathbb{P}(r_{m_*} > M_0) = 1\).

By definition,
\[
r_{m_*} = \left| \frac{\hat{T}_n(m_*-1)}{\hat{T}_n(m_*)} \right|.
\]

Since \(m_*-1\) corresponds to an underfitted model (either \(k_s^{(m_*-1)} < K_s\) or \(k_r^{(m_*-1)} < K_r\) or both), by Theorem \ref{thm:power} with fixed \(K_{\max}\), we have \(\hat{T}_n(m_*-1) \xrightarrow{P} \infty\). More precisely, from Lemma \ref{lem:underfit-bounds-ratio} with fixed \(K_{\max}\), there exists a constant \(c_1 > 0\) such that
\begin{equation}\label{eq:underfit-lower-bound}
\lim_{n \to \infty} \mathbb{P}\left( \hat{T}_n(m_*-1) \ge c_1 \sqrt{nL} \right) = 1.
\end{equation}

For the true model \(m_*\), by Theorem \ref{thm:null}, we have \(\hat{T}_n(m_*) = o_{\mathbb{P}}(1)\). That is, for any \(\epsilon > 0\),
\begin{equation}\label{eq:true-model-small}
\lim_{n \to \infty} \mathbb{P}\left( |\hat{T}_n(m_*)| < \epsilon \right) = 1.
\end{equation}

Now, fix an arbitrary \(M_0 > 0\). We will construct an event on which \(r_{m_*} > M_0\) and show that the probability of this event tends to 1.

Consider the event
\[
\mathcal{C}_n = \left\{ \hat{T}_n(m_*-1) \ge c_1 \sqrt{nL} \right\} \cap \left\{ |\hat{T}_n(m_*)| < \frac{c_1 \sqrt{nL}}{2M_0} \right\}.
\]

On \(\mathcal{C}_n\), we have
\[
r_{m_*} = \left| \frac{\hat{T}_n(m_*-1)}{\hat{T}_n(m_*)} \right| > \frac{c_1 \sqrt{nL}}{c_1 \sqrt{nL}/(2M_0)} = 2M_0 > M_0.
\]

Thus, we have \(\mathbb{P}(r_{m_*} > M_0) \ge \mathbb{P}(\mathcal{C}_n)\). We now prove that \(\mathbb{P}(\mathcal{C}_n) \to 1\) as \(n \to \infty\). First, by Equation \eqref{eq:underfit-lower-bound}, we have
\[
\lim_{n \to \infty} \mathbb{P}\left( \hat{T}_n(m_*-1) \ge c_1 \sqrt{nL} \right) = 1.
\]

Second, we analyze \(\mathbb{P}\left( |\hat{T}_n(m_*)| < \frac{c_1 \sqrt{nL}}{2M_0} \right)\). Since \(\hat{T}_n(m_*) = o_{\mathbb{P}}(1)\), for any fixed \(\epsilon > 0\), there exists \(N_\epsilon \) such that for all \(n \ge N_\eta\),
\[
\mathbb{P}\left( |\hat{T}_n(m_*)| < \epsilon  \right) > 1 - \epsilon .
\]

Now, note that \(\frac{c_1 \sqrt{nL}}{2M_0} \to \infty\) as \(n \to \infty\) (since \(c_1\), \(M_0\) are fixed positive constants and \(\sqrt{nL} \to \infty\)). Therefore, for any fixed \(\epsilon > 0\), there exists \(N_1\) such that for all \(n \ge N_1\), \(\frac{c_1 \sqrt{nL}}{2M_0} > \epsilon \). Then for \(n \ge \max(N_\epsilon , N_1)\),
\[
\mathbb{P}\left( |\hat{T}_n(m_*)| < \frac{c_1 \sqrt{nL}}{2M_0} \right) \ge \mathbb{P}\left( |\hat{T}_n(m_*)| < \epsilon  \right) > 1 - \epsilon .
\]

Since \(\epsilon \) can be chosen arbitrarily small, we conclude that
\[
\lim_{n \to \infty} \mathbb{P}\left( |\hat{T}_n(m_*)| < \frac{c_1 \sqrt{nL}}{2M_0} \right) = 1.
\]

Now, by the union bound,
\[
\mathbb{P}(\mathcal{C}_n^c) \le \mathbb{P}\left( \hat{T}_n(m_*-1) < c_1 \sqrt{nL} \right) + \mathbb{P}\left( |\hat{T}_n(m_*)| \ge \frac{c_1 \sqrt{nL}}{2M_0} \right).
\]

Both terms on the right-hand side converge to 0. Hence, \(\mathbb{P}(\mathcal{C}_n^c) \to 0\), which implies \(\mathbb{P}(\mathcal{C}_n) \to 1\). Therefore, \(\mathbb{P}(r_{m_*} > M_0) \to 1\), proving part 1.

\vspace{0.5em}
\noindent\textbf{Part 2: Uniform upper bound for underfitted models.}

Take any \(m < m_*\) (so both \(m-1\) and \(m\) correspond to underfitted models). We need to show that there exists a constant \(C > 0\) such that \(\lim_{n \to \infty} \mathbb{P}(r_m > C) = 0\).

From Lemma \ref{lem:underfit-bounds-ratio}, there exist constants \(c_1, c_2 > 0\) such that with probability at least \(1 - O(K_{\max}^2 L n^{-3}) = 1 - O(Ln^{-3})\) (since \(K_{\max}\) is fixed),
\begin{equation}\label{eq:underfit-bounds}
c_1 \sqrt{nL} \le \hat{T}_n(m) \le c_2 \sqrt{nL}.
\end{equation}

The same bound holds for \(\hat{T}_n(m-1)\) with the same probability guarantee. Define the event
\[
\mathcal{A}_n^{(m)} = \left\{ c_1 \sqrt{nL} \le \hat{T}_n(m) \le c_2 \sqrt{nL} \right\},
\]
and similarly \(\mathcal{A}_n^{(m-1)}\) for \(\hat{T}_n(m-1)\). Then \(\mathbb{P}\left( (\mathcal{A}_n^{(m)})^c \right) = O(Ln^{-3})\) and \(\mathbb{P}\left( (\mathcal{A}_n^{(m-1)})^c \right) = O(Ln^{-3})\).

On the event \(\mathcal{A}_n^{(m)} \cap \mathcal{A}_n^{(m-1)}\), we have
\[
r_m = \left| \frac{\hat{T}_n(m-1)}{\hat{T}_n(m)} \right| \le \frac{c_2 \sqrt{nL}}{c_1 \sqrt{nL}} = \frac{c_2}{c_1}.
\]

Define \(C = \frac{c_2}{c_1} + \epsilon\) for some arbitrary \(\epsilon > 0\) (for concreteness, take \(\epsilon = 1\), so \(C = \frac{c_2}{c_1} + 1\)). Then on \(\mathcal{A}_n^{(m)} \cap \mathcal{A}_n^{(m-1)}\), we have \(r_m \le \frac{c_2}{c_1} < C\).

Now,
\[
\mathbb{P}(r_m > C) \le \mathbb{P}\left( (\mathcal{A}_n^{(m)} \cap \mathcal{A}_n^{(m-1)})^c \right) \le \mathbb{P}\left( (\mathcal{A}_n^{(m)})^c \right) + \mathbb{P}\left( (\mathcal{A}_n^{(m-1)})^c \right) = O(Ln^{-3}) + O(Ln^{-3}) = O(Ln^{-3}).
\]

Hence, by Assumption \ref{assump:a3}, we have
\[
\lim_{n \to \infty} \mathbb{P}(r_m > C) = 0.
\]

This holds for every \(m < m_*\). The constant \(C\) does not depend on \(n\) or on the specific \(m\) (since \(c_1\) and \(c_2\) are universal for all underfitted models). This completes the proof of this theorem.
\end{proof}
\subsection{Proof of Theorem \ref{thm:RDiGoF-consistency}}
\begin{proof} Since \(K_s\) and \(K_r\) are fixed, we have that \(K_{\max} = \max(K_s, K_r)\) is a constant independent of \(n\). Recall that \(m_*\) denotes the index of \((K_s, K_r)\) in \(\mathcal{P}\). Because \(K_s\) and \(K_r\) are fixed, \(m_*\) is fixed (does not depend on \(n\)). We consider two cases separately.

\vspace{0.5em}
\noindent\textbf{Case 1: \((K_s, K_r) = (1,1)\).}  
In this case, \(m_* = 1\). The algorithm first computes \(\hat{T}_n(1)\) for candidate \((1,1)\). By Theorem \ref{thm:null} (with \((K_{s0}, K_{r0}) = (1,1) = (K_s, K_r)\)), for any \(\epsilon > 0\),
\[
\lim_{n \to \infty} \mathbb{P}\bigl( \hat{T}_n(1) < \epsilon \bigr) = 1.
\]

The threshold used in Algorithm \ref{alg:RDIGoF} for the first candidate is \(t_n\), which satisfies \(t_n \to 0\). Therefore, taking \(\epsilon = t_n\) (which is positive and tends to 0), we have
\[
\lim_{n \to \infty} \mathbb{P}\bigl( \hat{T}_n(1) < t_n \bigr) = 1.
\]

Thus, with probability tending to 1, Algorithm \ref{alg:RDIGoF} returns \((1,1)\) at the first step. Hence, we have
\[
\lim_{n \to \infty} \mathbb{P}\bigl( (\hat{K}_s, \hat{K}_r) = (1,1) \bigr) = 1.
\]

\vspace{0.5em}
\noindent\textbf{Case 2: \((K_s, K_r) \neq (1,1)\).}  
Then \(m_* > 1\). We define the following events:

\begin{itemize}
    \item \(\breve{\mathcal{E}}_n := \{\hat{T}_n(1) \ge t_n\}\). This is the event that the algorithm does not stop at the first candidate.
    \item \(\breve{\mathcal{A}}_n := \{r_{m_*} > \tau_n\}\), where \(r_{m_*} = \left| \frac{\hat{T}_n(m_*-1)}{\hat{T}_n(m_*)} \right|\).
    \item For each \(m = 2, \dots, m_*-1\), define \(\breve{\mathcal{B}}_n^{(m)} := \{r_m \le \tau_n\}\), where \(r_m = \left| \frac{\hat{T}_n(m-1)}{\hat{T}_n(m)} \right|\). 
    \item \(\breve{\mathcal{B}}_n := \bigcap_{m=2}^{m_*-1} \breve{\mathcal{B}}_n^{(m)}\).
\end{itemize}

If all three events \(\breve{\mathcal{E}}_n\), \(\breve{\mathcal{A}}_n\), and \(\breve{\mathcal{B}}_n\) occur, then:
\begin{itemize}
    \item Since \(\breve{\mathcal{E}}_n\) occurs, \(\hat{T}_n(1) \ge t_n\), so the algorithm proceeds to the loop.
    \item For each \(m = 2, \dots, m_*-1\), since \(\breve{\mathcal{B}}_n^{(m)}\) occurs, we have \(r_m \le \tau_n\), so the condition \(r_m > \tau_n\) is not satisfied, and the algorithm does not stop at these underfitted candidates.
    \item At \(m = m_*\), since \(\breve{\mathcal{A}}_n\) occurs, we have \(r_{m_*} > \tau_n\), so the algorithm stops and returns \(\mathcal{P}(m_*) = (K_s, K_r)\).
\end{itemize}

Therefore, we have
\[
\{(\hat{K}_s, \hat{K}_r) = (K_s, K_r)\} \supseteq \breve{\mathcal{E}}_n \cap \breve{\mathcal{A}}_n \cap \breve{\mathcal{B}}_n,
\]
which gives
\[
\mathbb{P}\bigl( (\hat{K}_s, \hat{K}_r) = (K_s, K_r) \bigr) \ge \mathbb{P}(\breve{\mathcal{E}}_n \cap \breve{\mathcal{A}}_n \cap \breve{\mathcal{B}}_n) \ge 1 - \mathbb{P}(\breve{\mathcal{E}}_n^c) - \mathbb{P}(\breve{\mathcal{A}}_n^c) - \mathbb{P}(\breve{\mathcal{B}}_n^c).
\]

We will show that \(\mathbb{P}(\breve{\mathcal{E}}_n^c) \to 0\), \(\mathbb{P}(\breve{\mathcal{A}}_n^c) \to 0\), and \(\mathbb{P}(\breve{\mathcal{B}}_n^c) \to 0\) as \(n \to \infty\).

\vspace{0.5em}
\noindent\textbf{Step 1: Behavior of \(\breve{\mathcal{E}}_n\).}  
Since \((1,1)\) is an underfitted model (because \((K_s, K_r) \neq (1,1)\)), Theorem \ref{thm:power} gives \(\hat{T}_n(1) \xrightarrow{P} \infty\). For any \(\epsilon > 0\), there exists \(N\) such that for all \(n \ge N\), \(t_n < \epsilon\). Then, we have
\[
\mathbb{P}\bigl( \hat{T}_n(1) <t_n \bigr) \le \mathbb{P}\bigl( \hat{T}_n(1) < \epsilon \bigr).
\]

Since \(\hat{T}_n(1) \xrightarrow{P} \infty\), we have \(\lim_{n \to \infty} \mathbb{P}\bigl( \hat{T}_n(1) < \epsilon \bigr) = 0\). Hence, we get
\[
\lim_{n \to \infty} \mathbb{P}(\breve{\mathcal{E}}_n^c) = \lim_{n \to \infty} \mathbb{P}\bigl( \hat{T}_n(1) < t_n \bigr) = 0.
\]

\vspace{0.5em}
\noindent\textbf{Step 2: Behavior of \(\breve{\mathcal{A}}_n\).}  
We need to show \(\mathbb{P}(\breve{\mathcal{A}}_n^c) = \mathbb{P}(r_{m_*} \le \tau_n) \to 0\). Recall from Lemma \ref{lem:underfit-bounds-ratio} (with fixed \(K_{\max}\)) that there exist positive constants \(c_L\) and \(c_U\) (depending only on \(\delta, c_0, \eta\) from Assumptions \ref{assump:a1}, \ref{assump:a2}, and condition (A1)) such that for any underfitted model (in particular for \(m_*-1\)),
\begin{equation}\label{eq:underfit-bound-mstar-1}
\lim_{n \to \infty} \mathbb{P}\left( \hat{T}_n(m_*-1) \ge c_L \sqrt{nL} \right) = 1.
\end{equation}

More precisely, Lemma \ref{lem:underfit-bounds-ratio} gives \(\hat{T}_n(m_*-1) \ge c_1 \frac{\sqrt{nL}}{K_{\max}^{3/2}}\) with high probability. Since \(K_{\max}\) is fixed, we set \(c_L = c_1 / K_{\max}^{3/2}\).

For the true model \(m_*\), by Theorem \ref{thm:null} and Lemma \ref{lem:norm-error}, we have \(\hat{T}_n(m_*) = o_{\mathbb{P}}(1)\). Specifically, from Lemma \ref{ideal0} (Equation \eqref{eqp1}) and Lemma \ref{lem:norm-error}, there exists a constant \(M_0 > 0\) such that for all sufficiently large \(n\),
\begin{equation}\label{eq:true-model-bound}
\mathbb{P}\left( |\hat{T}_n(m_*)| \le M_0 \sqrt{\frac{L \log n}{n}} \right) \ge 1 - \upsilon_n,
\end{equation}
where \(\upsilon_n = O(n^{-1}) + O(K_{\max}^2 L n^{-3}) = O(n^{-1})\) (since \(K_{\max}\) and \(L\) satisfy Assumption \ref{assump:a3} and \(K_{\max}\) is fixed, \(L\) can grow but \(L = o(n/\log n)\)). Now define the event
\[
\breve{\mathcal{F}}_n := \left\{ \hat{T}_n(m_*-1) \ge c_L \sqrt{nL} \right\} \cap \left\{ |\hat{T}_n(m_*)| \le M_0 \sqrt{\frac{L \log n}{n}} \right\}.
\]

On \(\breve{\mathcal{F}}_n\), we have
\[
r_{m_*} = \left| \frac{\hat{T}_n(m_*-1)}{\hat{T}_n(m_*)} \right| \ge \frac{c_L \sqrt{nL}}{M_0 \sqrt{L \log n / n}} = \frac{c_L}{M_0} \cdot \frac{n}{\sqrt{\log n}}.
\]

Let \(b_n := \frac{c_L}{M_0} \cdot \frac{n}{\sqrt{\log n}}\). Note that \(b_n \to \infty\) as \(n \to \infty\). We now compare \(b_n\) with \(\tau_n\). Condition (D2) implies \(\tau_n = o\left( \sqrt{n/\log n} \right)\). Now, we have
\[
\frac{b_n}{\sqrt{n/\log n}} = \frac{c_L}{M_0} \cdot \frac{n}{\sqrt{\log n}} \cdot \sqrt{\frac{\log n}{n}} = \frac{c_L}{M_0} \sqrt{n} \to \infty.
\]

Hence, \(b_n\) grows faster than \(\sqrt{n/\log n}\), and consequently faster than \(\tau_n\) (since \(\tau_n = o(\sqrt{n/\log n})\)). Therefore, for sufficiently large \(n\), we have \(\tau_n < b_n\). Thus, on \(\mathcal{F}_n\), we have
\[
r_{m_*} \ge b_n > \tau_n,
\]
implying that \(\breve{\mathcal{F}}_n \subseteq \breve{\mathcal{A}}_n\). Hence, we have
\[
\mathbb{P}(\breve{\mathcal{A}}_n^c) \le \mathbb{P}(\breve{\mathcal{F}}_n^c).
\]

By the union bound, we get
\[
\mathbb{P}(\breve{\mathcal{F}}_n^c) \le \mathbb{P}\left( \hat{T}_n(m_*-1) < c_L \sqrt{nL} \right) + \mathbb{P}\left( |\hat{T}_n(m_*)| > M_0 \sqrt{\frac{L \log n}{n}} \right).
\]

From Equation \eqref{eq:underfit-bound-mstar-1}, the first term tends to 0. From Equation \eqref{eq:true-model-bound}, the second term is at most \(\upsilon_n = O(n^{-1})\). Therefore, \(\mathbb{P}(\breve{\mathcal{F}}_n^c) \to 0\), and consequently \(\mathbb{P}(\breve{\mathcal{A}}_n^c) \to 0\).

\vspace{0.5em}
\noindent\textbf{Step 3: Behavior of \(\breve{\mathcal{B}}_n\).}  
We need to show \(\mathbb{P}(\breve{\mathcal{B}}_n^c) \to 0\). Note that
\[
\breve{\mathcal{B}}_n^c = \bigcup_{m=2}^{m_*-1} \left( \breve{\mathcal{B}}_n^{(m)} \right)^c = \bigcup_{m=2}^{m_*-1} \{ r_m > \tau_n \}.
\]

By the union bound, we have
\[
\mathbb{P}(\breve{\mathcal{B}}_n^c) \le \sum_{m=2}^{m_*-1} \mathbb{P}\left( r_m > \tau_n \right).
\]

For each \(m = 2, \dots, m_*-1\), both \(m-1\) and \(m\) correspond to underfitted models (since \(m < m_*\)). By Theorem \ref{thm:ratio}, there exists a constant \(C > 0\) (depending only on \(\delta, c_0, \eta\)) such that for each such \(m\),
\[
\lim_{n \to \infty} \mathbb{P}\left( r_m > C \right) = 0.
\]

Condition (D1) ensures that there exists a constant \(C_0 > C\) and \(n_0\) such that for all \(n \ge n_0\), \(\tau_n > C_0\). Since \(C_0 > C\), we have for \(n \ge n_0\) that
\[
\{ r_m > \tau_n \} \subseteq \{ r_m > C_0 \} \subseteq \{ r_m > C \}.
\]

Therefore, for \(n \ge n_0\), we have
\[
\mathbb{P}\left( r_m > \tau_n \right) \le \mathbb{P}\left( r_m > C \right).
\]

Hence, get
\[
\lim_{n \to \infty} \mathbb{P}\left( r_m > \tau_n \right) \le \lim_{n \to \infty} \mathbb{P}\left( r_m > C \right) = 0.
\]

Since there are at most \(m_*-2\) terms in the sum (a fixed number, because \(m_*\) is fixed), we obtain
\[
\lim_{n \to \infty} \mathbb{P}(\breve{\mathcal{B}}_n^c) \le \sum_{m=2}^{m_*-1} \lim_{n \to \infty} \mathbb{P}\left( r_m > \tau_n \right) = 0.
\]

\vspace{0.5em}
\noindent\textbf{Step 4: Completion of Case 2.}  
We have shown $\lim_{n \to \infty} \mathbb{P}(\breve{\mathcal{E}}_n^c) = 0, \lim_{n \to \infty} \mathbb{P}(\breve{\mathcal{A}}_n^c) = 0, \lim_{n \to \infty} \mathbb{P}(\breve{\mathcal{B}}_n^c) = 0$. Therefore, we have
\[
\lim_{n \to \infty} \mathbb{P}\bigl( (\hat{K}_s, \hat{K}_r) = (K_s, K_r) \bigr) \ge \lim_{n \to \infty} \left[ 1 - \mathbb{P}(\breve{\mathcal{E}}_n^c) - \mathbb{P}(\breve{\mathcal{A}}_n^c) - \mathbb{P}(\breve{\mathcal{B}}_n^c) \right] = 1.
\]

Since the probability cannot exceed 1, we conclude
\[
\lim_{n \to \infty} \mathbb{P}\bigl( (\hat{K}_s, \hat{K}_r) = (K_s, K_r) \bigr) = 1.
\]

\vspace{0.5em}
\noindent\textbf{Conclusion:}  
Combining Case 1 and Case 2, we have shown that under the stated conditions,
\[
\lim_{n \to \infty} \mathbb{P}\bigl( (\hat{K}_s, \hat{K}_r) = (K_s, K_r) \bigr) = 1,
\]
which completes the proof.
\end{proof}
\bibliographystyle{elsarticle-harv}
\bibliography{reference}
\end{document}